\documentclass[11pt]{article}
\usepackage[margin=1in]{geometry}
\usepackage{amsmath,amssymb,amsthm,mathtools}
\usepackage[T1]{fontenc}
\usepackage{booktabs,longtable,array}
\usepackage{tikz}
\usepackage{xcolor}
\usetikzlibrary{arrows.meta,positioning}
\usepackage[colorlinks=true,linkcolor=blue,citecolor=blue,urlcolor=blue]{hyperref}
\hypersetup{pdftitle={The M\'ezard--Parisi formula},pdfauthor={Dmitry Panchenko}}
\newcommand{\E}{\mathbb E}
\newcommand{\Prob}{\mathbb P}
\newcommand{\N}{\mathbb N}

\newcommand{\Cov}{\operatorname{Cov}}
\newcommand{\supp}{\operatorname{supp}}
\newcommand{\Av}{\operatorname{Av}}

\newcommand{\diam}{\operatorname{diam}}
\newcommand{\conv}{\operatorname{conv}}

\newcommand{\cF}{\mathcal F}
\newcommand{\cG}{\mathcal G}
\newcommand{\cH}{\mathcal H}
\newcommand{\cT}{\mathcal T}

\newcommand{\cB}{{\mathcal{B}}}
\newcommand{\cP}{\mathcal P}
\newcommand{\cU}{\mathcal U}

\newcommand{\cA}{\mathcal A}
\newcommand{\cZ}{\mathcal Z}
\newcommand{\Sp}{\mathsf P}
\newcommand{\one}{\mathbf 1}

\makeatletter
\newsavebox\myboxA
\newsavebox\myboxB
\newlength\mylenA
\newcommand*\mybar[2][0.75]{%
    \sbox{\myboxA}{$\m@th#2$}%
    \setbox\myboxB\null
    \ht\myboxB=\ht\myboxA%
    \dp\myboxB=\dp\myboxA%
    \wd\myboxB=#1\wd\myboxA
    \sbox\myboxB{$\m@th\overline{\copy\myboxB}$}
    \setlength\mylenA{\the\wd\myboxA}
    \addtolength\mylenA{-\the\wd\myboxB}%
    \ifdim\wd\myboxB<\wd\myboxA%
       \rlap{\hskip 0.5\mylenA\usebox\myboxB}{\usebox\myboxA}%
    \else
        \hskip -0.5\mylenA\rlap{\usebox\myboxA}{\hskip 0.5\mylenA\usebox\myboxB}%
    \fi}
\makeatother

\newtheorem{theorem}{Theorem}[section]
\newtheorem{lemma}[theorem]{Lemma}
\newtheorem{proposition}[theorem]{Proposition}

\newtheorem{definition}[theorem]{Definition}
\theoremstyle{remark}
\newtheorem{remark}[theorem]{Remark}
\numberwithin{equation}{section}
\title{The M\'ezard--Parisi formula}
\author{Dmitry Panchenko\thanks{\tiny Department of Mathematics, University of Toronto, panchenk@math.toronto.edu. Partially supported by NSERC grant.}
}
\date{}
\begin{document}
\maketitle

\begin{abstract}
This paper proves the M\'ezard--Parisi formula for the free energy of random $K$-sat and diluted $p$-spin models with even $p$, completing an approach developed in earlier work of the author and collaborators. The upper bound was established by Franz and Leone. Here the matching lower bound is proved for a general class of diluted models under symmetry and integrability assumptions. The main new contribution is a proof of an approximate form of the M\'ezard--Parisi reproducibility hypothesis without assuming that the overlap distribution has finite support.
\end{abstract}

\begingroup
\setlength{\parskip}{0pt}
\tableofcontents
\endgroup
\clearpage

\section{Introduction}\label{sec:intro}

The Parisi solution \cite{Parisi79, Parisi80} of the
Sherrington-Kirkpatrick model \cite{SK} introduced what
is now called the Parisi ansatz, which became central
to the subsequent development of spin glass theory;
see \cite{MPV}. Twenty years later, this ansatz was
the starting point of the M\'ezard--Parisi approach to
diluted models \cite{MP1}. The diluted setting,
however, required an additional hypothesis, usually
called {reproducibility}: multioverlaps should
be determined by the pairwise overlaps of the replicas
involved.

The need for this additional information comes from
the different nature of the cavity fields. In the
SK model, these fields are Gaussian, and their
covariance is expressed in terms of pair overlaps.
In diluted models, their distribution involves the
full spin distribution, or equivalently all
multioverlaps, and the ultrametric organization of the
overlaps describes only part of the structure needed for the M\'ezard--Parisi formula.

The Parisi formula for the SK model was proved
rigorously by Talagrand \cite{TPF}, following the
breakthrough work of Guerra \cite{Guerra}. A different
proof, which also extended the formula to general
mixed $p$-spin models, was given in \cite{PPF}. It
used the fact, proved in \cite{Ultra}, that the
Ghirlanda--Guerra identities \cite{GG} imply
ultrametricity. These identities can also be enforced
in diluted models by a small perturbation, so the
ultrametric part of the Parisi picture is available
in this setting as well.

The analogue of Guerra's upper bound for diluted
models was obtained by Franz and Leone \cite{FL};
see also \cite{PT}, but the matching M\'ezard--Parisi
lower bound remained elusive. One approach to this
problem, based on the Ghirlanda--Guerra identities
and cavity equations, was developed in the series
of papers \cite{Spins, HEPS, AP, 1RSB, rRSB}.

In \cite{Spins}, a Poissonization technique was
introduced to justify cavity computations of the
Aizenman--Sims--Starr type \cite{ASS} and derive
cavity equations for spin distributions. In
\cite{HEPS}, the Ghirlanda--Guerra identities were
shown to imply {hierarchical exchangeability}
(the distribution of the spin arrays within clusters
is invariant under permutations preserving the
ultrametric tree structure), and in \cite{AP} an analogue
of the Aldous-Hoover representation was proved for
arrays with this symmetry. The resulting
representation was remarkably close to the
M\'ezard--Parisi ansatz, but contained additional
randomness whose removal was the remaining obstacle to proving reproducibility.

An approach to proving it was developed first for certain $1$-RSB diluted $p$-spin models in \cite{1RSB}, and then for the
finite-RSB case in \cite{rRSB}. In the latter paper,
an additional small perturbation supplied
constraints through the cavity equations that
enforced reproducibility. The proof involved a
technical induction on tree shapes, and
did not provide a way to control approximation
errors as the number of overlap levels increased.
Thus, the passage from finite-RSB to a general
overlap distribution remained open.

The purpose of this paper is to remove the
finite-RSB assumption. We prove an approximate
M\'ezard--Parisi structure, sufficient to obtain
the M\'ezard--Parisi lower bound, without assuming
that the overlap distribution has finite support.
Combined with the Franz--Leone upper bound, this
proves the M\'ezard--Parisi formula for the diluted
even $p$-spin and random $K$-sat models considered
below. The lower bound itself applies to a broader
class of diluted models.

The strategy follows \cite{rRSB} closely, with
several new ingredients. The main ingredient is a different
perturbation: a diluted version of the Random Energy
Model (REM) of Derrida \cite{DerridaREM}. Its random
energy vectors on the hypercube have full support, which allows us to prescribe arbitrary clauses to test the cavity equations. This freedom allows us to simplify the induction in \cite{rRSB} and organize it by the number of replicas instead of tree shapes. In this sense, the proof here is structurally simpler than the finite-RSB case in \cite{rRSB}, with much of the additional length due to a quantitative form of the argument. 

The exact identities used in the finite-RSB setting are now replaced by approximate estimates whose errors, for each fixed number of
replicas, remain controlled uniformly in the number
of overlap levels. We can therefore work with
increasingly fine finite partitions of a general
overlap distribution. A compactness argument
completes the passage from the tested cavity
estimates to the required reproducibility.
The proof strategy is described in more detail
in Section \ref{sec:outline}.

\paragraph{Some related work.}
The M\'ezard--Parisi ansatz was developed further at zero
temperature by M\'ezard and Parisi \cite{MPzero}, while Franz, Leone,
Ricci-Tersenghi and Zecchina \cite{FLRZ} studied finite-connectivity
$p$-spin and optimization models using a one-step replica-symmetry-breaking ansatz. Its application to random satisfiability was developed by
M\'ezard, Parisi and Zecchina \cite{MPZ} and M\'ezard and Zecchina
\cite{MZcavity}. Mertens, M\'ezard and Zecchina \cite{MMZ} used
these cavity equations to derive predictions for the satisfiability
threshold, including its large-$k$ asymptotics.

On the mathematical side, De Sanctis \cite{DeSanctis}
established an extended variational principle for the Viana-Bray model
in terms of random multi-overlap structures. Coja-Oghlan and Perkins
\cite{CojaPerkins} proved a decomposition into Bethe states for a broad
family of random regular factor-graph models and, under a positivity
condition, obtained variational formulas for their free energy.
Further work on the full replica-symmetry-breaking description of
diluted models includes Parisi's study of marginal stability
\cite{ParisiBethe} and Concetti's martingale formulation of the
hierarchical free-energy functional on random regular graphs
\cite{Concetti}.

For random $K$-sat, Coja-Oghlan and Panagiotou \cite{Coja2} proved the predicted large-\(k\) asymptotics of the satisfiability threshold, up to an additive error tending to zero as \(k\to\infty\). Ding, Sly and Sun \cite{DSS} established the exact threshold for all sufficiently large \(k\), confirming the one-step replica-symmetry-breaking prediction.

\paragraph{Acknowledgement on the use of AI.}
The proof in this paper was found by ChatGPT Astra. I asked it to look for a new perturbation in the argument of \cite{rRSB} and try to simplify the
induction. It returned an initial version of the argument
developed below. I have read, checked, and rewritten the proof. ChatGPT Astra and Claude Fable assisted me with the rewriting. 

A few days after the proof presented here was found, OpenAI released another proof of the M\'ezard--Parisi formula in \cite{OAI}. That proof appears to follow a different approach.

\paragraph{Acknowledgements.} I would like to thank Giorgio Parisi for his help in preparing this paper.

\section{Models and main result}\label{sec:models}

\subsection{The diluted models}
Fix an integer $K\ge2$. Let $\Sigma_N=\{-1,+1\}^N$, and let $\theta:\{-1,+1\}^K\to\mathbb R$ be a random clause whose law is invariant under permutations of its inputs. Independent copies are denoted by $\theta_a$. For a connectivity parameter $\lambda>0$, define
\begin{equation}\label{eq:Hmodel}
 H_N^{\rm model}(\sigma)
 =\sum_{a=1}^{\pi_0(\lambda N)}
       \theta_a(\sigma_{i_{1,a}},\ldots,\sigma_{i_{K,a}}),
 \qquad
 F_N^{\rm model}=\frac1N\E\log\sum_{\sigma\in\Sigma_N}e^{H_N^{\rm model}(\sigma)}.
\end{equation}
Here $\pi_0(\lambda N)$ is Poisson with mean $\lambda N$, and all indices $i_{j,a}$ are independent and uniform on $\{1,\ldots,N\}$. The Gibbs measure is proportional to $e^{H_N^{\rm model}}$. The main examples, at inverse temperature $\beta>0$, are
\begin{align}
 \theta(\sigma_1,\ldots,\sigma_K)&=\beta J\sigma_1\cdots\sigma_K
                 &&\text{(diluted $K$-spin)},\label{eq:pspin}\\
 \theta(\sigma_1,\ldots,\sigma_K)&=-\beta\prod_{j=1}^K\frac{1+J_j\sigma_j}{2}
                 &&\text{(random $K$-sat)}.\label{eq:ksat}
\end{align}
In the first example $J$ is standard Gaussian or a symmetric sign; in the second, the $J_j$ are independent symmetric signs. For a general clause law, we impose the integrability condition
\begin{equation}\label{eq:clauseintegrability}
 J_0=\max_{\sigma\in\{-1,+1\}^K}|\theta(\sigma)|,
 \qquad \E e^{tJ_0}<\infty\quad\text{for every }t>0.
\end{equation}
Both examples satisfy this condition. It is a convenient
sufficient condition for the estimates below; the precise
constant used in the structural proof is given in
\eqref{eq:Cns}.

As in \cite{rRSB}, we will use the same symbol for a clause and its extension to spin means. For a clause with $d$ inputs and $x_j\in[-1,1]$, we define
\begin{equation}\label{eq:extension}
 e^{\theta(x_1,\ldots,x_d)}
 =\sum_{\sigma_1,\ldots,\sigma_d=\pm1}
    e^{\theta(\sigma_1,\ldots,\sigma_d)}
       \prod_{j=1}^d\frac{1+\sigma_jx_j}{2}.
\end{equation}
In words: the Boltzmann factor, not its logarithm, is the average of the Boltzmann factors at the corners of the cube, when the spins are independent with means $x_j$. Thus $e^{\theta(x)}$ is affine in each input $x_j$ separately. 

\subsection{The M\'ezard--Parisi formula}\label{sec:MPformula}
For an integer $r\ge1$, let
\[
 \cA_r=\bigcup_{j=0}^r\N^j,\qquad \N^0=\{*\},
\]
be the rooted tree in which every vertex of depth less
than $r$ has children indexed by $\N$. Write
$\beta\preceq\alpha$ when $\beta$ is an initial segment
of $\alpha$, and let
$p(\alpha)=\{\beta:*\preceq\beta\preceq\alpha\}$ be
the path from the root to $\alpha$, including the root.
The depth is $|\alpha|$, and $\alpha\wedge\beta$ is the
depth of their last common ancestor.

Choose parameters $0<\zeta_0<\cdots<\zeta_{r-1}<1$
and let $(V_\alpha)_{\alpha\in\N^r}$ be the weights
of the corresponding Ruelle probability cascade
\cite{Ruelle}, with children ordered by decreasing
total mass. We refer to \cite[Chapter 2]{SKbook} for
the construction and its basic properties.
The additional functional order parameter is a
measurable function $h_0:[0,1]^r\to[-1,1]$. For every
input site $I$, take an independent array of uniform
variables $(Z_\beta^I)_{\beta\in\cA_r\setminus\{*\}}$,
also independent of the weights, and set
\begin{equation}\label{eq:MPstandard}
 m_\alpha^I
   =h_0((Z_\beta^I)_{\beta\in p(\alpha)\setminus\{*\}}),
       \qquad \alpha\in\N^r.
\end{equation}
The same function is used at every site; the site
arrays are independent copies.

Let $N_1\sim\operatorname{Pois}(\lambda K)$ and
$N_2\sim\operatorname{Pois}(\lambda(K-1))$. With
independent clauses and input-site arrays, define
\begin{align}
 A_\alpha^0(\varsigma)
   &=\sum_{a=1}^{N_1}\theta_{1,a}
       (m_\alpha^{1,a,1},\ldots,m_\alpha^{1,a,K-1},\varsigma),
         \qquad \varsigma=\pm1,\label{eq:MPcavity0}\\
 B_\alpha^0
   &=\sum_{a=1}^{N_2}\theta_{2,a}
       (m_\alpha^{2,a,1},\ldots,m_\alpha^{2,a,K}).
                                                        \label{eq:MPcorrection0}
\end{align}
All counts, clauses and site variables are independent
of each other and of the cascade weights.
Writing $\Av_\varsigma f(\varsigma)=(f(1)+f(-1))/2$,
the M\'ezard--Parisi functional is
\begin{equation}\label{eq:MPfunctional0}
 \cP_0(r,\boldsymbol\zeta,h_0)
 =\log2+\E\log\sum_\alpha V_\alpha
                         \Av_\varsigma e^{A_\alpha^0(\varsigma)}
        -\E\log\sum_\alpha V_\alpha e^{B_\alpha^0}.
\end{equation}
The expectation averages all the randomness just
introduced. The infimum below is over $r$, the
cascade parameters and the measurable functions $h_0$.

The following is the main result of the paper.
\begin{theorem}[M\'ezard--Parisi formula]\label{thm:MP}
For the diluted $K$-spin model \eqref{eq:pspin} with
even $K\ge2$ and $J$ standard Gaussian or a symmetric
sign, and for random $K$-sat \eqref{eq:ksat} with
$K\ge2$, the free energy has the limit
\begin{equation}\label{eq:MPformula}
 \lim_{N\to\infty}F_N^{\rm model}
   =\inf_{r,\boldsymbol\zeta,h_0}\cP_0(r,\boldsymbol\zeta,h_0).
\end{equation}
More generally, for every clause law satisfying the
symmetry and integrability \eqref{eq:clauseintegrability},
\begin{equation}\label{eq:lowerfinal}
 \liminf_{N\to\infty}F_N^{\rm model}
   \ge\inf_{r,\boldsymbol\zeta,h_0}\cP_0(r,\boldsymbol\zeta,h_0).
\end{equation}
\end{theorem}

\paragraph{The Franz--Leone upper bound.}
For the two classes of models in \eqref{eq:MPformula},
the opposite inequality holds at every $N$:
\begin{equation}\label{eq:FLupper}
 F_N^{\rm model}
   \le\inf_{r,\boldsymbol\zeta,h_0}\cP_0(r,\boldsymbol\zeta,h_0).
\end{equation}
Franz and Leone \cite{FL} proved this for even $K$
in both models; see also \cite{PT}. Talagrand observed
that the argument applies to every $K\ge2$ in $K$-sat
\cite{Talagrand,rRSB}. These interpolation bounds are
the diluted counterparts of Guerra's bound \cite{Guerra}
for the Sherrington-Kirkpatrick model.
It remains to prove \eqref{eq:lowerfinal}, and the
rest of the paper is devoted to this lower bound.
The proof does not require the restrictions used
for the upper bound.

\section{The cavity lower bound}\label{sec:lowerbound}

\subsection{The diluted REM perturbation}\label{sec:REMperturbation}
To prove the lower bound, we add a small perturbation
that will be used to enforce M\'ezard--Parisi reproducibility.
Its strength $\epsilon$ remains positive during the
structural argument and is sent to zero at the end.

Fix $\epsilon>0$. For every $d\ge2$, let $\theta_{d,a}^{\epsilon}$ be independent REM clauses, meaning that their $2^d$ values are independent Gaussians:
\begin{equation}\label{eq:tables}
 \bigl(\theta_{d,a}^{\epsilon}(\sigma)\bigr)_{\sigma\in\{-1,+1\}^d},
 \qquad \theta_{d,a}^{\epsilon}(\sigma)\sim N(0,\epsilon 2^{-4d}).
\end{equation}
Let $g_i\sim N(0,\epsilon)$, let $\pi_d(N)$ be Poisson with mean $N$, and take independent uniform indices $j_{d,a,t}$. All these variables are mutually independent and independent of the original model. Set
\begin{align}
 H_N^{\mathrm{pert},\epsilon}(\sigma)
 &=\sum_{i=1}^N g_i\frac{1+\sigma_i}{2}
   +\sum_{d=2}^{\infty}\sum_{a=1}^{\pi_d(N)}
     \theta_{d,a}^{\epsilon}
       (\sigma_{j_{d,a,1}},\ldots,\sigma_{j_{d,a,d}}),\label{eq:Hpert}\\
 H_N^\epsilon&=H_N^{\rm model}+H_N^{\mathrm{pert},\epsilon},
 \qquad F_N^\epsilon=\frac1N\E\log\sum_\sigma e^{H_N^\epsilon(\sigma)}.
                                                        \label{eq:Hfull}
\end{align}
This perturbation has the same extensive Poisson form as the one in \cite{rRSB}. The difference is in the clauses: in \cite{rRSB} a clause of arity $d$ is $g_d\prod_j(1+\sigma_j)/2$ with a single Gaussian coefficient, while a REM clause has $2^d$ independent Gaussian values. The power of this new choice of perturbation will become clear in Section \ref{sec:probes}, where it will allow us to prescribe an arbitrary positive multiaffine Boltzmann factor in the cavity equations. In particular, this will significantly simplify the induction compared to \cite{rRSB}, making it feasible to extend the finite-RSB result to general overlap distributions.

For $J_d=\max_\sigma|\theta_d^\epsilon(\sigma)|$, the Gaussian maximum bound gives
\begin{equation}\label{eq:tablemax}
 \E J_d\le C\sqrt\epsilon\sqrt d\,2^{-2d}.
\end{equation}
Consequently the series in \eqref{eq:Hpert} converges absolutely, uniformly over the finite set $\Sigma_N$, almost surely, and
\begin{equation}\label{eq:pressurepert}
 |F_N^\epsilon-F_N^{\rm model}|
 \le\E|g|+\sum_{d\ge2}\E J_d\le C\sqrt\epsilon.
\end{equation}

\subsection{Asymptotic Gibbs measures}\label{sec:asymptotic}
Sample infinitely many replicas from the Gibbs measure
of $H_N^\epsilon$ and consider the joint law of their
spins, averaged over the disorder. These laws are
exchangeable in the replica index and in the $N$ site
indices. Subsequential limits of their finite-dimensional
distributions, as $N\to\infty$, are therefore separately
exchangeable infinite arrays. By the Aldous-Hoover
representation \cite{Aldous,Hoover,Kallenberg}, in the case of $\pm 1$ spins, such a limit can be described by a random probability measure $G$ on
\[
 \mathcal H=\{m\in L^2([0,1]): |m(v)|\le1
                                  \text{ almost everywhere}\}.
\]
This is the asymptotic Gibbs measure of
\cite{Spins,rRSB}; see also \cite{Austin} for a formulation
through exchangeable random measures. A \emph{state}
is a point of $\supp G$. Conditionally on $G$, sample
independent replicas $m^1,m^2,\ldots$ from $G$. With
independent uniform variables $v_i$ and $x_{i,\ell}$,
independent also of $G$ and the sampled states, the Aldous-Hoover
representation for the $\pm 1$ spins can be written as
\begin{equation}\label{eq:spins}
 \sigma_i^\ell=2\one\left\{x_{i,\ell}
                   \le\frac{1+m^\ell(v_i)}2\right\}-1.
\end{equation}
The same $v_i$ is used across replicas at site $i$;
$m^\ell(v_i)$ is its magnetization coordinate in state $m^\ell$ obtained by averaging over the thermal coin flip $x_{i,\ell}$.
Write $\langle\,\cdot\,\rangle$ for the average over
replicas sampled from $G$, and $\E$ for the remaining
randomness, including the site variables and thermal
coins when they occur in an observable. Define the
overlap and multioverlaps by
\begin{equation}\label{eq:physicaloverlap}
 R_{\ell,\ell'}=\int m^\ell(v)m^{\ell'}(v)\,dv,
 \qquad
 R_{\ell_1,\ldots,\ell_k}^{(k)}
       =\int\prod_{j=1}^k m^{\ell_j}(v)\,dv.
\end{equation}
For distinct replica indices these are measurable
functions of the spin array. Indeed, conditionally
on the states, the site products are independent and
identically distributed, and the strong law gives
\[
 \frac1M\sum_{i=1}^M
          \sigma_i^{\ell_1}\cdots\sigma_i^{\ell_k}
 \longrightarrow R_{\ell_1,\ldots,\ell_k}^{(k)}
       \quad\text{almost surely}.
\]
The distinct-index restriction matters: the diagonal
$R_{\ell\ell}=\|m^\ell\|_2^2$ is the squared norm of
the magnetization profile, whereas $(\sigma_i^\ell)^2=1$.
We call $R_{\ell\ell'}$ the \emph{physical overlap}.

\subsection{Properties of asymptotic Gibbs measures}\label{sec:identities}

For the cavity lower bound, we also introduce the usual generic Gaussian and Poisson cavity regularizations, both with vanishing free energy per spin. These can be chosen jointly, together with a subsequence, so that the limiting spin distribution satisfies the Ghirlanda–Guerra identities and cavity equations below. This is a standard construction used, for example, in \cite{Spins, SKbook}. To avoid a lengthy repetition of these arguments, we record the details in a supplementary note \cite{GGcavInput}.

\paragraph{The Ghirlanda--Guerra identities for spins.}
We will assume that, for every $n\ge2$, every bounded measurable function $f$ of the spins of the first $n$ replicas and every
bounded measurable function $\psi$ of one overlap,
\begin{equation}\label{eq:GG}
 \E\langle f\psi(R_{1,n+1})\rangle
 =\frac1n\E\langle f\rangle\E\langle\psi(R_{1,2})\rangle
  +\frac1n\sum_{\ell=2}^n\E\langle f\psi(R_{1,\ell})\rangle.
\end{equation}
These identities \cite{GG} are obtained by adding a
small generic mixed $K$-spin Gaussian perturbation and
using free-energy concentration and Gaussian integration
by parts; see \cite[Section 3.2]{SKbook}. The test $f$
may depend on all spins of the first $n$ replicas,
with the same proof. This spin form was established
for $K$-sat in \cite{HEPS}; the argument for the other
clause laws considered here is the same under
\eqref{eq:clauseintegrability}. The additional Gaussian
regularization has vanishing free energy per spin;
it is separate from the REM perturbation, whose
strength $\epsilon$ is fixed at this stage.

\paragraph{The cavity equations.} At an added site $i$, the limiting cavity field includes both the model clauses and the perturbation. We use a fresh site coordinate in every incoming argument, with tuples $(i,d,a,j)$ serving as convenient site indices instead of a single index. Define
\begin{align}
 A_i(m,\varsigma)
 &=\sum_{a=1}^{N_{i,0}}\theta_{i,a}
       (m(v_{i,0,a,1}),\ldots,m(v_{i,0,a,K-1}),\varsigma)
       +g_i\frac{1+\varsigma}{2}\notag\\
 &\quad+\sum_{d\ge2}\sum_{a=1}^{N_{i,d}}
   \theta_{i,d,a}^{\epsilon}
       (m(v_{i,d,a,1}),\ldots,m(v_{i,d,a,d-1}),\varsigma),
                                                     \label{eq:cavityfield}
\end{align}
where $\varsigma=\pm1$ is the spin of the added site, $N_{i,0}\sim\operatorname{Pois}(\lambda K)$, and $N_{i,d}\sim\operatorname{Pois}(d)$. Counts, couplings and site coordinates are independent for different cavity sites and independent of $G$. Following \cite{rRSB}, let
\begin{equation}\label{eq:outputs}
 A_i(m)=\log\Av_\varsigma e^{A_i(m,\varsigma)},\quad
 \xi_i(m)=\frac{\Av_\varsigma\varsigma e^{A_i(m,\varsigma)}}{e^{A_i(m)}}.
\end{equation}
Thus $e^{A_i(m)}$ is the factor by which the state $m$ is reweighted when a cavity site $i$ is added, and $\xi_i(m)$ is the magnetization of the added site in that state. 
One can also express this in terms of the scalar effective field $[A_i(m,1)-A_i(m,-1)]/2$ as
$$
\xi_i(m) = \tanh\Bigl(\frac{A_i(m,1)-A_i(m,-1)}{2}\Bigr).
$$
We call $e^{A_i(m)}$ a \emph{multiplier} and
$\xi_i(m)$ a \emph{cavity output}.
For each input site $(i,d,a,j)$, its \emph{site array}
consists of the values $m(v_{i,d,a,j})$ as $m$ varies
over the states.

For $s$ cavity sites, let
\begin{equation}\label{eq:totalmultiplier}
 \Lambda(m)=\exp\sum_{i=1}^s A_i(m),\qquad \cZ=\langle\Lambda\rangle.
\end{equation}
For subsets $C_\ell\subseteq\{1,\ldots,L_0\}$, where $L_0\ge s$ and $1\le\ell\le n$, write $C_\ell^1=C_\ell\cap\{1,\ldots,s\}$ for the cavity sites and $C_\ell^2=C_\ell\setminus C_\ell^1$ for the noncavity sites. The cavity equations are
\begin{equation}\label{eq:SC}
 \E\prod_{\ell=1}^n\left\langle\prod_{i\in C_\ell}m(v_i)\right\rangle
 =\E\prod_{\ell=1}^n
 \frac{\left\langle \Lambda(m)\prod_{i\in C_\ell^1}\xi_i(m)
                        \prod_{i\in C_\ell^2}m(v_i)\right\rangle}{\cZ}.
\end{equation}
In words: a joint moment of spins at $L_0$ sites can be computed either directly, or by regarding the first $s$ sites as newly added sites, reweighting the states by $\Lambda$, and using the cavity magnetizations $\xi_i$ at those sites. The noncavity coordinates $v_i$ on the right are independent of the incoming coordinates in the cavity fields. The identity holds after averaging the clause disorder.

The cavity equations were obtained in \cite{Spins}
for standard diluted models; adding the REM clauses
changes only the list of clause types in that argument. The perturbation, or stabilization, argument in \cite{Spins} that yields these identities can be understood by working with systems of $N+\operatorname{Pois}(c_N)$ sites instead of $N$,
where, for example, $c_N=\log N$. Adding a fixed number
$s$ of sites changes the Poisson law by a total-variation
error tending to zero.  This regularization has vanishing free energy per spin, just like the generic Gaussian regularization for the Ghirlanda--Guerra identities, so it does not affect the free energy. We refer to \cite{GGcavInput} for details.

\subsection{The cavity functional and lower-bound theorem}\label{sec:cavitybound}

Next, we recall the Aizenman--Sims--Starr cavity functional derived in
\cite{Spins}. Recall the cavity field $A(m,\varsigma)$ in \eqref{eq:cavityfield}, and define
\begin{align}
 B(m)&=\sum_{a=1}^{M_0}\theta_a(m(v_{0,a,1}),\ldots,m(v_{0,a,K}))\notag\\
 &\quad+\sum_{d\ge2}\sum_{a=1}^{M_d}
          \theta_{d,a}^{\epsilon}(m(v_{d,a,1}),\ldots,m(v_{d,a,d})),
                                                     \label{eq:correction}
\end{align}
where $M_0\sim\operatorname{Pois}(\lambda(K-1))$, $M_d\sim\operatorname{Pois}(d-1)$, and the clauses use independent site coordinates $(d,a,j).$ This is the field in the correction term of the cavity functional
\begin{equation}\label{eq:PG}
 \cP_\epsilon(G)
 =\log2+\E\log\langle\Av_\varsigma e^{A(m,\varsigma)}\rangle
         -\E\log\langle e^{B(m)}\rangle.
\end{equation}
The cavity method in \cite{Spins} gives the following starting point
for the lower bound. Let $\cG_\epsilon$ be the class of random spin laws
satisfying \eqref{eq:GG} and \eqref{eq:SC} for the
fields \eqref{eq:cavityfield}. The structural theorem
below will apply to every law in this class.

\begin{theorem}[Cavity lower bound]\label{thm:cavitybound}
Under \eqref{eq:clauseintegrability}, for every fixed
$\epsilon>0$ there is $G_\epsilon\in\cG_\epsilon$,
obtained as a subsequential spin limit of regularized
finite systems, such that
\begin{equation}\label{eq:cavityinput}
\liminf_{N\to\infty}F_N^\epsilon\geq  \cP_\epsilon(G_\epsilon).
\end{equation}
In particular,
\begin{equation}\label{eq:ASSinput}
 \liminf_{N\to\infty}F_N^\epsilon
       \ge\inf_{G\in\cG_\epsilon}\cP_\epsilon(G).
\end{equation}
\end{theorem}
This is the REM counterpart of \cite[Theorem 1]{rRSB}.
The required modifications of the standard arguments are carried out in \cite{GGcavInput}. 

The remaining task is to prove that the laws in
$\cG_\epsilon$ have the M\'ezard--Parisi structure in
the approximate sense needed by the functional.
This is Theorem \ref{thm:main} below. The first key step is to use the hierarchical exchangeability of the asymptotic Gibbs measures satisfying the Ghirlanda--Guerra identities proved in \cite{HEPS}, together with  an analogue of the Aldous-Hoover representation for such arrays proved in \cite{AP}.

\section{M\'ezard--Parisi reproducibility}\label{sec:reproducibility}

Throughout this section, $G\in\cG_\epsilon$ is a random spin law satisfying the Ghirlanda--Guerra identities and the cavity equations.

\subsection{Overlap geometry, cuts and cluster barycenters}\label{sec:cutgeometry}
Let $\zeta$ be the law of $R_{1,2}$ under
$\E G^{\otimes2}$ and put $q_*=\sup\supp\zeta$.
The standard geometric consequences of the Ghirlanda--Guerra identities \cite{Ultra,SKbook,Talagrand} are
\begin{equation}\label{eq:GGconsequences}
 G(\|m\|_2^2=q_*)=1,\qquad
 R_{1,2}\ge0,\qquad
 R_{2,3}\ge\min(R_{1,2},R_{1,3})
       \quad\text{almost surely}.
\end{equation}
Thus all states lie on one sphere, the overlaps are
nonnegative and their geometry is ultrametric. In
particular, the diagonal of every physical overlap
matrix below is $q_*$, rather than the spin
self-overlap $1$. If $\zeta$ is a point mass, then
all states have the same norm and scalar product,
so $G$ is supported on one function and the structural
conclusion is immediate. We henceforth consider the
case in which $\zeta$ is not a point mass.

Choose thresholds
\[
 0=t_0<t_1<\cdots<t_r=q_{\rm top}<q_* .
\]
Let $I_j=[t_j,t_{j+1})$ for $0\le j<r$, and $I_r=[t_r,q_*]$. By ultrametricity, for each threshold $t_j$ the relation $m\cdot m'\ge t_j$ is an equivalence relation on $\supp G$, and the equivalence classes for $t_0,t_1,\ldots,t_r$ are nested. They form the \emph{cluster tree} of the cut. The two quantities measuring fineness on the overlap support are
\begin{equation}\label{eq:cuterrors}
 \delta=\max_{j<r}\diam(I_j\cap\supp\zeta),
 \qquad \eta=q_*-\inf(I_r\cap\supp\zeta).
\end{equation}
An empty intersection contributes zero to the maximum. We remove the levels of zero overlap mass and relabel, so that from now on every bin carries positive mass. Such cuts can be chosen with both $\delta$ and $\eta$ arbitrarily small.

Using the tree notation of Section \ref{sec:MPformula}, we enumerate the clusters as follows. The children of every cluster are enumerated in decreasing order of their masses. By \eqref{eq:GG} every cluster has infinitely many children, with distinct positive masses almost surely \cite{HEPS,SKbook}, so this enumeration identifies the cluster tree with $\cA_r$. We denote the clusters by $(C_\beta)_{\beta\in \cA_r}$ and their Gibbs weights by $ V_\beta=G(C_\beta).$ For a terminal cluster $C_\alpha$, $\alpha\in\N^r$, define
\begin{equation}\label{eq:barycenter}
 V_\alpha=G(C_\alpha),\qquad G_\alpha=G(\,\cdot\mid C_\alpha),\qquad
 b_\alpha(v)=\int m(v)\,G_\alpha(dm).
\end{equation}
The function $b_\alpha$ is the barycenter of the cluster; it is the magnetization profile of the approximate state $\alpha$. If $\alpha\wedge\beta=j<r$, the overlaps of states in $C_\alpha$ and $C_\beta$ lie in the level-$j$ bin, hence $b_\alpha\cdot b_\beta$ lies in its convex hull. For each terminal cluster,
\begin{align}
 q_*-\eta=\inf(I_r\cap\supp\zeta)\le\|b_\alpha\|_2^2&\le q_* ,\notag\\
 \int\|m-b_\alpha\|_2^2\,G_\alpha(dm)
 &=q_*-\|b_\alpha\|_2^2\le\eta,\label{eq:baryerror}\\
 \E_v\int|m(v)-b_\alpha(v)|\,G_\alpha(dm)&\le\sqrt\eta.
                                                        \label{eq:baryL1}
\end{align}
For the first line: almost surely $G\otimes G(R_{1,2}\notin\supp\zeta)=0$, since its expectation is $\zeta\bigl((\supp\zeta)^c\bigr)=0$, and two states of one terminal cluster have overlap at least $t_r$; hence $G_\alpha\otimes G_\alpha$-almost every pair has overlap in $I_r\cap\supp\zeta$, and $\|b_\alpha\|_2^2=\iint m\cdot m'\,G_\alpha(dm)\,G_\alpha(dm')\ge\inf(I_r\cap\supp\zeta)$; the upper bound is the Cauchy--Schwarz inequality. The second line uses $\|m\|_2^2=q_*$ for all states and the fact that $b_\alpha$ is the mean of $G_\alpha$; the third is Cauchy--Schwarz. These bounds hold for every cluster, not only on average over the weights. 

\subsection{Cascade weights and the hierarchical representation}\label{sec:cutrepresentation}

We now recall another implication of the Ghirlanda--Guerra identities proved in  \cite{HEPS,AP}. Define cut parameters and bin masses by
\begin{equation}\label{eq:zetapar}
 \zeta_j=\zeta(I_0\cup\cdots\cup I_j),\qquad
 0=\zeta_{-1}<\zeta_0<\cdots<\zeta_{r-1}<\zeta_r=1,
 \qquad w_j=\zeta_j-\zeta_{j-1}.
\end{equation}
The number $w_j$ is the probability that two replicas, sampled independently from $G$, have their last common ancestor at level $j$, that is, an overlap in $I_j$; $w_r$ is the probability that they fall in the same terminal cluster.

The GG identities imply \cite{HEPS,AP} that the cluster weights $(V_\beta)_{\beta\in \cA_r}$ have the Ruelle probability cascade (RPC) law with parameters \eqref{eq:zetapar}, and the weights are independent of the barycenter site array $(b_\alpha(v_i))_{\alpha\in\N^r,i\ge1}$. Moreover, this array has the following representation:
\begin{equation}\label{eq:he}
 b_\alpha^i=h\bigl((U_\beta)_{\beta\in p(\alpha)},
                         (Z_\beta^i)_{\beta\in p(\alpha)}\bigr),
 \qquad |h|\le1,
\end{equation}
where $h:[0,1]^{2(r+1)}\to[-1,1]$ is a deterministic measurable function and all the displayed uniform variables are mutually independent and independent of the weights. The variables $U_\beta$ are shared by all sites; different sites use independent arrays $(Z_\beta^i)_\beta$. Both families include the root, and $U_*$ is retained as a global mixture parameter. The representation retains the physical scalar products and multioverlaps: for any finite tuple of leaves $\boldsymbol\alpha=(\alpha_1,\ldots,\alpha_k)$, with $\E_Z$ the average over all the site variables of one site,
\begin{equation}\label{eq:heoverlaps}
 \begin{aligned}
 &b_\alpha\cdot b_\beta
   =\int b_\alpha(v)b_\beta(v)\,dv
     =\E_Z b_\alpha^1b_\beta^1,\\
 &M_{\boldsymbol\alpha}
   :=\int\prod_{\ell=1}^k b_{\alpha_\ell}(v)\,dv
     =\E_Z\prod_{\ell=1}^k b_{\alpha_\ell}^1.
 \end{aligned}
\end{equation}
simultaneously for all tuples, on one event of probability one. Here the left sides, defined in the original realization, are carried to the representation as almost sure limits of site averages; what is used later is the resulting equality of the joint laws of the weights, the represented array and these quantities.

The following is the barycenter formulation of
Theorem 1 and Corollary 1 of \cite{HEPS}, as explained
in the discussion following that corollary, which
also gives the multioverlap identities
\eqref{eq:heoverlaps}.

\begin{lemma}[Hierarchical representation at a cut]\label{lem:A3}
Under \eqref{eq:GG}, the statements \eqref{eq:he} and \eqref{eq:heoverlaps} hold at every cut as above.
\end{lemma}

The representation is cut-dependent: neither $h$ nor the particular root variable has to be coupled across different cuts, and we never compare representations of different cuts, only averaged quantities.

The representation above contains
state variables shared across sites. We now state the
main structural result: after averaging over the
site variables, the multioverlaps become approximately independent
of the nonroot state variables as the cut becomes fine. This will be the consequence of the cavity equations \eqref{eq:SC} and, in particular, of the perturbation term in the cavity field.

\subsection{Replica shapes and the reproducibility theorem}\label{sec:defect}

Fix $\epsilon>0$ and $G\in\cG_\epsilon$, and choose a cut and its representation as in Section \ref{sec:cutgeometry}. Sample $k$ terminal labels $\boldsymbol\alpha=(\alpha_1,\ldots,\alpha_k)$ from the weights, independently given the weights. Their \emph{shape} $T=T(\boldsymbol\alpha)$ is the $k\times k$ matrix of last-common-ancestor depths $\alpha_\ell\wedge\alpha_{\ell'}$. Terminal labels may repeat, in which case the corresponding entry is $r$. Choose strictly increasing representatives
\begin{equation}\label{eq:qrepresentatives}
 q_j\in\overline{\conv(I_j\cap\supp\zeta)}\quad(j<r),
 \qquad q_r=q_*.
\end{equation}
The \emph{quantized overlap matrix} of the shape is
\begin{equation}\label{eq:qmatrix}
 q(T)_{\ell,\ell'}=
 \begin{cases}q_{\alpha_\ell\wedge\alpha_{\ell'}},&\ell\ne\ell',\\
 q_*,&\ell=\ell'.\end{cases}
\end{equation}
Its entries differ from the overlaps of the sampled states, and from the scalar products of their cluster barycenters, by at most $\max(\delta,\eta)$; in the diagonal barycenter comparison, by at most $\eta$. All equalities and inequalities between the levels of the shape can be read from $q(T)$.

The RPC probability of a labeled shape is
\begin{equation}\label{eq:muglobal}
 \mu_k^0(T)=\E\sum_{\boldsymbol\alpha:T(\boldsymbol\alpha)=T}
                                 V_{\alpha_1}\cdots V_{\alpha_k}.
\end{equation}
Shapes are labeled throughout, and $\sum_T\mu_k^0(T)=1$. The law $\mu_k^0$ depends only on the cascade parameters \eqref{eq:zetapar}; its explicit form is given in Appendix \ref{app:shapes}.

Let $\E_U$ denote the average over the nonroot state variables, with $U_*$ retained. The \emph{defect} is
\begin{equation}\label{eq:Dk}
 D_k=\E\sum_{\boldsymbol\alpha}V_{\alpha_1}\cdots V_{\alpha_k}
       \bigl(M_{\boldsymbol\alpha}
                    -\E_U[M_{\boldsymbol\alpha}\mid U_*]\bigr)^2 .
\end{equation}
By the independence of the weights from the represented array and by \eqref{eq:he}, the conditional mean $\E_U[M_{\boldsymbol\alpha}\mid U_*]$ depends only on $U_*$ and on the shape of $\boldsymbol\alpha$. Thus $D_k$ is the average, over shapes with their RPC probabilities, of the conditional variance of the multioverlap given $U_*$ and the shape. If $h$ has the M\'ezard--Parisi form, $D_k=0$ exactly. If $D_k$ is small, then multioverlaps depend approximately only on \(U_*\) and the shape. This is enough to prove the M\'ezard--Parisi formula.

\begin{theorem}[Approximate reproducibility]\label{thm:main}
For every fixed $\epsilon>0$ and $G\in\cG_\epsilon$,
along every sequence of finite cuts with
$\delta,\eta\to0$ and with no restriction on their
numbers of levels,
\[
 D_k\longrightarrow0\qquad\text{for each fixed }k.
\]
The conclusion holds for every choice of the
hierarchical representation at each cut. No common
representation across different cuts is required.
\end{theorem}

The estimates proving this theorem do not depend on
the number of levels. This is what permits general
overlap distributions in place of the finite-RSB
condition of \cite{rRSB}. The next subsection gives
its consequence for the spin law and the variational
lower bound.

\subsection{From reproducibility to the MP lower bound}\label{sec:whydefects}
The nonroot state variables $U_\beta$, which are shared across sites, are not allowed in the M\'ezard--Parisi representation. The natural approximation is to give each site its own copy:
\begin{equation}\label{eq:resampling}
 \widehat b_\alpha^i
 =h\bigl(U_*,(U_\beta^{(i)})_{\beta\in p(\alpha)\setminus\{*\}},
                    (Z_\beta^i)_{\beta\in p(\alpha)}\bigr).
\end{equation}
Here $U_*$ and the weights are retained, and all copied variables are independent over sites and vertices. Each pair $(U_\beta^{(i)},Z_\beta^i)$ can be encoded in one uniform variable $\widehat Z_\beta^i$. Conditionally on $U_*$, the resampled field has the M\'ezard--Parisi form with an additional root site variable:
\begin{equation}\label{eq:MProot}
 \widehat b_\alpha^i
   =\widehat h_{U_*}((\widehat Z_\beta^i)_{\beta\in p(\alpha)}).
\end{equation}

\paragraph{The root site coordinate.}\label{sec:mainstatements}
Accordingly, we allow a measurable function
$h_0:[0,1]^{r+1}\to[-1,1]$ and fields
\begin{equation}\label{eq:MPtrial}
 m_\alpha^I=h_0((Z_\beta^I)_{\beta\in p(\alpha)}).
\end{equation}
All site arrays are independent, including their root
coordinates. Let $\cT_{\rm root}$ denote this enlarged
trial class and $\cT_0$ its subclass that ignores the
root, which is the class used in Section
\ref{sec:MPformula}. Substituting \eqref{eq:MPtrial}
into \eqref{eq:cavityfield} and \eqref{eq:correction}
gives the perturbed trial functional
\begin{equation}\label{eq:Ptrial}
 \cP_\epsilon(r,\boldsymbol\zeta,h_0)
 =\log2+\E\log\sum_\alpha V_\alpha\Av_\varsigma
                                       e^{A_\alpha(\varsigma)}
        -\E\log\sum_\alpha V_\alpha e^{B_\alpha}.
\end{equation}
For these enlarged trials we also write $\cP_0$ when
the perturbation is omitted, extending
\eqref{eq:MPfunctional0}. The two trial classes have
the same infimum, with or without the perturbation:
the root coordinate can be replaced by an extra first
level whose parameter tends to zero. This is proved
in \eqref{eq:sameinf} in Section \ref{sec:bridge}.

\begin{proposition}[Approximation of finite spin distributions]\label{prop:MPapprox}
Let $C_1,\ldots,C_s$ be nonempty subsets of $\{1,\ldots,n\}$. Then
\begin{align}
 &\left|\E\left\langle\prod_{i=1}^s\prod_{\ell\in C_i}\sigma_i^\ell\right\rangle
 -\E\sum_{\alpha_1,\ldots,\alpha_n}
       \left(\prod_{\ell=1}^nV_{\alpha_\ell}\right)
       \prod_{i=1}^s\prod_{\ell\in C_i}\widehat b_{\alpha_\ell}^i\right|\notag\\
 &\hspace{15mm}\le
       \left(\sum_{i=1}^s|C_i|\right)\sqrt\eta
                      +\sum_{i=1}^s\sqrt{D_{|C_i|}}
       \le ns\sqrt\eta+s\max_{1\le j\le n}\sqrt{D_j}.
                                                        \label{eq:resamplingerror}
\end{align}
The second expectation includes the weights, the resampled site arrays, and the common global variable $U_*$. The spin distributions generated by \eqref{eq:MProot} are mixtures over $U_*$ of the trials \eqref{eq:MPtrial}.
\end{proposition}
\begin{proof}
Since $\sigma^2=1$, a product of spins with repetitions reduces to one without, so the monomials on the left determine all finite spin distributions. Given the cluster labels of the $n$ replicas, replace each sampled state by the barycenter of its cluster. Since $\sigma_i^\ell$ has conditional mean $m^\ell(v_i)$ given the state and the site variable, and all factors are bounded by one, this costs at most $\sum_i|C_i|\sqrt\eta$ by \eqref{eq:baryL1}. Put
\[
 \overline M_{\boldsymbol\alpha_C}(U_*)
       =\E_U[M_{\boldsymbol\alpha_C}\mid U_*],
\]
where $\boldsymbol\alpha_C$ is the subtuple of labels indexed by $C$. After averaging the site variables, the barycenter expectation and the resampled expectation are, by \eqref{eq:heoverlaps} and by the independence over sites of the resampled variables,
\[
 \E\sum_{\boldsymbol\alpha}\left(\prod_{\ell=1}^nV_{\alpha_\ell}\right)
                       \prod_{i=1}^s M_{\boldsymbol\alpha_{C_i}},
 \qquad
 \E\sum_{\boldsymbol\alpha}\left(\prod_{\ell=1}^nV_{\alpha_\ell}\right)
                       \prod_{i=1}^s\overline M_{\boldsymbol\alpha_{C_i}}(U_*).
\]
In the first expression the product over sites is evaluated at the same state variables; in the second, each site sees independent nonroot state variables, so its factor is the conditional mean given $U_*$. All factors have absolute value at most one. Telescope the two products, one factor at a time, and apply Cauchy--Schwarz to each difference. Marginalizing the replica labels not used in $C_i$ turns the weights into those of \eqref{eq:Dk} with $k=|C_i|$, so each term is at most $\sqrt{D_{|C_i|}}$.

Finally, conditionally on $U_*=u$, the field \eqref{eq:MProot} is exactly a trial with $h_0=\widehat h_u$, including its independent root site coordinates. Averaging the conditional spin law over $U_*$ gives the mixture.
\end{proof}

Theorem \ref{thm:main} and Proposition
\ref{prop:MPapprox} show that the finite-dimensional
spin distributions are approximated by mixtures of
these finite M\'ezard--Parisi fields. Continuity of the
cavity functional then gives the following consequence.

\begin{theorem}[Variational reduction]\label{thm:variational}
For every fixed $\epsilon>0$ and every $G\in\cG_\epsilon$,
\begin{equation}\label{eq:Pbound}
 \cP_\epsilon(G)\ge\inf_{\cT_{\rm root}}\cP_\epsilon.
\end{equation}
\end{theorem}

The continuity argument and the removal of the global
mixture, root site coordinate and perturbation are
given in Section \ref{sec:bridge}. Together with the
cavity lower bound of Theorem \ref{thm:cavitybound},
they prove \eqref{eq:lowerfinal}, and hence the
formula of Theorem \ref{thm:MP} when the upper bound
\eqref{eq:FLupper} is available. The main task is
therefore Theorem \ref{thm:main}, to which the
following proof sections are devoted.

\section{Outline of the proof}\label{sec:outline}

\paragraph{The induction in one picture.} The main task is to prove the approximate reproducibility statement of Theorem \ref{thm:main}, namely, that the multioverlaps $M_{\boldsymbol\alpha}$ are approximately independent of the nonroot state variables $U$ on average. To utilize positivity, we will be proving an equivalent statement for the so-called \emph{positive multioverlaps},
$$
M_{\boldsymbol\alpha}^S:= \E_Z \prod_{\ell=1}^k S^{\alpha_\ell},\quad
\mbox{ where }\quad S^\alpha=(1+b_\alpha)/2\in[0,1],
$$
so we will describe the induction in terms of these. 

Consider $k$ replicas whose tree $T$ first branches at depth $a$, as in
Figure \ref{fig:induction-stem}. Fix the state and site prefixes $(U_{\le p},Z_{\le p})$ through depth $p\le a$. The remaining variables lie on the common stem $(p,a]$ or in the child subtrees $T_1,\ldots,T_b$.

We will aim to show that averaging $\prod_{\ell=1}^k S^{\alpha_\ell}$ in the site variables $Z_{>p}$ above depth $p$ -- which will be called a \emph{future multioverlap} $F_T^p$ -- makes it almost independent of the state variables $U_{>p}$ also above depth $p$, on average over RPC tree shapes, uniformly in the retained depth \(p\), with estimates independent of the total number of levels. This dependence will be measured by \emph{covariance norms} introduced in Section \ref{sec:induction}.  These norms might look mysterious at first, but in fact they are very natural, originating in Lemma 2 in \cite{rRSB}.

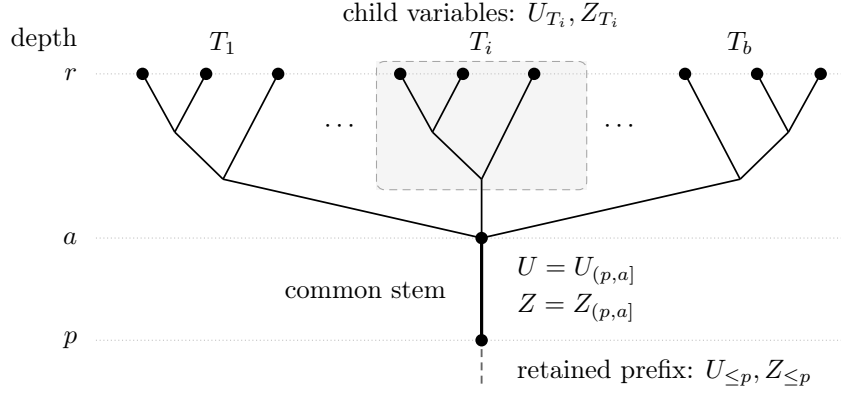
\begin{figure}[t]
\centering
\begin{tikzpicture}[x=1.12cm,y=0.82cm,
  every node/.style={font=\small},
  outline edge/.style={line width=0.65pt},
  outline dot/.style={circle,fill=black,inner sep=1.65pt}]
  \fill[black!4] (-1.24,2.43) rectangle (1.24,4.50);
  \draw[black!35,densely dashed,rounded corners=3pt]
    (-1.24,2.43) rectangle (1.24,4.50);
  \foreach \yy/\lab in {0/p,1.65/a,4.30/r}{
    \draw[black!25,densely dotted] (-4.55,\yy)--(4.35,\yy);
    \node[anchor=east] at (-4.65,\yy) {$\lab$};
  }
  \node[anchor=east] at (-4.65,4.85) {depth};

  \draw[black!60,densely dashed,line width=0.7pt] (0,-0.70)--(0,0);
  \draw[line width=1.2pt] (0,0)--(0,1.65);
  \node[outline dot] at (0,0) {};
  \node[outline dot] at (0,1.65) {};
  \node[anchor=west,align=left] at (0.30,0.82)
    {$U=U_{(p,a]}$\\[2pt]$Z=Z_{(p,a]}$};
  \node[anchor=east] at (-0.30,0.82) {common stem};
  \node[anchor=west] at (0.30,-0.52)
    {retained prefix: $U_{\le p},Z_{\le p}$};

  \draw[outline edge] (0,1.65)--(-3.05,2.60);
  \draw[outline edge] (0,1.65)--(0,2.60);
  \draw[outline edge] (0,1.65)--(3.05,2.60);
  \node at (-1.65,3.45) {$\cdots$};
  \node at (1.65,3.45) {$\cdots$};

  \draw[outline edge] (-3.05,2.60)--(-3.62,3.36);
  \draw[outline edge] (-3.05,2.60)--(-2.40,4.30);
  \draw[outline edge] (-3.62,3.36)--(-4.00,4.30);
  \draw[outline edge] (-3.62,3.36)--(-3.25,4.30);
  \draw[outline edge] (0,2.60)--(-0.58,3.36);
  \draw[outline edge] (0,2.60)--(0.62,4.30);
  \draw[outline edge] (-0.58,3.36)--(-0.96,4.30);
  \draw[outline edge] (-0.58,3.36)--(-0.22,4.30);
  \draw[outline edge] (3.05,2.60)--(2.40,4.30);
  \draw[outline edge] (3.05,2.60)--(3.62,3.36);
  \draw[outline edge] (3.62,3.36)--(3.25,4.30);
  \draw[outline edge] (3.62,3.36)--(4.00,4.30);
  \foreach \xx in {-4.00,-3.25,-2.40,-0.96,-0.22,0.62,2.40,3.25,4.00}
    \node[outline dot] at (\xx,4.30) {};
  \node at (-3.05,4.80) {$T_1$};
  \node at (0,4.80) {$T_i$};
  \node at (3.05,4.80) {$T_b$};
  \node at (0,5.25) {child variables: $U_{T_i},Z_{T_i}$};
\end{tikzpicture}
\caption{Induction controls the dependence on the child state
variables $U_{T_i}$ after averaging the child site variables $Z_{T_i}$. The common stem remains to be treated by the cavity equations.}
\label{fig:induction-stem}
\end{figure}

Let $D_i$ be the replicas in child $i$. Independence of the child site variables, with all state variables fixed, expresses the future multioverlap $ F_T^p$ as
\[
 F_T^p=\E_{Z_{(p,a]}}\prod_{i=1}^b f_i,\quad
 \mbox{ where}\quad
 f_i=\E_{Z_{T_i}}\prod_{\ell\in D_i}S^{\alpha_\ell}.
\]
Each $f_i$ is itself a future multioverlap of the child $i$, and it involves fewer than $k$ replicas. Thus, the induction hypothesis allows us to approximate $f_i$ by $m_i=\E_{U_{T_i}}f_i$, with the prefix through $a$ held fixed. In this way, all child subtrees can be pruned from the tree as in the Branch Replacement Lemma \ref{lem:branch}. The induction hypothesis thus reduces the problem to showing that the remaining \emph{common-stem moment}
\[
 H_T^p=\E_{Z_{(p,a]}}\prod_{i=1}^b m_i
\]
is approximately independent of the common-stem state variables $U_{(p,a]}$. This is where the cavity equations and diluted REM perturbation will play a crucial role, starting from Section \ref{sec:marking}.\vspace{-3mm}

\paragraph{How overlap geometry becomes a cavity constraint.}
Next, let us briefly explain the key idea of how the cavity equations will be used in order to transfer fine-cut geometric information about overlaps to cavity outputs.
Consider four replicas that separate into four distinct branches
at depth $p$. For this paragraph, denote the overlaps of their barycenters by $R_{ij}=b_{\alpha_i}\cdot b_{\alpha_j}$. All six scalar products
belong to one interval of width at most $\delta$, so
\begin{align*}
&\E\bigl[R_{12}^2-2R_{12}R_{13}+R_{13}R_{24}\bigr]
=\E\bigl[(R_{12}-R_{13})(R_{12}-R_{24})\bigr]
 \le\delta^2,
\end{align*}
where the equality uses replica symmetry.
Now, let us notice that there is a way to rewrite the left-hand side in terms of a conditional covariance of two barycenter spins $b_{\alpha_1}^i$ for $i=1,2$ in \eqref{eq:he} at one replica $\alpha_1$: 
\[
 C_p^b=\Cov(b_{\alpha_1}^1,b_{\alpha_1}^2\mid\mathcal H_p),\quad
 \mbox{ where }\quad
  \mathcal H_p=\sigma(U_{\le p},Z_{\le p}^1,Z_{\le p}^2).
\]
In other words, the covariance is computed with the common prefix variables $U_{\le p},Z_{\le p}^1,Z_{\le p}^2$ fixed. 
It is then easy to check that 
$$
\E (C_p^b)^2 = \E\bigl[R_{12}^2-2R_{12}R_{13}+R_{13}R_{24}\bigr]
 \le\delta^2.
$$
The quantity $\E (C_p^b)^2$ is an example of (the forth power of) a covariance norm.
The cavity equations then allow us to transfer this inequality to the cavity outputs $\xi_i^\alpha$ up to a vanishing cut error for the analogous conditional covariance under the corresponding change of density. This is done in Proposition \ref{prop:cov} below; this also explains why we use the covariance norms in the induction. After this transfer, the cavity magnetizations are correlated through a common change of density, which enforces strong constraints encoded in the perturbation.\vspace{-3mm}

\paragraph{Overview of the rest of the proof.}
These constraints are then leveraged in Section \ref{sec:probes}, where the REM perturbation plays its key role. The fact that the distribution of the random energies vector on a hypercube has full support allows us to translate the covariance estimates above to \emph{prescribed deterministic clauses} with positive multiaffine Boltzmann factors. From this point on, we work with these prescribed clauses and can vary and differentiate their parameters. The derivatives can still be controlled via transferred covariance constraints, and will yield an approximate \emph{factorization of the joint change of density} in Proposition \ref{prop:factorization}. This is another key step.

In Section \ref{sec:response}, the decoupled density estimates will transform the covariance estimates of cavity magnetizations at two cavity sites into control of a one-site cavity quantity in covariance norm. For specific choices of prescribed clauses, the Bolthausen--Sznitman invariance property of the RPC will yield identities for \emph{mixed replica tests} in Proposition \ref{prop:mixedcoeff}.

In Section \ref{sec:graphs}, averages of specific future multioverlaps will be extracted from mixed replica tests. This is another place where freedom to prescribe cavity fields will be key. In Proposition \ref{prop:tests}, we will control covariance norms for averages of common-stem moments weighted by shape test functions, in some sense almost completing the induction. The remaining step is to pass from control of these tested averages to control of the covariance norms averaged over tree shapes. Sections \(\ref{sec:compactness}\) and \(\ref{sec:closing}\) achieve this through an explicit compactness argument on the space of physical overlap matrices.

Finally, Section \(\ref{sec:bridge}\) proves both Theorem \(\ref{thm:variational}\) and the M\'ezard--Parisi formula in Theorem \(\ref{thm:MP}\). Some elementary facts about the covariance norms and RPC probabilities of tree shapes will be relegated to the appendices.

\section{The induction on the number of replicas}\label{sec:induction}

\subsection{Future multioverlaps, and the covariance norm}\label{sec:kernels}

\paragraph{Future multioverlaps and auxiliary cascades.}
Fix a vertex $\beta\in \N^p$ at depth $p$ and suppose that finitely many leaves $\boldsymbol\alpha =(\alpha_1,\ldots,\alpha_k)$ all descend from $\beta$. Recall that their shape 
\[
 T   =(\alpha_\ell\wedge\alpha_{\ell'})_{\ell,\ell'\le k}
\]
is the $k\times k$ matrix of last-common-ancestor depths. For any shape $T$, we define the \emph{first split} of $T$ as
\begin{equation}\label{eq:aT}
 a(T)=\min_{\ell,\ell'}\alpha_\ell\wedge\alpha_{\ell'}.
\end{equation}
In our case $a(T)\geq p.$ 

Denote the state and site variables on the path to $\beta$, as well as the future variables indexed by $ V=\bigcup_{\ell=1}^k 
       \bigl(p(\alpha_\ell)\setminus p(\beta)\bigr)$ by
\begin{align}\label{eq:variables}
 U_{\le p}&=(U_\gamma)_{\gamma\in p(\beta)},
 \quad
 Z_{\le p}=(Z_\gamma)_{\gamma\in p(\beta)},
\\
 U_T&=(U_\gamma)_{\gamma\in V},
 \qquad
 Z_T=(Z_\gamma)_{\gamma\in V}.
\end{align}
Throughout this section a site index $i$ is fixed, so we omit it and write the site variables $Z_\gamma$ instead of $Z_\gamma^i$.

\begin{remark}
For fixed $p$ and $T$, the particular choice of $\beta$
and $\boldsymbol\alpha$ realizing this shape does not
matter: after identifying corresponding vertices, the
state and site variables have the same joint law, and
the formulas below are unchanged. We use concrete
vertices only for convenience, to describe which
variables are shared. We therefore suppress the
particular choice of these vertices in the notation for the state and site variables above.
\end{remark}

Recall the cluster barycenter representation \eqref{eq:he} and write
\begin{equation}\label{eq:posS}
 S^\alpha:=\frac{1+b_\alpha}{2}
 =\frac{1+
 h\bigl((U_\gamma)_{\gamma\in p(\alpha)},
        (Z_\gamma)_{\gamma\in p(\alpha)}\bigr)}{2}.
\end{equation}
This rescaled magnetization takes values in $[0,1]$. Similarly, we define \emph{positive multioverlaps} by
\begin{equation}\label{eq:MaS}
M_{\boldsymbol\alpha}^S:= \E_Z \prod_{\ell=1}^k S^{\alpha_\ell}
  =\iint \prod_{\ell=1}^k S^{\alpha_\ell}\,dZ_T dZ_{\le p}.
\end{equation}
Since each $S=\frac{1+b}{2}$, this is just a linear combination of the usual multioverlaps of subtuples of $\boldsymbol\alpha$, together with the constant term. The advantage here is that each $S\geq0$, so all these multioverlaps are nonnegative.

The following quantity will be fundamental.

\begin{definition}[Future multioverlaps]
Given a shape $T$ with the first split $a(T)\geq p,$ define the \emph{future multioverlap of shape $T$ at depth $p$} by
\begin{equation}\label{eq:FT}
 F_T^p(U_{\le p},U_T,Z_{\le p})
   =\int \prod_{\ell=1}^k S^{\alpha_\ell}\,dZ_T.
\end{equation}
\end{definition}

Notice that averaging $F_T^p(U_{\le p},U_T,Z_{\le p})$ also over the retained site prefix $Z_{\le p}$ gives the full multioverlap $M_{\boldsymbol\alpha}^S$ above.

Recall the probability $ \mu_k^0(T)$ of a labeled shape in \eqref{eq:muglobal}. For shapes with the first split $a(T)\geq p$, we will also want to average them with respect to the measure
\begin{equation}
 \mu_k^p(T)=\E\sum_{\boldsymbol\alpha:T(\boldsymbol\alpha)=T}
                                 V_{\alpha_1}\cdots V_{\alpha_k},
\end{equation}
where $(V_\alpha)$ are the weights of an auxiliary RPC with depths $\{p,\ldots,r\}$ and parameters $\zeta_p,\ldots,\zeta_{r-1}$. In other words, for two replicas from this RPC, the probabilities of their last-common-ancestor depths $\alpha_1\wedge\alpha_{2}$ are
\begin{equation}\label{eq:localweights}
 w_a^{(p)}=
 \begin{cases}
   \zeta_p,&a=p,\\
   \zeta_a-\zeta_{a-1}=w_a,&p<a\le r,
 \end{cases}
 \qquad
 \sum_{a=p}^r w_a^{(p)}=1,
\end{equation}
with the convention $\zeta_r=1$. This corresponds to the root of the tree shifting to depth $p$, while retaining the depth values $\{p,\ldots,r\}$ of the original tree. Note that this is not the original shape law conditioned on all replicas sharing a depth-$p$ cluster.

\paragraph{The covariance norm.}
For a bounded measurable function $f=f(U_{\le p},U_T,Z_{\le p})$ of the variables \eqref{eq:variables} corresponding to some shape $T$ with the first split $a(T)\geq p$, define
\begin{equation}\label{eq:Xnorm}
 \|f\|_p^4
 :=\E_{U_{\le p}}\iint
       \left|
         \int
           f(U_{\le p},U_T,Z)
           f(U_{\le p},U_T,Z')\,dU_T
       \right|^2
       dZ\,dZ'.
\end{equation}
The expression on the right is homogeneous of degree four in \(f\), so we take its fourth root to obtain a norm. The subscript $p$ here refers to the tree depth, not to an $L^p$-norm. Variables $Z$ and $ Z'$ are independent copies of the site prefix $Z_{\le p}$. The two factors in the inner integral use the same future state variables $U_T$.  Lemma \ref{lem:normfacts} in Appendix \ref{app:norm} proves that $\|f\|_p$ is a norm. It also proves the following contraction properties of $\|f\|_p$, together with a comparison with the ordinary $L^2$ norm. Some of these properties follow directly from Jensen's inequality. 

\begin{itemize}
\item Define centering in the future state variables by
\begin{equation}\label{eq:center}
 \Pi_p f
   =f-\E_{U_T}[f\mid U_{\le p},Z_{\le p}].
\end{equation}
Centering in the future state variables is a contraction:
\[
 \|\Pi_p f\|_p\le\|f\|_p.
\]

\item If the future state variables are split into two disjoint
collections $U_T=(U_A,U_B)$, averaging over one collection
also contracts the norm:
\[
 \left\|\E_{U_B}f\right\|_p\le\|f\|_p.
\]
Here $U_{\le p}$, $U_A$ and $Z_{\le p}$ are held fixed.

\item Finally, suppose that $f$ also depends on an extra
independent site prefix $W$. Averaging this prefix gives
\[
 \left\|\E_W f\right\|_p\le\|f\|_p.
\]
On the right, the site variable in \eqref{eq:Xnorm}
is the tuple $(Z_{\le p},W)$; on the left, only
$Z_{\le p}$ remains.

\item 
We will also use the following comparison with the
ordinary $L^2$ norm. For any bounded function $e$ with
$|e|\le M$,
\begin{equation}\label{eq:normconversion}
 \|e\|_p^4
 \le
 \E_{U_{\le p}}
   \left(\E_{U_T,Z_{\le p}}|e|^2\right)^2
 \le M^2\E|e|^2.
\end{equation}
\end{itemize}

We will now define the covariance norm of a bounded measurable function $F=F(U_{\le p},U_T,Z_{\le p})$, with the main example being the future multioverlaps in \eqref{eq:FT}. Recall the centering in \eqref{eq:center} with respect to all future state variables.
The inner integral in (\ref{eq:Xnorm}) with $f=\Pi_p F$,
\begin{align}
&C_T(Z,Z'):= \int
           \Pi_p F(U_{\le p},U_T,Z)
           \Pi_p F(U_{\le p},U_T,Z')\,dU_T,
 \label{eq:CUZZT}
\end{align}
is the covariance, with respect to $U_T$, of $F(U_{\le p},U_T,Z)$ and $F(U_{\le p},U_T,Z')$, with the state prefix $U_{\le p}$ fixed. Its dependence on this prefix is suppressed in the notation.

\begin{definition}[Covariance norm]
Given a bounded function $F=F(U_{\le p},U_T,Z_{\le p})$, we call the norm of its centering, $\|\Pi_p F\|_p$, its covariance norm at depth $p$. Explicitly,
\begin{align}\label{eq:Covnorm}
\|\Pi_p F\|_p^4
 &=\E_{U_{\le p}}\iint
       \left|
        C_T(Z,Z')\right|^2
       dZ\,dZ'.
\end{align}
\end{definition}
The covariance norm is a norm of the centering $\Pi_p F$, thus a seminorm of $F$. For simplicity, we will sometimes call it a covariance norm of $F.$ It measures the dependence of $F$ on the future
state variables $U_T$ through this covariance kernel. The norm vanishes precisely when $F$ does not depend on
$U_T$, up to null sets, or equivalently when
\[
 F=\E_{U_T}[F\mid U_{\le p},Z_{\le p}]
 \qquad\text{almost surely}.
\]
For the future multioverlaps, this gives
a covariance-based measure of their dependence on the
future state variables.

\subsection{The geometric estimate at one site, and a stronger induction quantity}\label{sec:geometry}

\paragraph{The geometric estimate at one site}
In this section, we consider a future multioverlap in the case $k=1$, i.e. when the shape $T$ corresponds to one path from the root to a leaf. We show that its covariance norms for all $p$ are controlled by the geometry of the cut. For the moment, we will work with the usual $1$-overlap (magnetization) $b$ instead of a shifted positive one, $S=\frac{1+b}{2}$. 

Recall the barycenter representation in \eqref{eq:he}. On a canonical
root-to-leaf path, write
\[
 b=h(U_0,\ldots,U_r;Z_0,\ldots,Z_r),
\]
where we write $0$ instead of $*$ for the root. Fix a depth $p<r$ and separate the retained prefixes from the
future coordinates:
\[
 \begin{aligned}
 U_{\le p}&=(U_0,\ldots,U_p),
 &U_T&=(U_{p+1},\ldots,U_r),\\
 Z_{\le p}&=(Z_0,\ldots,Z_p),
 &Z_T&=(Z_{p+1},\ldots,Z_r).
 \end{aligned}
\]
Define a future $1$-overlap at depth $p$ by 
\begin{equation}\label{eq:gp}
 F(U_{\le p},U_T,Z_{\le p})
   =\E_{Z_T} h(U_{\le p},U_T,Z_{\le p},Z_T).
\end{equation}
The dependence of $F$ on the depth $p$ is suppressed
in the notation. Also, we will write $U$ for $U_T$ to lighten notation, so $F= F(U_{\le p},U,Z_{\le p}).$

Centering in the future state variables gives
\begin{equation}
 \Pi_p F
   =F-\E_U[F\mid U_{\le p},Z_{\le p}].
\end{equation}
Recall the conditional covariance kernel in \eqref{eq:CUZZT},
\begin{align}
 C(Z,Z')
 &=\int
     \Pi_p F(U_{\le p},U,Z)
     \Pi_p F(U_{\le p},U,Z')\,dU.
 \label{eq:CUZZ}
\end{align}
where $Z,Z'$ are two independent copies of the site prefix $Z_{\le p}.$ The dependence of $C$ on $U_{\le p}$ and on $p$ is suppressed in the notation. The retained state prefix $U_{\le p}$ is fixed in this covariance, and both factors use the same future state variables $U$. Thus the kernel records how the magnetizations at two site
prefixes vary together when the future state variables vary.

\begin{lemma}[Geometric variance and covariance bounds]
\label{lem:gram}
For every $p<r$,
\begin{align}
 \E|\Pi_p F|^2
   &\le\delta,
   \label{eq:geomL2}\\
\|\Pi_p F\|_p^4 = \E_{U_{\le p}}\iint
       |C(Z,Z')|^2\,dZ\,dZ'
   &\le\frac{\delta^2}{4}.
   \label{eq:geomcov}
\end{align}
\end{lemma}

The left-hand side of \eqref{eq:geomcov} is the covariance norm defined in \eqref{eq:Xnorm} above. The lemma therefore gives the
one-replica instance of the control that we will seek for
future multioverlaps of arbitrary fixed order, and will provide the base of the induction.

Lemma \ref{lem:gram} is a quantitative version of
\cite[Lemma 1]{rRSB}. There the overlap of two states
separating at depth $p$ takes a single value, so $\delta=0$, and it is proved that $F$ does not depend on the future state variables. Indeed, when $\delta=0$, \eqref{eq:geomL2} gives $\Pi_p F=0$ almost surely; in other words, $F$ does not depend on the future state variables $U$. The same conclusion follows from \eqref{eq:geomcov}, since its left-hand side is the fourth power of the norm of $\Pi_p F$.

\begin{proof}
Let $[a_p,b_p]$ be the closed convex hull of
$I_p\cap\supp\zeta$, so that $b_p-a_p\le\delta$.
Consider two leaves whose last common ancestor has depth $p$.
They share the retained state and site prefixes and have
independent future coordinates $U,U'$ and $Z_T,Z_T'$.
After averaging the site variables $Z_T,Z_T'$, their barycenters have
scalar product
\begin{equation}\label{eq:Gram}
 K(U,U')
   :=\int F(U_{\le p},U,Z)F(U_{\le p},U',Z)\,dZ
   \in[a_p,b_p].
\end{equation}
This relation holds almost everywhere in
$(U_{\le p},U,U')$. Indeed, the scalar products of clusters
separating at depth $p$ lie in this interval
(Section \ref{sec:cutgeometry}). Lemma \ref{lem:A3} preserves their
joint law in the hierarchical representation, where the
scalar product is the displayed kernel evaluated at path
coordinates having the product uniform law.

Fix a retained state prefix $U_{\le p}$ for which \eqref{eq:Gram}
holds almost everywhere in $(U,U')$, and suppress it from
the notation. Write
\[
 F(U,Z)=F(U_{\le p},U,Z),
 \qquad
 \mybar F(Z)=\int F(U,Z)\,dU.
\]
Regard $F(U,\,\cdot\,)$ as a random element of the separable
Hilbert space $L^2(dZ)$, with law $\mu$ and barycenter
$\mybar F$. The bound \eqref{eq:Gram} means that
$a_p\le(\varphi,\psi)\le b_p$
holds for $\mu^{\otimes2}$-almost every pair.
Since the scalar product is continuous, this holds on
$\supp\mu\times\supp\mu$. In particular, taking
$\varphi=\psi$ gives
\[
 \int\|\varphi\|_2^2\,d\mu(\varphi)\le b_p,
 \qquad
 \|\mybar F\|_2^2
   =\iint(\varphi,\psi)\,d\mu(\varphi)\,d\mu(\psi)
   \ge a_p.
\]
Consequently,
\[
 \begin{aligned}
 \E_{U,Z}|\Pi_p F(U,Z)|^2
   &=\int\|\varphi-\mybar F\|_2^2\,d\mu(\varphi)
   =\int\|\varphi\|_2^2\,d\mu(\varphi)
       -\|\mybar F\|_2^2
   \le b_p-a_p\le\delta.
 \end{aligned}
\]
Averaging over $U_{\le p}$ proves \eqref{eq:geomL2}.

For the covariance bound, let
\[
 f(U,Z)=\Pi_p F(U,Z)=F(U,Z)-\mybar F(Z).
\]
Let $\Pi_p^{(1)}$ and $\Pi_p^{(2)}$ denote centering in
$U$ and $U'$, respectively, with the other coordinate
held fixed. Double centering of the Gram kernel gives
\begin{align*}
 K^\circ(U,U')
   &:=(\Pi_p^{(1)}\Pi_p^{(2)}K)(U,U')\\
   &=K(U,U')
      -\int K(V,U')\,dV
      -\int K(U,V)\,dV
      +\iint K(V,V')\,dV\,dV'\\
   &=\int
       \bigl(F(U,Z)-\mybar F(Z)\bigr)
       \bigl(F(U',Z)-\mybar F(Z)\bigr)\,dZ\\
   &=\int f(U,Z)f(U',Z)\,dZ.
\end{align*}
Thus $K^\circ$ is a kernel on the future state coordinates,
whereas
\[
 C(Z,Z')=\int f(U,Z)f(U,Z')\,dU
\]
is a kernel on the site prefixes. Expanding the square and
applying Fubini's theorem identifies their squared
$L^2$ norms:
\begin{align}
 \iint C(Z,Z')^2\,dZ\,dZ'
 &=\int f(U,Z)f(U,Z')f(U',Z)f(U',Z')
                  \,dU\,dU'\,dZ\,dZ' \notag\\
 &=\iint K^\circ(U,U')^2\,dU\,dU'.
 \label{eq:gramfubini}
\end{align}
The two centering operators are commuting orthogonal
projections in $L^2(dU\,dU')$. Their product is therefore
a contraction and annihilates constants. With
$c_p=(a_p+b_p)/2$, it follows that
\[
 \begin{aligned}
 \|K^\circ\|_{L^2(dU\,dU')}^2
   &=\|\Pi_p^{(1)}\Pi_p^{(2)}(K-c_p)\|_{L^2(dU\,dU')}^2\\
   &\le\|K-c_p\|_{L^2(dU\,dU')}^2
   \le\frac{(b_p-a_p)^2}{4}
   \le\frac{\delta^2}{4}.
 \end{aligned}
\]
Together with \eqref{eq:gramfubini}, averaging over
$U_{\le p}$ proves \eqref{eq:geomcov}.
\end{proof}

For the induction we will use the positive magnetizations
$
 S_\alpha=\frac{1+b_\alpha}{2}\in[0,1].
$
For one path, the future site average of $S$ is $(1+F)/2$.
Its centered covariance kernel is
$C/4$, and hence
\begin{equation}\label{eq:S}
 \left\|\Pi_p\frac{1+F}{2}\right\|_p^4
   =\frac1{16}\|\Pi_p F\|_p^4
   \le\frac{\delta^2}{64}.
\end{equation}
This will give the base case of the induction.
The nonnegativity of $S$ will be used in
Section \ref{sec:compactness}.

\paragraph{A stronger induction quantity.}
If we recall the defect $D_k$ defined in \eqref{eq:Dk}, we can similarly define the defect for positive multioverlaps \eqref{eq:MaS} by
\begin{equation}\label{eq:DkS}
 D_k^S :=\E\sum_{\boldsymbol\alpha}V_{\alpha_1}\cdots V_{\alpha_k}
       \bigl(M_{\boldsymbol\alpha}^S
                    -\E_U[M_{\boldsymbol\alpha}^S\mid U_*]\bigr)^2 .
\end{equation}
We will prove in Lemma \ref{lem:base} below that the defects $D_j^S$, $1\le j\le k$, control $D_k$, so we are not losing any information. To control $D_k^S$ and $D_k$, we introduce a stronger induction quantity:
\begin{equation}\label{eq:kappa}
 \kappa_k
 =\sup_{0\le p<r}
     \sum_T\mu_k^p(T)\|\Pi_p F_T^p\|_p^4,
\end{equation}
where, for each depth $p$, the sum is over shapes $T$ with the first split $a(T)\geq p$. It averages the fourth power of the covariance norm of the future multioverlap $F_T^p$, defined in \eqref{eq:Covnorm}. We then take the maximum over all depths.

We will prove $\kappa_k\to0$ by induction on $k$,
assuming $\kappa_j\to0$ for every $j<k$, as the cut
errors $\delta$ and $\eta$ tend to zero. Lemma \ref{lem:base} below shows that these estimates imply $D_k\to0$. In fact, the proof shows that the depth-$0$ terms ($p=0$ in \eqref{eq:kappa}) for $1\le j\le k$ already suffice to control $D_k$, and the main reason we take the supremum $\sup_{0\le p<r}$ over all depths is to make the induction argument work.

\begin{lemma}[Base case and the scalar defect]\label{lem:base}
At each cut,
\begin{equation}\label{eq:base}
 \kappa_1\le\frac{\delta^2}{64},
 \qquad D_k^S\le\sqrt{\kappa_k},
 \qquad D_k\le9^k\max_{1\le j\le k}D_j^S,
\end{equation}
where $D_k^S$ is defined in \eqref{eq:DkS}. Consequently, $\kappa_j\to0$ for $1\le j\le k$ implies $D_k\to0$.
\end{lemma}
\begin{proof}
For one replica, $F_T^p=(1+F)/2$, where $F$ is defined in \eqref{eq:gp}. Thus, by \eqref{eq:S},
\[
 \kappa_1=\frac1{16}\sup_{0\le p<r}\|\Pi_pF\|_p^4
          \le\frac{\delta^2}{64}.
\]
For the second bound, work at $p=0$ and write $f_T=\Pi_0F_T^0$, a function of $(U_0,U_T,Z_0)$ (recall that $Z_0=Z_*$). Integrating the root site variable gives 
\[
 \int f_T\,dZ_0=M^S_{\boldsymbol\alpha}-\E_U[M^S_{\boldsymbol\alpha}\mid U_*],
\]
where $M^S$ is the multioverlap defined in \eqref{eq:MaS}. By Fubini, the Cauchy--Schwarz inequality on the probability space of $(Z_0,Z_0')$, and Jensen's inequality in $U_*$,
\begin{align*} 
&\E_{U_*,U_T}\left(M^S_{\boldsymbol\alpha}-\E_U[M^S_{\boldsymbol\alpha}\mid U_*]\right)^2\\
&=\E_{U_*,U_T}\left(\int f_T(U_*,U_T,Z_0)\,dZ_0\right)^2\\
 &=\E_{U_*}\iint\E_{U_T}f_T(U_*,U_T,Z_0)f_T(U_*,U_T,Z_0')\,dZ_0\,dZ_0'
 \\
 &\leq
 \Bigl(\E_{U_*}\iint \Bigl(\E_{U_T}f_T(U_*,U_T,Z_0)f_T(U_*,U_T,Z_0')\Bigr)^2 \,dZ_0\,dZ_0'\Bigr)^{1/2}\\
 &=\|f_T\|_0^2.
\end{align*}
Average over shapes with $\mu_k^0$ and apply Cauchy--Schwarz once more: 
\[
D_k^S\le\sum_T\mu_k^0(T)\|\Pi_0F_T^0\|_0^2
\le\Bigl(\sum_T\mu_k^0(T)\|\Pi_0F_T^0\|_0^4\Bigr)^{1/2}
\le\sqrt{\kappa_k}.
\]
For the third bound, expand
\[
 \prod_{\ell=1}^k(2S^{\alpha_\ell}-1)=\sum_{J\subseteq\{1,\ldots,k\}}2^{|J|}(-1)^{k-|J|}\prod_{\ell\in J}S^{\alpha_\ell}.
\]
The defect of a linear combination of $2^k$ terms with coefficients $c_J$ is at most $(\sum_J|c_J|)\sum_J|c_J|D_J$, where $D_J$ is the defect of the $J$ term, and $\sum_J|c_J|=3^k$. The marginal law of the labels of a subset $J$ of replicas is the RPC shape law for $|J|$ replicas, so the defect of the $J$ term is $D^S_{|J|}$ (zero for $J=\emptyset$).
\end{proof}

\subsection{Child shapes and shape probabilities}\label{sec:shapes}

\paragraph{The child shapes.}
Consider a shape $T=(a_{\ell\ell'})$ (where we write $a_{\ell\ell'}$ for $\alpha_\ell\wedge\alpha_{\ell'}$), such that its first split $a(T)=\min_{\ell,\ell'}a_{\ell\ell'}\ge p.$
Write $a=a(T)$. If $a<r$, all replicas share a path through
depth $a$ and then, at the next level $a+1$ separate into at least two groups.
These groups are the equivalence classes $D_1,\ldots,D_b$
of $\{1,\ldots,k\}$ under the relation
\begin{equation}\label{eq:childpartition}
 \ell\sim\ell'
 \quad\Longleftrightarrow\quad a_{\ell\ell'}\geq a+1.
\end{equation}
We call these sets the \emph{children at the first split}.
Their \emph{sizes} and \emph{child shapes} are
\begin{equation}\label{eq:childshapes}
 k_i=|D_i|,\qquad
 T_i=(a_{\ell\ell'})_{\ell,\ell'\in D_i},\qquad
 \sum_{i=1}^b k_i=k.
\end{equation}
Each $k_i<k$, which is what will allow us to apply the
induction hypothesis to a child. Within each child, the
replica labels may be relabeled $1,\ldots,k_i$; the depths
keep their original values. A nonsingleton child has its
own first split $a(T_i)\geq a+1$, so its replicas may share a
further stem before separating. For a singleton,
$T_i=(r)$.

We write the child-shape probabilities as $\mu_{k_i}^a(T_i)$.
For a nonsingleton child, $k_i\geq 2$, its own first split satisfies
$a(T_i)>a\ge p$, so Lemma \ref{lem:product}(a) below gives
\[
 \mu_{k_i}^a(T_i)=\mu_{k_i}^p(T_i)=\mu_{k_i}^0(T_i).
\]
For a singleton, all three probabilities equal one.
Thus these shape probabilities are the same under the
original cascade and the auxiliary cascades rooted at either
$p$ or $a$. We use the superscript $a$ because the child
moments in the induction retain the prefix variables through depth $a$.

If $a=r$, all replicas share one terminal label. In this
case, we define the children to be the singleton groups
$D_i=\{i\}$, $i=1,\ldots,k$, with child shapes $T_i=(r)$
and $\mu_1^r((r))=1$ by convention. These children represent
separate replica occurrences of that state. This convention
lets us include repeated terminal labels in the same notation.

\paragraph{The probability bound.}
Recall from \eqref{eq:localweights} that $w_a^{(p)}$ is
the probability that two replicas of the auxiliary cascade
rooted at $p$ have overlap (last-common-ancestor) at depth $a$:
\[
 w_a^{(p)}=
 \begin{cases}
   \zeta_p,&a=p,\\
   \zeta_a-\zeta_{a-1}=w_a,&p<a\le r,
 \end{cases}
 \qquad \zeta_r=1.
\]

\begin{lemma}[Shape probabilities]\label{lem:product}
Given $p<r$, let $T$ be a labeled $k$-replica shape, $k\ge2$, such that $a(T)\geq p$.
\begin{enumerate}
\item[(a)] If the first split $a(T)>p$, then $\mu_k^p(T)=\mu_k^0(T)$.
\item[(b)] Let $a=a(T)$, and let $k_1,\ldots,k_b$ and $T_1,\ldots,T_b$ be the child sizes and shapes defined in \eqref{eq:childshapes}, with the singleton convention above when $a=r$. Then
\begin{equation}\label{eq:rootdom}
 \mu_k^p(T)\le w_a^{(p)}\prod_{i=1}^b\mu_{k_i}^{a}(T_i),
\end{equation}
where $\mu_1^a\equiv1$, also at $a=r$.
\end{enumerate}
\end{lemma}
The proof is in Appendix \ref{app:shapes}. Part (b) bounds the probability of the whole shape by the pair weight at its first split and the smaller child-shape probabilities, with constant one. Part (a) lets us use the original shape law at any retained depth $p$, as long as the first split is deeper than $p$.

Only the first split and its children are needed for branch replacement below.
In Section \ref{sec:compactness}, the physical overlap tree will record
all successive replica groups and let us vary their assigned depths.

\subsection{Branch replacement}\label{sec:branches}

Now we will look at the process of branch replacement. We will show that, under the induction hypothesis, the smaller branches can be averaged out, leaving the common-stem moment.

Consider a shape $T$ such that its first split is $a=a(T)\ge p$ with $b$ children, and let $D_i$, $k_i$ and $T_i$ be the child groups, sizes and shapes from \eqref{eq:childshapes}. Recall the future multioverlap $F_{T_i}^{a}(U_{\le a}, U_{T_i},Z_{\le a})$ of the child shape $T_i$ at depth $a$ in \eqref{eq:FT}, and let
\begin{equation}\label{eq:childfeatures}
 f_i=F_{T_i}^{a}(U_{\le a}, U_{T_i},Z_{\le a}),\qquad
 m_i=\E_{U_{T_i}}f_i,\qquad e_i=f_i-m_i.
\end{equation}
In \(m_i\), only the future child state variables \(U_{T_i}\) are averaged; the shared prefixes \(U_{\le a}\) and \(Z_{\le a}\) are held fixed.

The child state variables $U_{T_i}$ are independent. We can express the future multioverlap $F_T^p$ of the shape $T$ at depth $p$ in terms of its children by averaging in $Z_{(p,a]}$ in their common stem $(p,a]:=(p+1,\ldots, a)$:
\[
 F_T^p
       =\E_{Z_{(p,a]}}\prod_{i=1}^b f_i.
\]
If we replace each $f_i$ by its conditional mean $m_i$, the result is the \emph{common-stem moment}
\begin{equation}\label{eq:HT}
  H_T^p 
        =\E_{Z_{(p,a]}}\prod_{i=1}^b m_i.
\end{equation}
Note that $F_T^p$ is a function of $(U_{\le p}, U_{T},Z_{\le p})$, which is the same as $(U_{\le a}, (U_{T_i}),Z_{\le p})$, and $H_T^p$ is a function of $(U_{\le a}, Z_{\le p}).$ Consequently
\begin{equation}\label{eq:stemproperties}
 \Pi_pH_T^p=0\quad\text{if }a(T)=p .
\end{equation}
The closeness $f_i\approx m_i$ is controlled by the induction quantity $\kappa_m$ for $m<k$ and if, by induction hypothesis, these are small then indeed $f_i\approx m_i.$  The following lemma makes this statement quantitative, which will be an input to prove that this implies that $\kappa_k$ is then also small.

\begin{lemma}[Branch replacement]\label{lem:branch}
Let $k\ge2$. At every cut,
\begin{equation}\label{eq:branchbound}
 \sup_{p<r}\sum_T\mu_k^p(T)
        \|F_T^p-H_T^p\|_p^4
       \le k^4B_k\max_{1\le m<k}\kappa_m ,
\end{equation}
where $B_k$ is the number of partitions of a set of $k$ elements. The difference \(F_T^p-H_T^p\) has zero mean when all child state variables are averaged, with the common stem and retained site prefix fixed,
\[
\mathbb E_{U_{T_1},\ldots,U_{T_b}}
   [F_T^p-H_T^p\mid U_{\le a},Z_{\le p}]=0.
\]
\end{lemma}
\begin{proof}
Let us simplify notation for state and site variables:
\begin{align*}
&u=U_{\leq p},\quad U=U_{(p,a]},\quad U_i=U_{T_i};\\
& z=Z_{\le p},\quad Z=Z_{(p,a]}.
\end{align*}
All functions depend on $u=U_{\leq p}$, but we will suppress this dependence until the very end when we integrate over it. Then,
$$
f_i = f_i(U,U_i;z,Z),\qquad m_i = m_i(U;z,Z)=\E_{U_i} f_i,\qquad e_i = f_i-m_i.
$$
Let us telescope the difference replacing one child at a time,
$$ 
F_T^p -  H_T^p = \E_Z \Bigl(\prod_{i=1}^b f_i - \prod_{i=1}^b m_i\Bigr)
=\sum_{i\leq b} R_i.
$$
One term of the telescoping sum has the form
\[
 R_i(U,U_1,\ldots,U_b;z)
     =\E_Z\, e_i(U,U_i;z,Z)K(U,U_{-i};z,Z),\qquad |K|\le1,
\]
where $K$ is a product of some $f_j$ and some $m_j$, $j\ne i$, and $U_{-i}=(U_j)_{j\neq i}$. For every fixed $U,U_{-i},z$ we have $\E_{U_i}R_i=0$, because $\E_{U_i}e_i=0$ and $K$ does not depend on $U_i$. Hence $R_i$ is centered in the state coordinates $U_T = (U,(U_i)_{1\leq i\leq b})$ of the shape $T$ above depth $p$, so centering $\Pi_p R_i=R_i,$ and the covariance kernel defined in \eqref{eq:CUZZT} in the case of $R_i$ is:
\begin{align}
&C_R(z,z'):= \int
           R_i(u,U_T,z)
           R_i(u,U_T,z')\,dU_T,
\end{align}
For two independent site prefixes $z,z'$ define
\begin{align*}
 C_i(U;z,Z;z',Z')
      &=\E_{U_i}e_i(U,U_i;z,Z)e_i(U,U_i;z',Z'),\\
 \mathcal K(U;z,Z;z',Z')
      &=\E_{U_{-i}}K(U,U_{-i};z,Z)K(U,U_{-i};z',Z').
\end{align*}
Then $|\mathcal K|\le1$, and since $U_i$ and $U_{-i}$ are independent given $U$, the covariance kernel of $R_i$ (at depth $p$) is exactly
\[
 C_R(z,z')
   =\E_{U,Z,Z'}\, C_i(U;z,Z;z',Z')\,\mathcal K(U;z,Z;z',Z').
\]
The state stem $U$ is shared by the two factors, while the site stems $Z,Z'$ are independent because they come from two separate site averages in $R_i(z)R_i(z')$. Jensen's inequality gives
\[
 |C_R(z,z')|^2
     \le\E_{U,Z,Z'}|C_i(U;z,Z;z',Z')|^2 .
\]
Recalling that everything depends on $u=U_{\leq p}$, and integrating over $u,z,z'$ turns $(u,U)$ into the original law of $U_{\le a}$ and the concatenations $Z_{\leq a}=(z,Z), Z_{\leq a}'=(z',Z')$ into two independent complete site prefixes through level $a$. Thus
\begin{equation}\label{eq:onebranch}
 \|\Pi_p R_i\|_p^4= \|R_i\|_p^4
       \le\E_{U_{\le a},Z_{\le a},Z'_{\le a}}|C_i|^2
       =\|\Pi_aF_{T_i}^{a}\|_a^4 .
\end{equation}
This is the key estimate: the right side is the covariance norm of the smaller child moment at retained depth $a$, the depth at which it separated from the other children. No length of the stem appears.

For a fixed shape, the triangle inequality for the norm and $(\sum_{i\le b}x_i)^4\le b^3\sum_ix_i^4$ give that
$\|F_T^p-H_T^p\|_p^4\le b^3\sum_i\|R_i\|_p^4$. 
Thus the left hand side of \eqref{eq:branchbound} can be bounded by 
$$
\sum_T\mu_k^p(T) \|F_T^p-H_T^p\|_p^4
\leq
\sum_T\mu_k^p(T) b^3\sum_{i\leq b} \|\Pi_aF_{T_i}^{a}\|_a^4,
$$ 
where $b,a$ and $T_i$ depend on $T$. We can first consider shapes $T$ with a fixed labeled partition of the $k$ replicas into $b$ children of sizes $k_1,\ldots, k_b$, and bound $\sum_T\mu_k^p(T) \|\Pi_aF_{T_i}^{a}\|_a^4$ over such shapes for each $i\leq b.$ We can sum this over the first split $a$ and the child shapes using \eqref{eq:rootdom}:
\begin{align*}
 &\sum_{a=p}^{r-1} w_a^{(p)}
    \sum_{T_1,\ldots,T_b}
       \Bigl(\prod_{j=1}^b\mu_{k_j}^{a}(T_j)\Bigr)
            \|\Pi_aF_{T_i}^{a}\|_a^4
 \le\sum_{a=p}^{r-1}w_a^{(p)}\sum_{T_i}\mu_{k_i}^a(T_i)\|\Pi_aF_{T_i}^a\|_a^4
 \le\sum_{a=p}^{r-1}w_a^{(p)}\kappa_{k_i}\le\kappa_{k_i}
\end{align*}
(for a genuine child, its first split must exceed \(a\), unless it is a singleton; in the upper bound above, we enlarge each child-shape sum to all shapes of the auxiliary cascade rooted at \(a\), which adds only nonnegative terms). At \(a=r\), there are no state variables beyond the first split, so \(f_i=m_i\) and all branch replacement errors vanish. Summing over the $b\le k$ children gives $b^4\max_{m<k}\kappa_m\le k^4\max_{m<k}\kappa_m$ for each labeled partition, and there are at most $B_k$ labeled partitions.
\end{proof}

\subsection{What remains to be proved}\label{sec:target}
The replacement leaves only the state variables on the common path through the first split. If the first split is already at $p$, there is no future common stem and $\Pi_pH_T^p=0$. We regard these common-stem moments as functions of the same canonical state path, using its prefix through \(a(T)\) for each shape.
To complete the induction it remains to show
\begin{equation}\label{eq:inductiontarget}
 \sup_{p<r}\sum_T\mu_k^p(T)\|\Pi_pH_T^p\|_p^4\longrightarrow0
 \quad\text{assuming }\kappa_m\longrightarrow0\ (m<k).
\end{equation}
Indeed, the triangle inequality and Lemma \ref{lem:branch} will give
\[
 \kappa_k\le
 8\sup_{p<r}\sum_T\mu_k^p(T)\|\Pi_pH_T^p\|_p^4
       +8k^4B_k\max_{m<k}\kappa_m
       \longrightarrow0.
\]
The cavity equations will first give \emph{tested averages} of these centered moments: for every continuous permutation-invariant function $\psi$ of the quantized overlap matrix, $\sum_T\mu_k^p(T)\psi(q(T))\Pi_pH_T^p$ is small in norm (Proposition \ref{prop:tests}). At a fixed finite depth, these tests separate the finitely many shapes up to permutation of the replica labels; this suffices because the common-stem moments agree under such relabelings. When the depth grows, the number of shapes grows too, and one needs the compactness argument of Section \ref{sec:compactness} to pass from tested averages to \eqref{eq:inductiontarget}.

\section{Cavity equations for cluster barycenters}\label{sec:marking}

In this section, we will write the cavity equations for the cluster barycenters and transfer the geometric covariance bound of Section \ref{sec:geometry} to two cavity sites. The argument follows \cite[Theorem 6]{rRSB}. The difference is that a terminal cluster now contains a small neighborhood of states, so replacing a sampled state by its barycenter introduces an error of order $\sqrt\eta$. We will keep track of this error before removing the average over cluster labels.

\paragraph{Inserting overlap tests.} We first record the cavity equation \eqref{eq:SC} with a bounded function of the overlap matrix inserted. Define
\[
 G^\Lambda(dm)=\frac{\Lambda(m)}{\cZ}G(dm),\qquad
 \widetilde m_i(m)=
 \begin{cases}\xi_i(m),&i\le s,\\m(v_i),&i>s.\end{cases}
\]
For any bounded Borel function $\psi$ of the $n$-replica overlap matrix,
\begin{align}
 &\E\int\psi((m^\ell\cdot m^{\ell'})_{\ell,\ell'\le n})
       \prod_{\ell=1}^n\prod_{i\in C_\ell}m^\ell(v_i)
                        \,G^{\otimes n}(d\boldsymbol m)\notag\\
 &=\E\int\psi((m^\ell\cdot m^{\ell'})_{\ell,\ell'\le n})
       \prod_{\ell=1}^n\prod_{i\in C_\ell}\widetilde m_i(m^\ell)
                         \, (G^\Lambda)^{\otimes n}(d\boldsymbol m).
                                                        \label{eq:SCoverlap}
\end{align}
This follows from \eqref{eq:SC}: add independent noncavity sites; products of spins at these sites give every polynomial in the off-diagonal overlaps, and the diagonal is $q_*$; polynomial approximation gives continuous tests on the compact matrix cube; equality of the resulting finite signed measures gives all bounded Borel tests. Equivalently, the empirical overlaps computed from the added noncavity sites are the same on both sides. Note that the test $\psi$ on the right side refers to the overlaps of the original states, that is to the cluster labels of the states, and not to overlaps computed from the cavity outputs.

\subsection{The cavity fields of the cluster barycenters}
Fix a cut. Recall from \eqref{eq:barycenter} that $C_\alpha$ is a terminal cluster (an approximate pure state), and 
$$
V_\alpha=G(C_\alpha), \qquad
G_\alpha=G(\,\cdot\mid C_\alpha),\qquad
 b_\alpha=\int m\,G_\alpha(dm).
 $$ 
 We want to write the cavity equations \eqref{eq:SC} in terms of these weights and barycenters. For this, write $G=\sum_\alpha V_\alpha G_\alpha$ on both sides of the cavity equation and replace a state sampled from $G_\alpha$ by $b_\alpha$. This replacement is made both in the observed spin products and in every input $m(v_I)$ of the cavity fields. It is controlled by
\[
 \iint |m(v)-b_\alpha(v)|\,dv\,G_\alpha(dm)
       \le\sqrt\eta \qquad\text{for every }\alpha,
\]
from \eqref{eq:baryL1}. The resulting cavity equation will therefore be approximate; we will display it below and prove its error bound in Section \ref{sec:transfer}.

To write this equation, first evaluate the cavity fields \eqref{eq:cavityfield} at the represented barycenters \eqref{eq:he}. For $s$ added cavity sites, set
\begin{align}
 A_i^\alpha(\varsigma)
  ={}&\sum_{a=1}^{N_{i,0}}\theta_{i,a}
       (b_\alpha^{i,0,a,1},\ldots,b_\alpha^{i,0,a,K-1},\varsigma)
       +g_i\frac{1+\varsigma}{2}\notag\\
 &+\sum_{d\ge2}\sum_{a=1}^{N_{i,d}}\theta_{i,d,a}^{\epsilon}
       (b_\alpha^{i,d,a,1},\ldots,b_\alpha^{i,d,a,d-1},\varsigma),
       \qquad i\le s.\label{eq:barycavityfield}
\end{align}
Here $N_{i,0}\sim\operatorname{Pois}(\lambda K)$ and $N_{i,d}\sim\operatorname{Pois}(d)$, as before. The tuple $I=(i,d,a,j)$ is  a site index, and
\[
 b_\alpha^I
   =h\bigl((U_\beta)_{\beta\in p(\alpha)},
           (Z_\beta^I)_{\beta\in p(\alpha)}\bigr).
\]
The arrays of site variables $(Z_\beta^I)_\beta$ are independent for different $I$, while the state variables $U_\beta$ are shared by all sites. Let $\mathcal I_i$ be the set of site indices $I$ that appear in the cavity field $A_i^\alpha(\varsigma)$ of the cavity site $i$, and 
\begin{equation}\label{eq:cavIs}
\mathcal I=\bigcup_{i\le s}\mathcal I_i.
\end{equation}
Define the cavity magnetizations and multipliers as in \eqref{eq:outputs} by
\begin{equation}\label{eq:barycavityoutputs}
 A_i^\alpha=\log\Av_\varsigma e^{A_i^\alpha(\varsigma)},\qquad
 \xi_i^\alpha=\frac{\Av_\varsigma\varsigma e^{A_i^\alpha(\varsigma)}}{e^{A_i^\alpha}},
 \qquad \Lambda_\alpha=e^{\sum_{i=1}^s A_i^\alpha}.
\end{equation}
Thus $\xi_i^\alpha$ is the magnetization of the new cavity spin when the clause arguments are taken from $b_\alpha$. In the right side of the cavity equation, replacing the states by their barycenters changes the normalizing factor $\langle\Lambda(m)\rangle$ to $\sum_\beta V_\beta\Lambda_\beta$. The normalized weights appearing there are consequently
\begin{equation}\label{eq:reweighted}
 \widetilde V_\alpha
      =\frac{V_\alpha\Lambda_\alpha}{\sum_\beta V_\beta\Lambda_\beta}.
\end{equation}
For $n$ replicas and sets of sites $J_\ell\subseteq\{1,\ldots,s\}$ for $\ell\leq n$, the approximate cavity equation we will prove is
\begin{align}
 &\Biggl|
 \E\prod_{\ell=1}^n
       \left(\sum_\alpha V_\alpha\prod_{i\in J_\ell}b_\alpha^i\right)
-
 \E\prod_{\ell=1}^n
       \left(\sum_\alpha\widetilde V_\alpha
                         \prod_{i\in J_\ell}\xi_i^\alpha\right)
 \Biggr|
       \le C_{n,s}\sqrt\eta .
       \label{eq:barycavityequation}
\end{align}
The finite constant $C_{n,s}$ depends on $n,s$ and the clause laws, but not on the cut; its value is given in \eqref{eq:Cns}. The expectations include all the cavity disorder and uniform variables, as in the original cavity equation. If $\eta=0$, the clusters are pure states and \eqref{eq:barycavityequation} is the exact identity in \cite[equation (82)]{rRSB}. On the left, the weights $V$ are independent of the barycenter array. On the right, $\widetilde V$ and $\xi$ depend on the same cavity fields. The RPC change of density property can be used to express this right-hand side using an RPC weight array independent of a changed array of cavity magnetizations. We define that array next.

\subsection{The hierarchical change of density}\label{sec:changeofmeasure}
At a vertex $\beta$, denote all the state and site variables appearing in the $s$ cavity fields by the tuple below, with $\mathcal I$ as in \eqref{eq:cavIs}:
\[
 \omega_\beta=\bigl(U_\beta,(Z_\beta^I)_{I\in\mathcal I}\bigr).
\]
Let $\mathcal F_{\mathrm{cav}}$ be the sigma-algebra generated by the Poisson counts, clause disorder and unary fields of these cavity sites:
\begin{equation}
 \mathcal F_{\mathrm{cav}}
 =\sigma\bigl(N_{i,0},N_{i,d},\theta_{i,a},
                 \theta_{i,d,a}^{\epsilon},g_i:
                 i\le s,\ d\ge2,\ a\ge1\bigr).
 \label{eq:Fcav}
\end{equation}
It contains none of the uniform state or site variables. We carry out the RPC construction conditionally on
$$
 \mathcal F=\mathcal F_{\mathrm{cav}}\vee\sigma(\omega_*),
$$
so both the cavity disorder and the root variables are fixed. Along a path $\beta_0=*,\ldots,\beta_r=\alpha$, abbreviate $\omega_{\beta_j}$ by $\omega_j$ and set
\[
 X_r(\omega_0,\ldots,\omega_r)
       =\log\Lambda_\alpha=\sum_{i=1}^s A_i^\alpha.
\]
Recursively, from the leaf to the root let us define
\begin{equation}\label{eq:Xrec}
 X_j=\frac1{\zeta_j}\log\E_{\omega_{j+1}}e^{\zeta_jX_{j+1}},
 \qquad
 \rho_j=e^{\zeta_j(X_{j+1}-X_j)},\qquad 0\le j<r.
\end{equation}
The expectation $\E_{\omega_{j+1}}$ integrates the original
product uniform law, with the cavity disorder and
$\omega_0,\ldots,\omega_j$ fixed. Since
$\E_{\omega_{j+1}}\rho_j=1$, the function $\rho_j$ is a
conditional density for the variable at depth $j+1$
relative to this law.

Conditionally on $ \mathcal F,$ generate a new array
\[
\widetilde \omega_\beta=\bigl(\widetilde U_\beta,(\widetilde Z_\beta^I)_{I\in\mathcal I}\bigr)
\]
down the tree from the root to the leaves by using \(\rho_j((\widetilde\omega_\gamma)_{\gamma\prec\beta},\,\cdot\,)\) as the  density of $\widetilde \omega_\beta$ for $\beta$ at depth $j+1$ conditionally on $(\widetilde \omega_\gamma)_{\gamma \prec \beta}$, independently for different children given their ancestors. The root retains its original value. Let $\widetilde\xi_i^\alpha$ be the function in \eqref{eq:barycavityoutputs} evaluated on this new array. Although the coordinates of $\omega_\beta$ were independent, those of $\widetilde\omega_\beta$ need not be.

We use the hierarchical version of the RPC invariance property of Bolthausen and Sznitman \cite{Bolthausen}; see \cite[Theorem 4.4]{SKbook}. In the sorted-weight formulation of \cite[Theorem 5]{rRSB}, there is a random tree relabeling $\pi$ preserving ancestors such that, conditionally on $\mathcal F$,
\[
 \bigl((\widetilde V_{\pi(\alpha)})_\alpha,
       (\xi_i^{\pi(\alpha)})_{i,\alpha}\bigr)
 \stackrel{d}{=}
 \bigl((V_\alpha)_\alpha,(\widetilde\xi_i^\alpha)_{i,\alpha}\bigr)
\]
and the two arrays on the right are independent, also given \(\mathcal F\). As a result, the second term in \eqref{eq:barycavityequation} can be rewritten as
\begin{equation}\label{eq:barycavityequation2}
 \E\prod_{\ell=1}^n
       \left(\sum_\alpha\widetilde V_\alpha
                         \prod_{i\in J_\ell}\xi_i^\alpha\right)
  =
   \E\prod_{\ell=1}^n
       \left(\sum_\alpha V_\alpha
                         \prod_{i\in J_\ell}\widetilde\xi_i^\alpha\right).                       
\end{equation}
For later calculations, the conditional density of a single path segment is
\begin{equation}\label{eq:Rsegment}
 \rho_{p:q}=\prod_{j=p}^{q-1}\rho_j,\qquad
 \E^X_{p\to q}f
   =\int f\,\rho_{p:q}\,d\omega_{p+1}\cdots d\omega_q,
       \qquad p\le q.
\end{equation}
The prefix through $p$ is held fixed; $\rho_{p:p}=1$. We will write $\Cov^X_{p\to q}$ for covariance under this probability law.

\subsection{The cavity identity at a prescribed shape}\label{sec:transfer}

Our next goal is to prove the approximate cavity equation \eqref{eq:barycavityequation} and obtain a version for prescribed cluster labels, without averaging over $\alpha$. For the latter, we start from the exact equation \eqref{eq:SCoverlap} with an overlap test restricting the label tuples to a fixed shape. We then make the barycenter replacement and use the change of density as in \eqref{eq:barycavityequation2}. Independence will let us factor out the shape probability.

For $n$ replicas, choose sets of cavity sites $J_\ell\subseteq\{1,\ldots,s\}$ and a fixed tuple of labels $\boldsymbol\alpha$ of shape $T$. The moment we want to transfer is
\begin{equation}\label{eq:testedmonomial}
 \mathcal O_T(x)
     =\prod_{\ell=1}^n\prod_{i\in J_\ell}x_i^{\alpha_\ell}.
\end{equation}
We compare $\E\mathcal O_T(b)$ with $\E\mathcal O_T(\widetilde\xi)$. Both expectations include all uniform variables; the second also includes the cavity disorder. 

\paragraph{Bounds for replacing a state by its barycenter.}\label{sec:integrability}
To estimate the replacement in \eqref{eq:barycavityfield}, regard its arguments as a list $x=(x_I)_{I\in\mathcal I}$ of numbers in $[-1,1]$, where $\mathcal I$ was defined in \eqref{eq:cavIs}. For a clause $c$ in any of the $s$ cavity fields, let $d_c$ be its arity and $J_c$ its sup norm over $\{-1,1\}^{d_c}$. The following random bounds are $\mathcal F_{\mathrm{cav}}$-measurable:
\[
 \mathsf B=\sum_{i=1}^s|g_i|+\sum_cJ_c,\qquad
 \mathsf L=\sum_c(d_c-1)\frac{e^{2J_c}-1}{2},\qquad
 \mathsf K=\max_{i\le s}\sum_{c\text{ at }i}(d_c-1)\frac{e^{2J_c}-1}{2}.
\]
Here $\mathsf B$ bounds $|\log\Lambda|$ as well as each $|A_i(m,\varsigma)|$ in \eqref{eq:cavityfield}, $\mathsf L$ bounds the sum of the coordinate Lipschitz constants of $\log\Lambda$, and $\mathsf K\le \mathsf L$ does the same for any $\xi_i$. More explicitly, for suitable nonnegative $a_I,c_{iI}$,
\begin{align}
 |\log\Lambda(x)|&\le \mathsf B,\notag\\
 |\log\Lambda(x)-\log\Lambda(y)|
     &\le\sum_Ia_I|x_I-y_I|,\qquad \sum_Ia_I=\mathsf L,\label{eq:Lipschitz}\\
 |\xi_i(x)-\xi_i(y)|
     &\le\sum_Ic_{iI}|x_I-y_I|,\qquad
                         \sum_Ic_{iI}\le \mathsf K.\notag
\end{align}
Indeed, differentiating the affine Boltzmann interpolation \eqref{eq:extension} gives $|\partial_{x_I}\theta|\le(e^{2J_c}-1)/2$ for each input of clause $c$. The derivative of a log sum of exponentials is a weighted average of these derivatives. For the output, use
\[
 \xi_i(x)=\tanh\frac{A_i(x,1)-A_i(x,-1)}2,
 \qquad
 |\partial_{x_I}\xi_i|
   \le\frac{|\partial_{x_I}A_i(x,1)|+|\partial_{x_I}A_i(x,-1)|}{2}.
\]
The constant in the transfer estimate will be
\begin{equation}\label{eq:Cns}
 C_{n,s}=\E\left[
 ns+n e^{2n\mathsf B}
       \{s\mathsf K+\mathsf L(1+e^{2\mathsf B})\}\right].
\end{equation}
We assume it is finite for the $n,s$ under consideration; the application below requires only $C_{4,2}$. We verify finiteness for our models after the proof. The following is the analogue of  \cite[Theorem 6]{rRSB}.

\begin{proposition}[Shape-uniform cavity transfer]\label{prop:transfer}
For every monomial \eqref{eq:testedmonomial} at a shape $T$ of positive probability,
\begin{equation}\label{eq:transfer}
 \left|\E\mathcal O_T(b)-\E\mathcal O_T(\widetilde\xi)\right|
                    \le C_{n,s}\sqrt\eta.
\end{equation}
The constant is independent of the shape probability, the cascade parameters and the depth.
\end{proposition}
\begin{proof}
We begin with the exact cavity equation, before replacing
the sampled states by their barycenters. Thus $\Lambda(m)$
and $\xi_i(m)$ are evaluated at a state $m$, whereas
$\Lambda_\alpha=\Lambda(b_\alpha)$ and
$\xi_i^\alpha=\xi_i(b_\alpha)$ are their barycenter versions.
Write $\mathcal A_T=\{\boldsymbol\alpha:T(\boldsymbol\alpha)=T\}$ and
\[
 p_T(V)=\sum_{\boldsymbol\alpha\in\mathcal A_T}
                      \prod_{\ell=1}^nV_{\alpha_\ell},
 \qquad p_T=\E p_T(V)>0.
\]
For states $\boldsymbol m=(m^1,\ldots,m^n)$, put
\[
 O(\boldsymbol m)=\prod_{\ell=1}^n\prod_{i\in J_\ell}m^\ell(v_i),
 \qquad
 O_\xi(\boldsymbol m)=\prod_{\ell=1}^n\prod_{i\in J_\ell}\xi_i(m^\ell),
 \qquad
 G_{\boldsymbol\alpha}=\bigotimes_{\ell=1}^nG_{\alpha_\ell}.
\]
The shape test specifies the bins containing the overlaps of the sampled states. Thus the cavity equation in the form \eqref{eq:SCoverlap} is exactly
\begin{align}
 &\E\sum_{\boldsymbol\alpha\in\mathcal A_T}
       \Bigl(\prod_{\ell=1}^nV_{\alpha_\ell}\Bigr)
                  \int O(\boldsymbol m)\,G_{\boldsymbol\alpha}(d\boldsymbol m)
       \notag\\
 &=\E\sum_{\boldsymbol\alpha\in\mathcal A_T}
       \Bigl(\prod_{\ell=1}^nV_{\alpha_\ell}\Bigr)
       \int\frac{\prod_{\ell=1}^n\Lambda(m^\ell)}{\cZ^n}\,
                       O_\xi(\boldsymbol m)\,G_{\boldsymbol\alpha}(d\boldsymbol m),
       \label{eq:clusterSC}
\end{align}
where $\cZ=\sum_\alpha V_\alpha\int\Lambda(m)G_\alpha(dm)$. The auxiliary noncavity sites used to obtain this test have already been averaged. No approximation has yet been made. In particular, the states are sampled independently from their respective $G_{\alpha_\ell}$ even when some labels repeat.

On the left, replace $m^\ell(v_i)$ by $b_{\alpha_\ell}(v_i)$. Since every factor is bounded by one, the product error is bounded by the sum of the at most $ns$ factor errors. By \eqref{eq:baryL1}, its site and state average is at most $ns\sqrt\eta$ for every label tuple. The total error is therefore at most $ns\sqrt\eta\,p_T$.

On the right, both the outputs and the reweighting must be replaced. Define
\[
 \overline{\cZ}=\sum_\alpha V_\alpha\Lambda_\alpha,\qquad
 w=\frac{\prod_\ell\Lambda(m^\ell)}{\cZ^n},\qquad
 \bar w=\frac{\prod_\ell\Lambda_{\alpha_\ell}}{\overline{\cZ}^{n}}.
\]
Fix $G$, the cavity disorder and $\boldsymbol\alpha$. In the next estimates $\E_{v,\boldsymbol m}$ averages the site coordinates and $G_{\boldsymbol\alpha}$, with these fixed. By \eqref{eq:baryL1} and \eqref{eq:Lipschitz},
\begin{align*}
 \E_{v,\boldsymbol m}|\log\Lambda(m^\ell)-\log\Lambda_{\alpha_\ell}|
      &\le \mathsf L\sqrt\eta,\\
 \E_v|\cZ-\overline{\cZ}|&\le e^{\mathsf B}\mathsf L\sqrt\eta,\\
 \E_v|\log\cZ-\log\overline{\cZ}|&\le e^{2\mathsf B}\mathsf L\sqrt\eta.
\end{align*}
The last line uses $\cZ,\overline{\cZ}\ge e^{-\mathsf B}$. Consequently
\[
 w,\bar w\le e^{2n\mathsf B},\qquad
 \E_{v,\boldsymbol m}|\log w-\log\bar w|
     \le n\mathsf L(1+e^{2\mathsf B})\sqrt\eta.
\]
For the output product, the same telescoping gives
\[
 \E_{v,\boldsymbol m}
   |O_\xi(\boldsymbol m)-O_\xi(\boldsymbol b)|
     \le ns\mathsf K\sqrt\eta,
 \qquad \boldsymbol b=(b_{\alpha_1},\ldots,b_{\alpha_n}).
\]
Use $|w-\bar w|\le e^{2n\mathsf B}|\log w-\log\bar w|$ and $|O_\xi|\le1$. For each label tuple, the error in the weighted product is at most
\[
 n e^{2n\mathsf B}
       \{s\mathsf K+\mathsf L(1+e^{2\mathsf B})\}\sqrt\eta.
\]
Summing the constrained weights and averaging the cavity disorder therefore turns \eqref{eq:clusterSC} into
\begin{align*}
 \left|\E\sum_{\boldsymbol\alpha\in\mathcal A_T}
       \Bigl(\prod_\ell V_{\alpha_\ell}\Bigr)\mathcal O_{\boldsymbol\alpha}(b)
 -\E\sum_{\boldsymbol\alpha\in\mathcal A_T}
       \Bigl(\prod_\ell\widetilde V_{\alpha_\ell}\Bigr)
                       \mathcal O_{\boldsymbol\alpha}(\xi)\right|
 \le C_{n,s}\sqrt\eta\,p_T.
\end{align*}
Here $\mathcal O_{\boldsymbol\alpha}$ denotes the product in \eqref{eq:testedmonomial} with the displayed labels. Summing this estimate over all shapes gives \eqref{eq:barycavityequation}, since $\sum_Tp_T=1$. We now remove the weights for a fixed shape.

By Lemma \ref{lem:A3}, $V$ is independent of the original barycenter array, whose moments depend on $\boldsymbol\alpha$ only through $T$. For the other term, use the relabeling $\pi$ from Section \ref{sec:changeofmeasure}; it preserves $\mathcal A_T$ and makes the changed weights independent of the changed array. Thus the two sums are, respectively,
\[
 p_T\,\E\mathcal O_T(b)
 \quad\text{and}\quad
 \E\sum_{\boldsymbol\alpha\in\mathcal A_T}
       \Bigl(\prod_\ell V_{\alpha_\ell}\Bigr)
                   \mathcal O_{\boldsymbol\alpha}(\widetilde\xi)
       =p_T\,\E\mathcal O_T(\widetilde\xi).
\]
Canceling $p_T$ proves \eqref{eq:transfer}. 
\end{proof}

\paragraph{Finiteness of the replacement constant.}
Recall $J_d=\max_\sigma|\theta_d^\epsilon(\sigma)|$ for a REM clause of arity $d$. Its Gaussian entries have standard deviation $\sqrt\epsilon\,2^{-2d}$. Gaussian maximum bounds give, for every fixed $t>0$,
\begin{equation}\label{eq:expsums}
 \sum_{d\ge2}d(\E e^{tJ_d}-1)<\infty,\qquad
 \sum_{d\ge2}d(d-1)\E[(e^{2J_d}-1)e^{tJ_d}]<\infty.
\end{equation}
Indeed, the summands are bounded by a polynomial in $d$ times $2^{-2d}$. The Poisson exponential formula applied to $\mathsf B$, and the same formula with one clause contribution inserted, now give
\[
 \E e^{t\mathsf B}<\infty,\qquad
 \E[\mathsf L e^{t\mathsf B}]<\infty.
\]
Since $\mathsf K\le \mathsf L$, these are precisely the bounds needed in \eqref{eq:Cns}. The unary Gaussians preserve them. For the original clauses it suffices that $\E e^{tJ_0}<\infty$ for all $t>0$, where $J_0=\max_\sigma|\theta(\sigma)|$; this holds in both \eqref{eq:pspin} and \eqref{eq:ksat}. The countable sums in the proof can equivalently be handled first with finitely many clause arities and then by these bounds.

\subsection{The covariance of two cavity outputs}\label{sec:twosite}

We now use Proposition \ref{prop:transfer} to obtain an analogue of the geometric covariance bound \eqref{eq:geomcov} for the cavity outputs. Take $s=2$. For each terminal label $\alpha$, the two cavity sites give the fields $A_1^\alpha,A_2^\alpha$ and magnetizations $\xi_1^\alpha,\xi_2^\alpha$ of \eqref{eq:barycavityoutputs}. Both sites are evaluated in the same state $\alpha$: they use the same state variables along its path, and their respective site variables indexed by $\mathcal I_1$ and $\mathcal I_2$.

Along the path $\beta_0=*,\ldots,\beta_r=\alpha$ to $\alpha$,
collect the joint variables
\[
 \omega_{\beta_j}
   =\bigl(U_{\beta_j},
          (Z_{\beta_j}^I)_{I\in\mathcal I}\bigr),
 \qquad
 \mathcal I=\mathcal I_1\cup\mathcal I_2.
\]
We will again abbreviate $\omega_{\beta_j}$ by $\omega_j$
and write
\[
 X_r(\omega_0,\ldots,\omega_r)=A_1^\alpha+A_2^\alpha,
 \qquad
 \xi_i(\omega_0,\ldots,\omega_r)=\xi_i^\alpha,
 \quad i=1,2.
\]
Each $\xi_i$ uses the common state variables and its own
site variables, indexed by $\mathcal I_i$.
Use this $X_r$ in the recursion \eqref{eq:Xrec}. Fix $p<r$ and consider the descendant paths of
$\beta_p\in\N^p$. Hold the cavity disorder and the joint prefix $\omega_{\le p}=(\omega_0,\ldots,\omega_p)$ on the path to $\beta_p$ fixed. The law with density $\rho_{p:r}$ from \eqref{eq:Rsegment} generates the remaining variables along one path from $\beta_p$ to a terminal label $\alpha$. In particular, it averages the future state variables and the future site variables of both cavity fields together.

Under this law, consider the conditional covariance
\begin{equation}\label{eq:Cp}
 C_p(\omega_{\le p})=\Cov^X_{p\to r}(\xi_1,\xi_2)
     =\E^X_{p\to r}(\xi_1\xi_2)
       -(\E^X_{p\to r}\xi_1)(\E^X_{p\to r}\xi_2).
\end{equation}
For example, writing $\omega_{>p}=(\omega_{p+1},\ldots,\omega_r)$,
\[
 \E^X_{p\to r}(\xi_1\xi_2)
 =\int \xi_1(\omega_{\le p},\omega_{>p})
        \xi_2(\omega_{\le p},\omega_{>p})
        \rho_{p:r}(\omega_{\le p},\omega_{>p})\,d\omega_{>p}.
\]

We want to bound the average of $(C_p)^2$ under the changed prefix law, whose density is $\rho_{0:p}$. In the proof below, four independent continuations of this prefix will express the squared covariance through moments to which Proposition \ref{prop:transfer} applies. The following is the analogue of  \cite[Lemma 2]{rRSB}.

\begin{proposition}[Two-site covariance bound]\label{prop:cov}
The conditional covariance in \eqref{eq:Cp} satisfies
\begin{equation}\label{eq:covsmall}
 \sup_{p<r}\E\int
           C_p(\omega_{\le p})^2\rho_{0:p}(\omega_{\le p})
                        \,d\omega_{\le p}
    \le\frac{\delta^2}{4}+4C_{4,2}\sqrt\eta
    =:\epsilon_{\rm cut}.
\end{equation}
\end{proposition}
The dependence of $C_p$ and $\rho_{0:p}$ on the cavity disorder is implicit, and the outer $\E$ averages the cavity disorder. Thus the factor $\rho_{0:p}$ gives the changed law of the nonroot prefix variables $\omega_1,\ldots,\omega_p$, while $\omega_0$ is averaged under its original law.
\begin{proof}
Choose four terminal labels $\alpha_1,\ldots,\alpha_4$ descending from $\beta_p$ through four distinct children such that $\alpha_j\wedge\alpha_{j'}=p$ for $j\ne j'$. We will use the following identity first with $Y_i^{(j)}=b_{\alpha_j}^i$ and then with $Y_i^{(j)}=\widetilde\xi_i^{\alpha_j}$. In each case, let $\mathcal H_p$ be the sigma-algebra generated by the cavity disorder and the common prefix of the array in question. For $b$ this prefix is $(U_{\le p},Z_{\le p}^1,Z_{\le p}^2)$; for $\widetilde\xi$ it is $\widetilde\omega_{\le p}$. Given $\mathcal H_p$, the four pairs $(Y_1^{(j)},Y_2^{(j)})$ are independent and identically distributed. Therefore
\begin{align}
 \E\Cov(Y_1^{(1)},Y_2^{(1)}\mid\mathcal H_p)^2
  ={}&\E Y_1^{(1)}Y_2^{(1)}Y_1^{(2)}Y_2^{(2)}\notag\\
    &-2\E Y_1^{(1)}Y_2^{(1)}Y_1^{(2)}Y_2^{(3)}
      +\E Y_1^{(1)}Y_2^{(2)}Y_1^{(3)}Y_2^{(4)}.
        \label{eq:fourcopy}
\end{align}
Apply Proposition \ref{prop:transfer} with $n=4$, $s=2$ to each term: replacing $Y_i^{(j)}=b_{\alpha_j}^i$ by $Y_i^{(j)}=\widetilde\xi_i^{\alpha_j}$ changes the combination by at most $4C_{4,2}\sqrt\eta$.

For the original barycenters $Y_i^{(j)}=b_{\alpha_j}^i$, average the two future site paths first. Using $F$ from \eqref{eq:gp}, the conditional covariance becomes
\[
 \Cov_{U_{>p}}\bigl(
       F(U_{\le p},U_{>p},Z_{\le p}^1),
       F(U_{\le p},U_{>p},Z_{\le p}^2)\bigr).
\]
Lemma \ref{lem:gram} bounds \eqref{eq:fourcopy} by $\delta^2/4$. 
To finish the proof it remains to notice that, for the changed array, the left side of \eqref{eq:fourcopy} is exactly the integral in \eqref{eq:covsmall}, by \eqref{eq:Rsegment}.
\end{proof}

\begin{remark} For comparison, instead of using Lemma \ref{lem:gram}, a direct scalar-product calculation gives the weaker bound $\delta^2$. Write
$R_{ij}=b_{\alpha_i}\cdot b_{\alpha_j}$.
After averaging the two site coordinates, the right-hand
side of \eqref{eq:fourcopy} becomes
\[
 \E\bigl[R_{12}^2-2R_{12}R_{13}+R_{13}R_{24}\bigr]
 =\E\bigl[(R_{12}-R_{13})(R_{12}-R_{24})\bigr]
 \le\delta^2.
\]
The equality uses symmetry of the four branches.
The inequality follows because all these scalar products
belong to the same interval $[a_p,b_p]$ of width at most
$\delta$. Lemma \ref{lem:gram} gives the sharper bound
$\delta^2/4$ used here.
\end{remark}

This completes the use of the original cavity equations: only four replicas and two cavity sites were needed, together with the noncavity sites supplying the shape tests. The next section uses the nonnegativity in \eqref{eq:covsmall} to obtain the same conclusion for prescribed, deterministic cavity clauses.

\section{Prescribed clauses via the REM perturbation}\label{sec:probes}

Section \ref{sec:twosite} bounds the covariance of two cavity outputs, averaged over the random cavity disorder. We now use this bound for cavity fields with prescribed multipliers $1+z_iA_i$. The first result is that their conditional covariance is small. We then vary the two unary fields and show that the joint path density is close to the product of the two one-site path densities. Both estimates will be used in Section \ref{sec:response}.

Throughout Sections \ref{sec:probes}--\ref{sec:graphs}, the perturbation strength $\epsilon>0$ is fixed, and limits are along a sequence of cuts with $\delta,\eta\to0$; the depth $r$ is unrestricted. The prescribed polynomials and their evaluation parameters are kept fixed while taking this limit.

\subsection{A prescribed test clause with multiplier $1+zA$}\label{sec:multiplier}

We now construct a cavity field $c$ whose multiplier $ \Av_\varsigma
 e^{c(b_\alpha^1,\ldots,b_\alpha^m,\varsigma)}$
                       has a
prescribed dependence on the cluster magnetizations $b_\alpha^1,\ldots,b_\alpha^m$.
Choose $m\ge1$ input sites and a deterministic function
$A:[0,1]^m\to(0,\infty)$ that is affine in each coordinate
separately. For each terminal label $\alpha$, write
\[
 S_j^\alpha=\frac{1+b_\alpha^j}{2},\qquad
 A^\alpha=A(S_1^\alpha,\ldots,S_m^\alpha).
\]
Our goal is to construct one clause
whose cavity multiplier in cluster $\alpha$ is proportional
to $1+zA^\alpha$, for a parameter $z>0$.

Let $\sigma_1,\ldots,\sigma_m$ be the input spins and
$\varsigma$ the cavity spin. Define a deterministic
clause $\vartheta_*:\{-1,+1\}^{m+1}\to\mathbb R$ by
\begin{align}
 e^{\vartheta_*(\sigma_1,\ldots,\sigma_m,-1)}&=1,\notag\\
 e^{\vartheta_*(\sigma_1,\ldots,\sigma_m,+1)}
   &=A\bigl((1+\sigma_1)/2,\ldots,(1+\sigma_m)/2\bigr).
       \label{eq:targettable}
\end{align}
Strict positivity of $A$ makes all these interaction
values finite. Since $A$ is multiaffine, the Boltzmann
interpolation \eqref{eq:extension} gives, for every $\alpha$,
\[
 e^{\vartheta_*(b_\alpha^1,\ldots,b_\alpha^m,-1)}=1,
 \qquad
 e^{\vartheta_*(b_\alpha^1,\ldots,b_\alpha^m,+1)}
     =A(S_1^\alpha,\ldots,S_m^\alpha)=A^\alpha.
\]
Now add to the cavity field the unary field $h_{\rm u}(1+\varsigma)/2$ with
$h_{\rm u}=\log z$. The resulting multiplier and cavity
magnetization are
\begin{equation}\label{eq:probe}
 \Av_\varsigma
 e^{\vartheta_*(b_\alpha^1,\ldots,b_\alpha^m,\varsigma)
                       +h_{\rm u}(1+\varsigma)/2}
       =\frac{1+zA^\alpha}{2},
 \qquad
 \xi^\alpha=\frac{zA^\alpha-1}{zA^\alpha+1}.
\end{equation}
The factor $1/2$ has no effect on the path densities:
adding a constant to $X_r$ adds it to every $X_j$ in
\eqref{eq:Xrec} and leaves $\rho_j$ unchanged. Along a path
ending at $\alpha$, we may therefore use
$X_r=\log(1+zA^\alpha)$. When working along a generic path,
we suppress $\alpha$ and write $X_r=\log(1+zA)$ and $\xi$;
there $A$ means the evaluated quantity $A(S_1,\ldots,S_m)$.

We call the deterministic clause $\vartheta_*$ a \emph{prescribed test clause} that realizes its desired multiplier. To connect this deterministic construction with our random perturbation, let $\theta_{m+1}^\epsilon$
be one REM clause of arity $m+1$. Its Gaussian entries
have full support, so, for the fixed $\epsilon>0$ and every $t>0$,
\[
 \Prob\Biggl(
   \max_{\substack{\sigma\in\{-1,+1\}^m\\
                    \varsigma\in\{-1,+1\}}}
   \bigl|\theta_{m+1}^\epsilon(\sigma,\varsigma)
               -\vartheta_*(\sigma,\varsigma)\bigr|<t
       \Biggr)>0.
\]
This means that a REM clause together with its Gaussian unary field can approximate any prescribed clause and unary field in \eqref{eq:probe} with positive probability. The next subsection uses this to transfer the covariance bound to the
prescribed field. The test polynomials needed in
Section \ref{sec:graphs} will have exactly this multiaffine form.

\subsection{The covariance for two prescribed fields}\label{sec:prescribe}

Our starting point is the estimate \eqref{eq:covsmall}:
\[
 \sup_{p<r}\E\int C_p(\omega_{\le p})^2
       \rho_{0:p}(\omega_{\le p})\,d\omega_{\le p}
       \le\epsilon_{\rm cut}\longrightarrow0.
\]
Here $C_p$ is computed from the two random cavity fields
in \eqref{eq:barycavityfield}, and the outer expectation
averages their counts and couplings. We want the same
small covariance when each of these entire fields is
replaced by a field of our choosing as in \eqref{eq:probe}. We first specify these two fields and the covariance that will replace $C_p$.

Choose strictly positive multiaffine polynomials
$A_i:[0,1]^{m_i}\to(0,\infty)$ and parameters $z_i>0$,
$i=1,2$. Use disjoint input sites for the two polynomials,
and write
\[
 \boldsymbol b_i^\alpha=(b_\alpha^{i,1},\ldots,b_\alpha^{i,m_i}),
 \qquad
 A_i^\alpha=A_i\bigl((1+b_\alpha^{i,j})/2:j\le m_i\bigr).
\]
Let $\vartheta_{*,i}$ be the deterministic clause
\eqref{eq:targettable} constructed from $A_i$. The prescribed
cavity field at site $i$ is
\[
 c_{i,*}^\alpha(\varsigma)
   =\vartheta_{*,i}(\boldsymbol b_i^\alpha,\varsigma)
           +(\log z_i)\frac{1+\varsigma}{2}.
\]
By \eqref{eq:probe}, its multiplier is $(1+z_iA_i^\alpha)/2$
and its cavity magnetization is $(z_iA_i^\alpha-1)/(z_iA_i^\alpha+1)$.

As in Section \ref{sec:twosite}, the two fields use the same
state path and separate site variables. At each vertex collect
them into
\[
 \omega_\beta
    =\bigl(U_\beta,(Z_\beta^{i,j})_{i\le2,\ j\le m_i}\bigr).
\]
Along a path to $\alpha$, suppress $\alpha$ and define the
normalized terminal field and magnetization by
\begin{equation}
 X_r^*=\log(1+z_1A_1)+\log(1+z_2A_2),
 \qquad
 \xi_i^*=\frac{z_iA_i-1}{z_iA_i+1},\quad i=1,2.
 \label{eq:AA12}
\end{equation}
Apply
\eqref{eq:Xrec} to $X_r^*$ to obtain $X_j^*$ and the
densities $\rho_{p:r}^*$. With the prefix $\omega_{\le p}$
fixed, write
\[
 \E^*_{p\to r}f
    =\int f\,\rho_{p:r}^*\,d\omega_{p+1}\cdots d\omega_r,
 \qquad
 C_p^*=\E^*_{p\to r}(\xi_1^*\xi_2^*)
           -(\E^*_{p\to r}\xi_1^*)(\E^*_{p\to r}\xi_2^*).
\]
Thus $C_p^*$ is the covariance from \eqref{eq:Cp} for the
prescribed fields. It is a function of the retained uniform
variables, with no random clause disorder left.

We use $\|f\|_{2,\mathrm{prod}}$ for the ordinary $L^2$ norm under
the original product law of all uniform variables on which
$f$ depends, including the root. This is different from
the covariance norm $\|f\|_p$ of Section \ref{sec:kernels}.
For $C_p^*$, it averages the state and site prefixes:
\[
 \|C_p^*\|_{2,\mathrm{prod}}^2
     =\int |C_p^*(\omega_{\le p})|^2\,d\omega_{\le p}.
\]

\begin{proposition}[Covariance for prescribed fields]\label{prop:prescribedcov}
For every fixed choice of $A_1,A_2,z_1,z_2$ as above,
\begin{equation}\label{eq:fixedcov}
 \sup_{p<r}\int |C_p^*(\omega_{\le p})|^2\,d\omega_{\le p}
   =\sup_{p<r}\|C_p^*\|_{2,\mathrm{prod}}^2
      \longrightarrow0
\end{equation}
along every sequence of cuts with $\delta,\eta\to0$.
\end{proposition}

\begin{proof}
\emph{Choosing the counts and couplings in the random cavity
fields.}
Recall the fields $A_i^\alpha(\varsigma)$ in
\eqref{eq:barycavityfield}, and take $s=2$. Put $d_i=m_i+1$.
We match the $j$th input of $c_{i,*}^\alpha$ with
\[
 I=(i,d_i,1,j),\qquad j\le m_i,
\]
the $j$th input of the first REM clause of arity $d_i$
at cavity site $i$. 

Fix $0<t\le1$. We will restrict the disorder
$\mathcal F_{\mathrm{cav}}$ in \eqref{eq:Fcav} to the following event $E_t$ on which the full cavity fields \eqref{eq:barycavityfield} approximate $c_{i,*}^\alpha$. Choose $D\ge\max(d_1,d_2)$, to be increased below, and require, for $i=1,2$,
\[
 N_{i,0}=0,\qquad N_{i,d_i}=1,\qquad
 N_{i,d}=0\quad(2\le d\le D,\ d\ne d_i),
\]
and
\[
 \|\theta_{i,d_i,1}^{\epsilon}-\vartheta_{*,i}\|_\infty<t,
 \qquad |g_i-\log z_i|<t.
\]
The norm here is the maximum over the $2^{d_i}$ spin
configurations of the clause. Thus we retain one REM
clause at each site, prescribe its interaction values
and unary field to accuracy $t$, and remove all other
clauses through arity $D$. For the remaining clauses,
also require
\[
 T_D:=\sum_{i=1}^2\sum_{d>D}\sum_{a=1}^{N_{i,d}}
             \|\theta_{i,d,a}^{\epsilon}\|_\infty\le t.
\]
These conditions together define $E_t$. By
\eqref{eq:expsums},
\[
 \E T_D=2\sum_{d>D}d\,\E J_d\longrightarrow0.
\]
Choose $D$ so that $\E T_D<t/2$. Then
$\Prob(T_D\le t)\ge1/2$. The finite count conditions
have positive Poisson probabilities, and the finitely
many prescribed Gaussian variables have full support.
Independence therefore gives $\Prob(E_t)>0$. This
probability depends on the prescribed fields, $\epsilon$
and $t$, but not on the cut.

On $E_t$, the sum of the errors in the two cavity fields
is bounded by
\[
 \sum_{i=1}^2
     \|A_i^\alpha(\varsigma)-c_{i,*}^\alpha(\varsigma)\|_\infty
       \le 4t+T_D\le5t.
\]
The supremum is over the cavity spin and all uniform input
variables. Indeed, the log Boltzmann interpolation
\eqref{eq:extension} changes by at most the maximum error
in its interaction values; each unary field contributes
at most $t$, and the unwanted clauses contribute at most
$T_D$.

For comparison with the normalized target $X_r^*$, use
$X_r=\log\Lambda_\alpha+2\log2$ in the random recursion
for the rest of this proof. The added constant leaves
its path densities, and hence $C_p$, unchanged. The above inequality implies that
\[
 \|X_r-X_r^*\|_\infty\le5t,\qquad
 \max_{i=1,2}\|\xi_i-\xi_i^*\|_\infty\le5t.
\]
Indeed, let
$
 \Delta_i
   =\|A_i^\alpha(\varsigma)-c_{i,*}^\alpha(\varsigma)\|_\infty,$
so that $\Delta_1+\Delta_2\le5t.$
For each site,
\[
 e^{-\Delta_i}\Av_\varsigma e^{c_{i,*}^\alpha(\varsigma)}
 \le \Av_\varsigma e^{A_i^\alpha(\varsigma)}
 \le e^{\Delta_i}\Av_\varsigma e^{c_{i,*}^\alpha(\varsigma)}.
\]
Thus the log-multiplier error at site $i$ is at most
$\Delta_i$, and summing gives
$\|X_r-X_r^*\|_\infty\le\Delta_1+\Delta_2\le5t$.
The magnetization of a field $c$ is
$\tanh((c(+1)-c(-1))/2)$. Since $\tanh$ is
$1$-Lipschitz, the magnetization error at site $i$
is also at most $\Delta_i\le5t$.

\emph{Propagating the approximation through the recursion.}
We next compare the path laws used in $C_p$ and $C_p^*$.
Both recursions may be performed on the full collection
of input arrays: unused coordinates
integrate out of the prescribed recursion.
Each map in \eqref{eq:Xrec} is $1$-Lipschitz in the
supremum norm. Consequently a terminal error at most
$\Delta$ remains at most $\Delta$ at every level, and
$|X_r|\le B$ implies $|X_j|\le B$.

To see why the errors in the densities also stay bounded
independently of the path length, rewrite their definition
by summation by parts:
\begin{equation}\label{eq:abel}
 \log\rho_{p:q}
   =\sum_{j=p}^{q-1}\zeta_j(X_{j+1}-X_j)
   =\zeta_{q-1}X_q-\zeta_pX_p
       -\sum_{j=p+1}^{q-1}(\zeta_j-\zeta_{j-1})X_j,
       \qquad p<q.
\end{equation}
The total absolute coefficient mass in the last expression
is $2\zeta_{q-1}\le2$. Hence
\begin{equation}\label{eq:densitybound}
 e^{-2B}\le\rho_{p:q}\le e^{2B},\qquad
 \|\log\rho_{p:q}-\log\rho_{p:q}^*\|_\infty
       \le2\|X_r-X_r^*\|_\infty.
\end{equation}
In particular, for each fixed prefix, a terminal error
at most $\Delta$ gives
\[
 \int|\rho_{p:r}-\rho_{p:r}^*|\,d\omega_{p+1}\cdots d\omega_r
       \le e^{2\Delta}-1.
\]
Indeed, the ratio $\rho_{p:r}/\rho_{p:r}^*$ lies between
$e^{-2\Delta}$ and $e^{2\Delta}$, so
$|\rho_{p:r}-\rho_{p:r}^*|
\le(e^{2\Delta}-1)\rho_{p:r}^*$, and the latter density
integrates to one for each fixed prefix.

Apply this with $\Delta=5t$ to the three expectations
defining $C_p$, and use the output bounds above and
$|\xi_i|,|\xi_i^*|\le1$. There is a constant $c$, independent of the cut and of $0<t\le1$, such that on the event $E_t$,
\[
 |C_p(\omega_{\le p})-C_p^*(\omega_{\le p})|\le ct
 \qquad\text{for every }p<r
\]
almost everywhere in the prefix variables.

\emph{Applying the estimate from Section \ref{sec:twosite}.}
Choose $B_*$ with $|X_r^*|\le B_*$. On $E_t$ we have
$|X_r|\le B_*+5t$, so \eqref{eq:densitybound} gives
$\rho_{0:p}\ge e^{-2(B_*+5t)}$. Also, $ |C_p^*|^2\le2|C_p|^2+2c^2t^2$ and, therefore
\[
\mathbf1_{E_t} |C_p^*|^2\le  2e^{2(B_*+5t)}\mathbf1_{E_t}|C_p|^2\rho_{0:p}  +2 \mathbf1_{E_t} c^2t^2.
\]
Average this inequality over the cavity disorder and
the uniform prefix variables. The target $C_p^*$ is independent of the cavity
disorder, and the extra input coordinates on which it
does not depend integrate to one. Therefore
\begin{align*}
 \Prob(E_t)\|C_p^*\|_{2,\mathrm{prod}}^2
 &\le 2e^{2(B_*+5t)}
       \E\left[
          \mathbf1_{E_t}\int C_p^2\rho_{0:p}\,d\omega_{\le p}
                         \right]
       +2c^2t^2\Prob(E_t).
\end{align*}
Here $\E$ averages the cavity
disorder only; $d\omega_{\le p}$ includes the root under
its original law. The integral is nonnegative, so
\eqref{eq:covsmall} bounds the expectation on the right
by $\epsilon_{\rm cut}$. Dividing by $\Prob(E_t)$ gives
\begin{equation}\label{eq:localization}
 \sup_{p<r}\|C_p^*\|_{2,\mathrm{prod}}^2
 \le\frac{2e^{2(B_*+5t)}}{\Prob(E_t)}\epsilon_{\rm cut}
                                      +2c^2t^2.
\end{equation}
For fixed $t$, the event probability is positive and
independent of the cut. First let $\delta,\eta\to0$,
and then let $t\downarrow0$. This proves
\eqref{eq:fixedcov}.
\end{proof}

This is the same prescription of Poisson counts and
Gaussian couplings used in \cite[Section 6]{rRSB}, with an additional approximation error to track. We keep a neighborhood of the
prescribed values fixed while the cut error tends to zero.
Only the nonnegative estimate \eqref{eq:covsmall} is
restricted to this neighborhood; we do not need cavity
equations conditional on the event.

\subsection{Differentiating the two unary fields}\label{sec:hessian}

We now compare the joint change of density for the two
prescribed fields with the changes of density obtained
from each field separately. Apply \eqref{eq:Xrec} to
the three terminal fields
\[
 X_r^{12}=\log(1+z_1A_1)+\log(1+z_2A_2),\qquad
 X_r^i=\log(1+z_iA_i),\quad i=1,2.
\]
Write $X_j^{12},X_j^i$ for the resulting recursions and
$\rho_{p:r}^{12},\rho_{p:r}^i$ for their path densities.
All three use the same state path. The $i$th one-site
field uses only its own site variables; when placed on
the joint space, it ignores the other site's variables.
The two-site field $X^{12}$ is exactly $X^*$ from
\eqref{eq:AA12}.

If $X_j^{12}=X_j^1+X_j^2$ at every depth, the definition
of the densities would give
$\rho_{p:r}^{12}=\rho_{p:r}^1\rho_{p:r}^2$.
We therefore first study the difference
\[
 \Phi_p=X_p^{12}-X_p^1-X_p^2.
\]
This difference is zero at the terminal level.
Recursive averaging can produce a nonzero difference
at earlier levels, because the two fields share the
state variables. To estimate $\Phi_p$, we vary the
two unary fields $h_i=\log z_i$.

The mixed derivative of $\Phi_p$ is
$\partial_{h_1h_2}X_p^{12}$, since each $X_p^i$
depends on only one parameter. The next identity
expresses this derivative through the covariances
controlled in the preceding subsection.
Here $C_j$ denotes $C_j^*$ from that subsection,
and $\E^{12}_{p\to j}$ is expectation with density
$\rho_{p:j}^{12}$, keeping the prefix through $p$ fixed.

\begin{lemma}[Unary response identity]\label{lem:hessian}
For every $p<r$,
\begin{equation}\label{eq:hessian}
 \partial_{h_1h_2}X_p^{12}
 =\frac14\left[
    \zeta_pC_p+
    \sum_{j=p+1}^{r-1}(\zeta_j-\zeta_{j-1})
                            \E^{12}_{p\to j}C_j
           \right].
\end{equation}
The nonnegative coefficients inside the brackets have
sum $\zeta_{r-1}\le1$.
\end{lemma}

\begin{proof}
Let $m_{i,j}=\partial_{h_i}X_j^{12}$ and
$H_j=\partial_{h_1h_2}X_j^{12}$.
Denote expectation and covariance in $\omega_{j+1}$
with density $\rho_j^{12}$ by $\E_j^{12}$ and
$\Cov_j^{12}$, with the prefix through $j$ fixed.
Differentiating the logarithm in \eqref{eq:Xrec} gives
\begin{align*}
 m_{i,j}
 &=\frac{\int m_{i,j+1}e^{\zeta_jX_{j+1}^{12}}
                         \,d\omega_{j+1}}
         {\int e^{\zeta_jX_{j+1}^{12}}\,d\omega_{j+1}}
 =\int m_{i,j+1}\rho_j^{12}\,d\omega_{j+1}
   =\E_j^{12}m_{i,j+1}.
\end{align*}
For the mixed derivative, observe first that
\[
 \partial_{h_2}\rho_j^{12}
 = \partial_{h_2} \exp\{\zeta_j(X_{j+1}^{12}-X_j^{12})\}
   =\zeta_j\rho_j^{12}(m_{2,j+1}-m_{2,j}).
\]
Differentiating the above identity $m_{1,j} =\int m_{1,j+1}\rho_j^{12}\,d\omega_{j+1}$ in $h_2$ therefore gives
\begin{align*}
 H_j
 &=\int\left(
       H_{j+1}\rho_j^{12}
       +m_{1,j+1}\partial_{h_2}\rho_j^{12}
       \right)d\omega_{j+1}\\
 &=\E_j^{12}H_{j+1}
   +\zeta_j\Bigl(
       \E_j^{12}(m_{1,j+1}m_{2,j+1})-m_{1,j}m_{2,j}
            \Bigr)\\
 &=\E_j^{12}H_{j+1}
   +\zeta_j\Cov_j^{12}(m_{1,j+1},m_{2,j+1}).
\end{align*}
At the terminal level,
\[
 m_{i,r}=\frac{z_iA_i}{1+z_iA_i}
        =\frac{1+\xi_i^*}{2},
 \qquad H_r=0.
\]
Iterating the first-derivative identity $m_{i,j}=\E_j^{12}m_{i,j+1}$ gives
\[
 m_{i,j}=\frac{1+\E^{12}_{j\to r}\xi_i^*}{2},\,\,\mbox{ or }\,\,
\E^{12}_{j\to r}\xi_i^* = 2 m_{i,j} -1.
\]
Thus, conditional on the prefix through $j+1$, the
means $\E^{12}_{j+1\to r}$ of the two outputs  $\xi_i^*$ are $2m_{i,j+1}-1$, and their covariance is $C_{j+1}$.
Their conditional product moment is consequently
\[
 \E^{12}_{j+1\to r}(\xi_1^*\xi_2^*)
   =C_{j+1}+(2m_{1,j+1}-1)(2m_{2,j+1}-1).
\]
Average this identity over $\omega_{j+1}$ with density
$\rho_j^{12}$, and subtract the product of the means
conditional on the prefix through $j$. This gives
\begin{align*}
 C_j
 &=\E_j^{12}C_{j+1}
   +\E_j^{12}\bigl[(2m_{1,j+1}-1)(2m_{2,j+1}-1)\bigr]
    -(2m_{1,j}-1)(2m_{2,j}-1)\\
 &=\E_j^{12}C_{j+1}
   +4\Cov_j^{12}(m_{1,j+1},m_{2,j+1}),
\end{align*}
where the last equality uses
$m_{i,j}=\E_j^{12}m_{i,j+1}$.
This is the law of total covariance in our notation.
At depth $r$, the entire path is fixed, so $C_r=0$.

Substituting this covariance identity into the recursion
for $H_j$ above yields
\[
 4H_j=4\E_j^{12}H_{j+1}
             +\zeta_j(C_j-\E_j^{12}C_{j+1}).
\]
Iterate from $p$ to $r$, using $H_r=0$ and the
composition rule
$\E^{12}_{p\to j}\E_j^{12}=\E^{12}_{p\to j+1}$:
\begin{align*}
 4H_p
 &=\sum_{j=p}^{r-1}\zeta_j\E^{12}_{p\to j}
                         (C_j-\E_j^{12}C_{j+1})=\sum_{j=p}^{r-1}\zeta_j\E^{12}_{p\to j}C_j
   -\sum_{j=p}^{r-1}\zeta_j\E^{12}_{p\to j+1}C_{j+1}\\
 &=\zeta_pC_p+
   \sum_{j=p+1}^{r-1}(\zeta_j-\zeta_{j-1})
                             \E^{12}_{p\to j}C_j
       -\zeta_{r-1}\E^{12}_{p\to r}C_r.
\end{align*}
The last term vanishes, proving \eqref{eq:hessian}.
\end{proof}

\subsection{Factorizing the two path densities}\label{sec:factorization}

We now integrate the preceding identity in the two
parameters to recover $\Phi_p=X_p^{12}-X_p^1-X_p^2$, and then use the
definition of the path densities. The resulting
factorization is the second estimate needed in
Section \ref{sec:response}.

\begin{proposition}[Factorization of the prescribed path densities]\label{prop:factorization}
For fixed prescribed clauses $A_1,A_2$ and $z_1,z_2>0$,
\begin{equation}\label{eq:Phismall}
 \sup_{p\le r}\|\Phi_p(z_1,z_2)\|_{2,\mathrm{prod}}
       \longrightarrow0,
\end{equation}
and
\begin{equation}\label{eq:Rfactor}
 \sup_{p<r}
   \|\rho_{p:r}^{12}-\rho_{p:r}^1\rho_{p:r}^2\|_{2,\mathrm{prod}}
       \longrightarrow0
\end{equation}
as $\delta,\eta\to0$, with $r$ unrestricted.
The first norm averages the retained prefixes; the second
averages the full path of state and site variables,
including the root. The three densities in
\eqref{eq:Rfactor} are evaluated on the same state path
and the two original, independent site variables.
\end{proposition}

\begin{proof}
\emph{From the covariance to the mixed derivative.}
Choose $M_i$ with $0<A_i\le M_i$, and let $(s_1,s_2)$
denote the varying parameters in the rectangle
$0\le s_i\le z_i$. The three terminal fields on this
rectangle are bounded in absolute value by
\[
 \overline B=\log\bigl((1+z_1M_1)(1+z_2M_2)\bigr).
\]
Fix a positive pair $(s_1,s_2)$ and put $h_i=\log s_i$.
For the two-site law at these parameters, conditional
Jensen and \eqref{eq:densitybound} give
\[
 \|\E^{12}_{p\to j}C_j\|_{2,\mathrm{prod}}^2
 \le\int C_j^2\rho_{p:j}^{12}\,d\omega_{\le j}
 \le e^{2\overline B}\|C_j\|_{2,\mathrm{prod}}^2.
\]
The coefficient bound in Lemma \ref{lem:hessian}
therefore implies
\[
 \sup_{p<r}\|\partial_{h_1h_2}X_p^{12}\|_{2,\mathrm{prod}}^2
    \le\frac{e^{2\overline B}}{16}
                    \sup_{j<r}\|C_j\|_{2,\mathrm{prod}}^2
        \longrightarrow0
\]
by Proposition \ref{prop:prescribedcov}. Equivalently,
\[
 \partial_{s_1s_2}\Phi_p
    =\frac{\partial_{h_1h_2}X_p^{12}}{s_1s_2}
\]
is small in $L^2$, uniformly in $p$, at every fixed
positive pair of parameters. Since $|C_j|\le1$,
Lemma \ref{lem:hessian} also gives the pointwise bound
\[
 |\partial_{s_1s_2}\Phi_p|\le\frac1{4s_1s_2}.
\]

\emph{Integrating the mixed derivative.}
When one parameter is zero, its terminal field is zero
and the two-site recursion becomes the other one-site
recursion. Hence
$
 \Phi_p(s_1,0)=\Phi_p(0,s_2)=0.
$
These are identities of the recursion, so no prescribing
event at zero is needed. Moreover,
$|\partial_{s_i}\Phi_p|\le M_i$: the derivatives of
$X_p^{12}$ and $X_p^i$ are both tilted averages of
$A_i/(1+s_iA_i)\in[0,M_i]$.

To stay away from zero while using the mixed-derivative
estimate, fix $0<a<\min(z_1,z_2)$. The fundamental
theorem of calculus gives
\begin{align*}
 \Phi_p(z_1,z_2)
  ={}&\Phi_p(z_1,a)+\Phi_p(a,z_2)-\Phi_p(a,a)+\int_a^{z_1}\!\!\int_a^{z_2}
             \partial_{s_1s_2}\Phi_p(s_1,s_2)\,ds_2\,ds_1.
\end{align*}
The three boundary terms have total absolute value at
most $2(M_1+M_2)a$. Minkowski's inequality thus yields
\begin{align*}
 \sup_{p\le r}\|\Phi_p(z_1,z_2)\|_{2,\mathrm{prod}}
 \le{}&2(M_1+M_2)a+\int_a^{z_1}\!\!\int_a^{z_2}
       \sup_{p<r}\|\partial_{s_1s_2}\Phi_p(s_1,s_2)\|_{2,\mathrm{prod}}
                   \,ds_2\,ds_1.
\end{align*}
For fixed $a$, dominated convergence applies with the
bound $1/(4s_1s_2)$, so the integral tends to zero.
Each positive pair $(s_1,s_2)$ uses its own event in
Proposition \ref{prop:prescribedcov}; pointwise convergence
suffices for this argument. First take the fine-cut
limit, then let $a\downarrow0$. This proves
\eqref{eq:Phismall}.

\emph{From the recursions to the densities.}
Apply \eqref{eq:abel} to the joint density and to each
one-site density. Since $\Phi_r=0$,
\[
 \log\rho_{p:r}^{12}-\log\rho_{p:r}^1-\log\rho_{p:r}^2
   =-\zeta_p\Phi_p
      -\sum_{j=p+1}^{r-1}(\zeta_j-\zeta_{j-1})\Phi_j.
\]
The absolute coefficient mass is at most one. Each
$\Phi_j$ depends only on the prefix through $j$, so its
$L^2$ norm is unchanged by integrating the remaining
path coordinates. It follows from \eqref{eq:Phismall}
that the left side tends to zero in $L^2$, uniformly
in $p$. The densities and their logarithms are bounded
by \eqref{eq:densitybound}; exponentiating therefore
proves \eqref{eq:Rfactor}.
\end{proof}

We now have both the small covariance
\eqref{eq:fixedcov} and the factorization
\eqref{eq:Rfactor} for prescribed fields. The latter
separates the two site contributions while keeping
their common state variables. The next section uses
these two estimates together to control a conditional
cavity output with its future state path left unaveraged.

\section{One-site control and the marked-replica expansion}\label{sec:response}

\subsection{From two-site estimates to one-site control}\label{sec:onesite}

We now combine the density factorization of Proposition
\ref{prop:factorization} with the two-site covariance bound of
Proposition \ref{prop:prescribedcov} to control one cavity output at a time. We will now also choose a special prescribed clause that will be utilized in the following sections. Let $m\ge1$ be the number of input sites, let
$L:[0,1]^m\to\mathbb R$ be multiaffine with
$\|L\|_\infty\le s<1/4$, and put
\[
 A=1+2L.
\]
Then $A>0$, so we may use this polynomial in \eqref{eq:probe}
with unary parameter $z=1$. Writing
$\boldsymbol b=(b^1,\ldots,b^m)$ for the input magnetizations,
the multiplier and output are
\begin{equation}\label{eq:linearprobe}
 \Av_\varsigma e^{\vartheta_*(\boldsymbol b,\varsigma)}
       =\frac{1+A}{2}=1+L,
 \qquad
 \xi^L=\frac{A-1}{A+1}=\frac{L}{1+L}.
\end{equation}
Here and below, $L$ along a path means
$L(S_1,\ldots,S_m)$, where $S_i=(1+b^i)/2$.
The function $L$ may have either sign; it is $A=1+2L$
that must be positive.

Fix $p<r$. Along one canonical path, write
\[
 U_{\le p}=(U_0,\ldots,U_p),\qquad
 U=(U_{p+1},\ldots,U_r).
\]
Thus $U$ is the vector of future state variables used in
Section \ref{sec:geometry}, the one-path case of $U_T$.
For the $m$ input sites, write
\[
 Z_{\le p}=(Z_j^i)_{0\le j\le p,\,i\le m},\qquad
 Z_{>p}=(Z_j^i)_{p<j\le r,\,i\le m}.
\]
Use the recursion \eqref{eq:Xrec} with terminal field
$X_r^L=\log(1+L)$, and denote its path density by
$\rho_{p:r}^L$. The terminal field $\log(1+A)$ used in
Section \ref{sec:probes} differs by the constant $\log2$,
so it gives the same path density.

Average the future site variables, keeping $U$ and all
prefixes fixed:
\begin{equation}\label{eq:conditionaloutput}
 \begin{aligned}
 a_p^L(U,Z_{\le p})
    &=\int\rho_{p:r}^L\,dZ_{>p},\\
 d_p^L(U,Z_{\le p})
    &=\int\rho_{p:r}^L\xi^L\,dZ_{>p}
      =\int\rho_{p:r}^L\frac{L}{1+L}\,dZ_{>p}.
 \end{aligned}
\end{equation}
The dependence on $U_{\le p}$ is suppressed. Thus $a_p^L$
is the marginal density of the future state path, and
$d_p^L$ is the density-weighted output $\xi^L$ after the same site
averaging. Normalization gives $\E_Ua_p^L=1$. The norm $\|\cdot\|_p$ below uses the entire input-site prefix $Z_{\le p}$ as its site variable.

\begin{proposition}[Covariance control of the conditional cavity output]
\label{prop:response}
For every fixed multiaffine $L$ with $\|L\|_\infty<1/4$,
\begin{equation}\label{eq:outputsmall}
 \sup_{p<r}\|a_p^L-1\|_p\longrightarrow0,
 \qquad
 \sup_{p<r}\|\Pi_p d_p^L\|_p\longrightarrow0
\end{equation}
as $\delta,\eta\to0$, with no restriction on the depth $r$.
\end{proposition}
\begin{proof}
Fix $L$. Take two copies of the prescribed clause, with
independent input-site arrays and the same state path. Write
$\rho^{12},\rho^1,\rho^2$ for the joint and separate path
densities from $p$ to $r$, and let $\xi_1,\xi_2$ be the
corresponding outputs in \eqref{eq:linearprobe}.
Suppress $U_{\le p}$ and abbreviate the two retained site
prefixes by $Z,Z'$. Set
\[
 a_Z(U)=a_p^L(U,Z),\qquad d_Z(U)=d_p^L(U,Z),
\]
and center in the future state variables:
\[
 a_Z^\circ=a_Z-1,\qquad
 \bar d_Z=\E_Ud_Z,\qquad d_Z^\circ=d_Z-\bar d_Z.
\]
The densities and all these functions are bounded uniformly
in the cut and in $p$. Indeed,
$|\log(1+L)|\le-\log(1-s)$, so \eqref{eq:densitybound}
applies to the separate and joint densities, while
$|\xi^L|\le s/(1-s)$. Write
\[
 \E_{\mathrm{fut}}=\E_{U,Z_{>p},Z_{>p}'}
\]
for expectation under the original product law of the future
variables, with $U_{\le p},Z,Z'$ fixed. Below, $L^2$ without
further qualification averages $U_{\le p},Z,Z'$.

\emph{Expressing the two-site estimate through one-site quantities.}
The covariance controlled by \eqref{eq:fixedcov} is
\[
 C_p=\E_{\mathrm{fut}}[\rho^{12}\xi_1\xi_2]
       -\E_{\mathrm{fut}}[\rho^{12}\xi_1]
        \E_{\mathrm{fut}}[\rho^{12}\xi_2].
\]
We can use factorization in \eqref{eq:Rfactor} to replace $\rho^{12}$ by $\rho^1\rho^2$ in each expectation. Integrating the
independent future site arrays first gives, respectively,
\[
 \E_Ud_Zd_{Z'},\qquad
 \E_Ud_Za_{Z'},\qquad
 \E_Ua_Zd_{Z'}.
\]
For example,
\[
 \E_{\mathrm{fut}}[\rho^1\rho^2\xi_1]
   =\E_U\left[
        \left(\E_{Z_{>p}}\rho^1\xi_1\right)
        \left(\E_{Z_{>p}'}\rho^2\right)\right]
   =\E_Ud_Za_{Z'}.
\]
By Jensen's inequality, boundedness, and \eqref{eq:Rfactor},
each replacement error tends to zero in $L^2$, uniformly
in $p$. The same holds for the error in the product of the
last two expectations. Combining this with \eqref{eq:fixedcov}
gives
\begin{equation}\label{eq:ddidentity}
 \mathcal E_p(Z,Z'):=\E_Ud_Zd_{Z'}
       -(\E_Ud_Za_{Z'})(\E_Ua_Zd_{Z'})
       \longrightarrow0
\end{equation}
in $L^2$, uniformly in $p$.

This is almost the desired covariance kernel of $d$, but
the two means still contain the density factors $a_Z,a_{Z'}$.
They remain because the two sites share the state path $U$.
We next show that their centered parts are small in
$\|\cdot\|_p$, and the mixed-kernel bound will then allow us to
remove them from \eqref{eq:ddidentity}.

\emph{Controlling the remaining state-path density.}
Normalization of the joint density gives
$\E_{\mathrm{fut}}\rho^{12}=1$, whereas independence of
the future site arrays gives
$\E_{\mathrm{fut}}[\rho^1\rho^2]=\E_Ua_Za_{Z'}$.
Since $\E_Ua_Z=\E_Ua_{Z'}=1$, it follows that
\[
 K_a(Z,Z'):=\E_Ua_Z^\circ a_{Z'}^\circ
      =\E_Ua_Za_{Z'}-1
      =\E_{\mathrm{fut}}[\rho^1\rho^2-\rho^{12}].
\]
Thus the covariance kernel of the marginal density is
itself controlled by the factorization error. Jensen's
inequality and \eqref{eq:Rfactor} yield
\[
 \|a^\circ\|_p^4
   =\E_{U_{\le p},Z,Z'}|K_a(Z,Z')|^2
   \le\E|\rho^1\rho^2-\rho^{12}|^2
   \longrightarrow0
\]
uniformly in $p$. The expectation on the right also averages
all future variables. This proves the first assertion of
\eqref{eq:outputsmall}.

\emph{Removing the density factors from the output estimate.}
Put
\[
 K_{da}(Z,Z')=\E_Ud_Z^\circ a_{Z'}^\circ,
 \qquad
 K_{ad}(Z,Z')=\E_Ua_Z^\circ d_{Z'}^\circ.
\]
The mixed-kernel bound of Lemma \ref{lem:normfacts} gives
\begin{equation}\label{eq:crosskernel}
 \E|K_{da}|^2
      \le\|d^\circ\|_p^2\|a^\circ\|_p^2
      \longrightarrow0;
\end{equation}
the same holds for $K_{ad}$. This uses only boundedness of
$d^\circ$ and the estimate just proved for $a^\circ$.
Since the centered functions have zero mean in $U$,
\[
 \E_Ud_Za_{Z'}=\bar d_Z+K_{da},\qquad
 \E_Ua_Zd_{Z'}=\bar d_{Z'}+K_{ad}.
\]
Substitution into \eqref{eq:ddidentity} now identifies the
desired covariance kernel:
\[
 K_d(Z,Z'):=\E_Ud_Z^\circ d_{Z'}^\circ
   =\mathcal E_p+\bar d_ZK_{ad}
                    +\bar d_{Z'}K_{da}+K_{da}K_{ad}.
\]
Every term on the right tends to zero in $L^2$, uniformly
in $p$; for the product term, use boundedness of the
mixed kernels. Hence
$\|d^\circ\|_p^4=\E|K_d|^2\to0$, proving the second
assertion of \eqref{eq:outputsmall}.
\end{proof}

\subsection{The marked-replica expansion}\label{sec:spine}

We next express $d_p^L$ in \eqref{eq:conditionaloutput} through replica moments. The point of the particular output choice of the form $L/(1+L)$ is that its denominator cancels the multiplier when a state is sampled from the reweighted cascade. Expanding the
remaining normalization produces the replica moments.

Fix a vertex $\beta\in\N^p$. For a terminal label
$\alpha$ descending from $\beta$, write
\[
 L^\alpha=L(S_1^\alpha,\ldots,S_m^\alpha).
\]
The state prefix is $U_{\le p}$, and the site prefix
contains all $m$ inputs,
$
 Z_{\le p}=(Z_\gamma^i)_{\gamma\in p(\beta),\,i\le m}.
$
For $n\ge1$, choose terminal labels
$\alpha_1,\ldots,\alpha_n$ realizing a shape $T$ of
the auxiliary cascade rooted at $\beta$. Define
\begin{equation}\label{eq:probevariables}
 \begin{aligned}
 V_T&=\bigcup_{\ell=1}^n
             \bigl(p(\alpha_\ell)\setminus p(\beta)\bigr),\qquad
 U_T=(U_\gamma)_{\gamma\in V_T},\qquad
 Z_T=(Z_\gamma^i)_{\gamma\in V_T,\,i\le m}.
 \end{aligned}
\end{equation}
The corresponding future moment is
\begin{equation}\label{eq:FTprobe}
 F_T^{p,L}(U_{\le p},U_T,Z_{\le p})
      =\int\prod_{\ell=1}^n L^{\alpha_\ell}\,dZ_T.
\end{equation}
Here $dZ_T$ is the product uniform measure, with one
coordinate for each distinct vertex and each input site.
Replicas share the coordinates on their common paths;
if terminal labels repeat, the same factor $L^\alpha$
occurs repeatedly. All state variables and all site
prefixes remain unaveraged. This is the construction
of \eqref{eq:FT}, with $L$ in place of $S$.

The function $d_p^L$ retains only one future state path,
whereas $F_T^{p,L}$ depends on all of $U_T$. To relate them, choose the realizations of all shapes with replica $1$ on the same
fixed descendant path of $\beta$, and write
\[
 V_1=p(\alpha_1)\setminus p(\beta),\qquad
 U=(U_\gamma)_{\gamma\in V_1}
       =(U_{p+1},\ldots,U_r).
\]
For a function $F=F(U_{\le p},U_T,Z_{\le p})$, define
\begin{equation}\label{eq:SpT}
 (\Sp_TF)(U_{\le p},U,Z_{\le p})
     =\int F(U_{\le p},U_T,Z_{\le p})
                    \prod_{\gamma\in V_T\setminus V_1}dU_\gamma.
\end{equation}
Only the state coordinates outside replica $1$'s path
are integrated. All these integrals use the original
product uniform law. Consequently every
$\Sp_TF_T^{p,L}$ is a function of the same variables
$(U_{\le p},U,Z_{\le p})$, and we can sum them over shapes.

\begin{lemma}[Marked-replica expansion]\label{lem:spinal}
Let $\|L\|_\infty\le s<1/4$. For every integer $k\ge1$,
\begin{equation}\label{eq:spinal}
 d_p^L
   =\sum_{n=1}^k(-1)^{n-1}
       \sum_T\mu_n^p(T)\,\Sp_TF_T^{p,L}
       +R_{p,k}^L,
\end{equation}
where the inner sum is over shapes of $n$ replicas, and
\begin{equation}\label{eq:spinalremainder}
 \|R_{p,k}^L\|_\infty\le\frac{s^{k+1}}{1-s}.
\end{equation}
\end{lemma}
\begin{proof}
Keep $U_{\le p}$ and $Z_{\le p}$ fixed. On the descendant
tree of $\beta$, take auxiliary RPC weights
$v=(v_\alpha)$ with parameters
$\zeta_p,\ldots,\zeta_{r-1}$, independently of the uniform
arrays. Write $\E_{\mathrm{aux}}$ for expectation over
these weights and all state and site variables at
vertices strictly descending from $\beta$.
For a terminal label $\alpha$, its future state path is
$$
 U^\alpha=(U_\gamma)_{\gamma\in
                        p(\alpha)\setminus p(\beta)}.
$$
Let $f$ be any bounded measurable function of this path,
and set
\[
 \langle L\rangle_v=\sum_\alpha v_\alpha L^\alpha,
 \qquad
 \langle Lf\rangle_v
      =\sum_\alpha v_\alpha L^\alpha f(U^\alpha).
\]

\emph{An exact pairing identity.}
Reweight the cascade by its multiplier $1+L^\alpha$:
\[
 \widetilde v_\alpha
     =\frac{v_\alpha(1+L^\alpha)}{1+\langle L\rangle_v}.
\]
In the weighted output the factor $1+L^\alpha$ cancels,
so
\[
 \sum_\alpha\widetilde v_\alpha
            \frac{L^\alpha f(U^\alpha)}{1+L^\alpha}
       =\frac{\langle Lf\rangle_v}{1+\langle L\rangle_v}.
\]
The change-of-density invariance of
Section \ref{sec:changeofmeasure} says that the expectation
of the left side is the expectation along one path with
density $\rho_{p:r}^L$. Hence
\begin{align}
 \E_{\mathrm{aux}}
       \frac{\langle Lf\rangle_v}{1+\langle L\rangle_v}
   &=\iint f(U)\frac{L}{1+L}\rho_{p:r}^L
                                    \,dZ_{>p}\,dU
                                    =\int f(U)d_p^L(U,Z_{\le p})\,dU.
       \label{eq:spinalpairing}
\end{align}

\emph{Expanding the normalization.}
Since $|\langle L\rangle_v|\le s<1$, the finite geometric
identity gives
\begin{equation}\label{eq:finitegeometric}
 \frac{\langle Lf\rangle_v}{1+\langle L\rangle_v}
   =\sum_{n=1}^k(-1)^{n-1}
           \langle Lf\rangle_v\langle L\rangle_v^{n-1}
     +(-1)^k
       \frac{\langle Lf\rangle_v\langle L\rangle_v^k}
            {1+\langle L\rangle_v}.
\end{equation}
The term indexed by $n$ is an $n$-replica moment:
\[
 \langle Lf\rangle_v\langle L\rangle_v^{n-1}
    =\sum_{\alpha_1,\ldots,\alpha_n}
       \Bigl(\prod_{\ell=1}^n v_{\alpha_\ell}\Bigr)
       f(U^{\alpha_1})\prod_{\ell=1}^nL^{\alpha_\ell}.
\]
For a tuple of shape $T$, integrate its future site
variables first, and then the state variables outside
replica $1$'s path. The definition \eqref{eq:SpT} gives
\[
 \iint f(U^{\alpha_1})\prod_{\ell=1}^nL^{\alpha_\ell}
                                         \,dZ_T\,dU_T
     =\int f(U)(\Sp_TF_T^{p,L})(U,Z_{\le p})\,dU.
\]
The weights are independent of the uniform arrays, and
their contribution for this shape is
\[
 \E_v\sum_{\boldsymbol\alpha:T(\boldsymbol\alpha)=T}
                  \prod_{\ell=1}^n v_{\alpha_\ell}
       =\mu_n^p(T).
\]
Grouping by shapes therefore identifies the first $k$
terms of \eqref{eq:finitegeometric} with the sum in
\eqref{eq:spinal}, paired against $f$.

\emph{A uniform bound on the remainder.}
Let $R_{p,k}^L$ be the difference in \eqref{eq:spinal}.
By \eqref{eq:spinalpairing} and \eqref{eq:finitegeometric},
\[
 \int f(U)R_{p,k}^L(U,Z_{\le p})\,dU
    =(-1)^k\E_{\mathrm{aux}}
       \frac{\langle Lf\rangle_v\langle L\rangle_v^k}
            {1+\langle L\rangle_v}.
\]
We have $|\langle Lf\rangle_v|\le s\langle|f|\rangle_v$.
Also, a label sampled from the original weights has the
original product uniform law for its state path: the
weights are independent of the arrays and sum to one.
Thus
\[
 \E_{\mathrm{aux}}\langle|f|\rangle_v
      =\int|f(U)|\,dU,
\]
and consequently
\[
 \left|\int f(U)R_{p,k}^L(U,Z_{\le p})\,dU\right|
      \le\frac{s^{k+1}}{1-s}\int|f(U)|\,dU.
\]
Since this holds for every bounded measurable $f$, it
implies \eqref{eq:spinalremainder}. The bound is independent
of the fixed prefixes, so it holds in the full product
space as well.
\end{proof}

We record how the state averages interact with centering.
With the prefixes fixed, integrating the path $U$ after
the coordinates outside it integrates all of $U_T$.
It follows that
\begin{equation}\label{eq:commute}
 \Sp_T\Pi_pF=\Pi_p\Sp_TF.
\end{equation}
On the left $\Pi_p$ centers in $U_T$; on the right it
centers in $U$. Moreover, $\Sp_T$ contracts the norm by
the conditional-expectation property in
Section \ref{sec:kernels}.
For later use, $\Sp_TH_T^p=H_T^p$: the common-stem
moment \eqref{eq:HT} depends only on state coordinates
along the common stem, all of which belong to replica
$1$'s path.

\subsection{Extracting a mixed coefficient}\label{sec:mixedcoeff}

We will now further specialize the choice of the multiaffine function $L$ in the previous sections. Fix $k\ge1$ and multiaffine functions
\begin{equation}
 P_v:[0,1]^m\to[-1,1],\qquad 1\le v\le k,
\label{eq:Pvmixed}
\end{equation}
all using the same $m$ input sites. For
$\mathbf t=(t_1,\ldots,t_k)$ with
$\sum_v|t_v|<1/4$, set
\begin{equation}\label{eq:mixedprobe}
 L_{\mathbf t}=\sum_{v=1}^k t_vP_v,
 \qquad A_{\mathbf t}=1+2L_{\mathbf t},
 \qquad z=1.
\end{equation}
Putting a separate parameter on each term will allow us to generate arbitrary multioverlap moments in the next two sections.

Since $|P_v|\le1$, we have
$\|L_{\mathbf t}\|_\infty\le\sum_v|t_v|<1/4$, so
$A_{\mathbf t}>1/2$ even when the factors are signed.
All parameter choices use the same site arrays, including
at values where a parameter is zero and some coordinates
are unused. Proposition \ref{prop:response} applies to
every fixed choice of $\mathbf t$.
For a shape $T$ of $n$ replicas, the future multioverlap
$$
 F_T^{p,L_{\mathbf t}}(U_{\le p},U_T,Z_{\le p})
      =\int\prod_{\ell=1}^n L_{\mathbf t}^{\alpha_\ell}\,dZ_T.
$$
is homogeneous of total degree $n$ in $\mathbf t$.
When $n=k$, the coefficient of $t_1\cdots t_k$ assigns
each $P_v$ to one replica, with every replica used once.
Let $\mathfrak S_k$ denote the set of permutations of
$\{1,\ldots,k\}$, and write $\pi(v)$ for the replica
assigned to the factor $P_v$. Thus, the coefficient is
\begin{equation}\label{eq:mixedmoment}
 G_T^p
   :=\int\sum_{\pi\in\mathfrak S_k}
               \prod_{v=1}^k P_v^{\alpha_{\pi(v)}}\,dZ_T.
\end{equation}
Here,
$P_v^\alpha=P_v(S_1^\alpha,\ldots,S_m^\alpha)$.
All input-site prefixes remain as arguments of $G_T^p$.

\begin{proposition}[Mixed replica tests]\label{prop:mixedcoeff}
For every fixed collection $P_1,\ldots,P_k$ as above,
\begin{equation}\label{eq:testedsum}
 \sup_{p<r}\left\|
    \sum_T\mu_k^p(T)\,\Sp_T\Pi_pG_T^p
                  \right\|_p\longrightarrow0.
\end{equation}
The norm uses the full tuple of the $m$ input-site prefixes.
\end{proposition}
\begin{proof}
Write $s=\sum_v|t_v|<1/4$. Apply $\Pi_p$ to
Lemma \ref{lem:spinal} and use \eqref{eq:commute}:
\[
 \begin{aligned}
 Q_p(\mathbf t)
   &:=\sum_{n=1}^k(-1)^{n-1}
       \sum_T\mu_n^p(T)\,\Sp_T\Pi_pF_T^{p,L_{\mathbf t}}
 =\Pi_p d_p^{L_{\mathbf t}}
       -\Pi_p R_{p,k}^{L_{\mathbf t}}.
 \end{aligned}
\]
Since centering contracts the norm and
$\|R\|_p\le\|R\|_\infty$, the remainder bound and
Proposition \ref{prop:response} give, for every fixed
$\mathbf t$,
\begin{equation}\label{eq:mixedpolybound}
 \limsup_{\delta,\eta\to0}\sup_{p<r}
       \|Q_p(\mathbf t)\|_p
       \le \frac{s^{k+1}}{1-s}.
\end{equation}
This bound is uniform in the hierarchy depth.

We now extract the coefficient of $t_1\cdots t_k$ in \eqref{eq:mixedmoment} from this polynomial. Fix $0<\Delta<1/(4k)$.
A mixed difference kills every monomial missing
one of the parameters. Since $Q_p$ has total degree
at most $k$, only $t_1\cdots t_k$ survives, and its
normalized difference is one. Thus, writing
$|\mathbf e|=\sum_v e_v$, we have the exact identity
\begin{equation}\label{eq:mixedextraction}
 \frac1{\Delta^k}
   \sum_{\mathbf e\in\{0,1\}^k}
      (-1)^{k-|\mathbf e|}Q_p(\mathbf e \Delta)
   =(-1)^{k-1}
      \sum_T\mu_k^p(T)\,\Sp_T\Pi_pG_T^p.
\end{equation}
The sign comes from the degree-$k$ term in $Q_p$.

At each of the $2^k$ evaluation points,
$s\le k\Delta<1/4$, so \eqref{eq:mixedpolybound}
applies. The triangle inequality therefore gives
\[
 \limsup_{\delta,\eta\to0}\sup_{p<r}
 \left\|\sum_T\mu_k^p(T)\,\Sp_T\Pi_pG_T^p\right\|_p
 \le
 \frac{2^k}{\Delta^k}
       \frac{(k\Delta)^{k+1}}{1-k\Delta}
 =
 \frac{2^k k^{k+1}\Delta}{1-k\Delta}.
\]
Letting $\Delta\downarrow0$ after the fine-cut limit
proves \eqref{eq:testedsum}.
\end{proof}

This single mixed coefficient produces the $k$ replica
factors in \eqref{eq:mixedmoment}. The next section chooses
each $P_v$ to contain one factor at the site we want to
test and additional factors at sites assigned to graph edges.

\section{Weighted multioverlap tests}\label{sec:graphs}

We will now choose the factors $P_v$ so that these sums contain the future multioverlap $F_T^p$ multiplied by a function of the overlap matrix $q(T)$. Extra input sites produce pair moments of the magnetizations $b$, which approximate the pair overlaps directly. Their products give graph polynomials, and we then pass to continuous permutation-invariant tests. This is the role played by the parameter differentiations and overlap identities in \cite[Sections 8 and 9]{rRSB}.

\subsection{Graph polynomial tests}\label{sec:graphpoly}

Fix $k\ge2$. At a retained depth $p<r$, consider
terminal labels descending from a vertex $\beta\in\N^p$.
We apply Proposition \ref{prop:mixedcoeff} with one
input site, indexed by $0$, for the future multioverlap
$F_T^p$, and extra input sites for the pair moments
that will weight it.

Let $\Gamma$ be a graph on vertices $1,\ldots,k$.
Several edges may join the same two vertices, but no
edge joins a vertex to itself; this is a
\emph{loopless multigraph}. Write $E(\Gamma)$ for its
edges, counting repeated edges separately, and assign
each edge $e$ its own input site. At these sites use
the unshifted variable $b_e=2S_e-1$, and choose
\begin{equation}\label{eq:graphprobe}
 M_v=\prod_{e\ni v}b_e,\qquad P_v=S_0M_v,
 \qquad 1\le v\le k,
\end{equation}
where $P_v$ appeared in \eqref{eq:Pvmixed}.
The product is over the edges incident to $v$, with
an empty product equal to one. Each $P_v$ takes values
in $[-1,1]$ and is multiaffine in the input variables
$(S_0,(S_e)_e)$, since each occurs at most once in it. Thus these functions satisfy
the assumptions of Proposition \ref{prop:mixedcoeff}.
Its construction \eqref{eq:mixedprobe} uses one REM
clause with $|E(\Gamma)|+1$ input sites and a separate
cavity spin $\varsigma$, hence arity $|E(\Gamma)|+2$.
The factor $S_0$ produces the nonnegative moments used
in Section \ref{sec:recursion}, while the edge factors
produce unshifted pair moments.

For a symmetric $k\times k$ matrix $q$, define the graph polynomial
\begin{equation}\label{eq:graphpolynomial}
 \Psi_\Gamma(q)=\sum_{\pi\in\mathfrak S_k}
              \prod_{e=(v,w)\in E(\Gamma)}q_{\pi(v),\pi(w)}.
\end{equation}
The following holds.
\begin{proposition}[Graph polynomial tests]\label{prop:graphtests}
For every fixed loopless multigraph $\Gamma$ on
$\{1,\ldots,k\}$, as $\delta,\eta\to0$,
\begin{equation}\label{eq:graphqtests}
 \sup_{p<r}\Bigl\|
  \sum_T\mu_k^p(T)\Psi_\Gamma(q(T))\,
                       \Sp_T(\Pi_pF_T^p)\Bigr\|_p
                       \longrightarrow0,
\end{equation}
where $\Sp_T$ is the operator in \eqref{eq:SpT}.
\end{proposition}
\begin{proof}
Substituting \eqref{eq:graphprobe} into
\eqref{eq:mixedmoment} gives
\begin{equation}\label{eq:graphcoefficient}
 G_T^p=\int
     \Bigl(\prod_{\ell=1}^kS_0^{\alpha_\ell}\Bigr)
       \sum_{\pi\in\mathfrak S_k}
          \prod_{e=(v,w)\in E(\Gamma)}
                  b_e^{\alpha_{\pi(v)}}b_e^{\alpha_{\pi(w)}}
       \,dZ_T,
\end{equation}
where $dZ_T=dZ_T^0\prod_{e\in E(\Gamma)}dZ_T^e$
integrates the future variables at all input sites.
Every replica contributes a site-$0$ factor, while
each edge site occurs in the two replica factors
assigned to its endpoints. All site prefixes remain
as arguments of $G_T^p$.

Define
\[
 \E_{\mathrm{edge}}G_T^p
    =\int G_T^p\prod_{e\in E(\Gamma)}dZ_{\le p}^e,
\]
holding all state variables and $Z_{\le p}^0$ fixed.
By Section \ref{sec:kernels}, this average contracts
the norm and it commutes with the state averages
$\Pi_p$ and $\Sp_T$, so by \eqref{eq:testedsum},
\begin{align*}
 &\sup_{p<r}\Bigl\|
    \sum_T\mu_k^p(T)\,\Sp_T\Pi_p
                      \E_{\mathrm{edge}}G_T^p
                    \Bigr\|_p\,\, \leq\,\,
   \sup_{p<r}\Bigl\|
    \sum_T\mu_k^p(T)\,\Sp_T\Pi_pG_T^p
                    \Bigr\|_p
   \longrightarrow0.
\end{align*}
It remains to identify $ \E_{\mathrm{edge}}G_T^p$.
The future site variables were already integrated in
$G_T^p$, so $\E_{\mathrm{edge}}$ completes the full site
average at every edge site. Site $0$ still has its
prefix fixed. Independence across sites in
\eqref{eq:graphcoefficient} gives
\begin{equation}\label{eq:graphtag}
 \begin{aligned}
 \E_{\mathrm{edge}}G_T^p
     &=F_T^p\,W_\Gamma(\boldsymbol\alpha),
 \,\, \mbox{ where }   \,\, 
 W_\Gamma(\boldsymbol\alpha)
     &=\sum_{\pi\in\mathfrak S_k}
         \prod_{e=(v,w)\in E(\Gamma)}\E_Z
                 b^{\alpha_{\pi(v)}}b^{\alpha_{\pi(w)}}.
 \end{aligned}
\end{equation}
Here $F_T^p$ is the future multioverlap at site $0$.
In each pair moment, $\E_Z$ integrates the full union
of the two site paths, including the root, with all
state variables fixed. We suppress the edge-site
index because every site has the same representation.
The weight $W_\Gamma$ therefore still depends on the
state variables along the replica paths.

Substituting \eqref{eq:graphtag} into the norm estimate
above yields
\begin{equation}\label{eq:graphsmall}
 \sup_{p<r}\Bigl\|
    \sum_T\mu_k^p(T)\,\Sp_T\Pi_p
                 \bigl(F_T^pW_\Gamma(\boldsymbol\alpha)\bigr)
                    \Bigr\|_p\longrightarrow0.
\end{equation}
Each pair moment in \eqref{eq:graphtag} is the scalar
product of two cluster barycenters, by
\eqref{eq:heoverlaps}. For distinct terminal labels,
the overlap-bin geometry and the choice
\eqref{eq:qrepresentatives} bound the difference from
$q_{\alpha\wedge\alpha'}$ by $\delta$. When the
labels coincide, \eqref{eq:baryerror} bounds the
difference from $q_r=q_*$ by $\eta$. Thus
\begin{equation}\label{eq:pairapprox}
 \E_Zb^\alpha b^{\alpha'}
     =q_{\alpha\wedge\alpha'}+e_{\alpha,\alpha'},
 \qquad
 \sup_{\alpha,\alpha'}|e_{\alpha,\alpha'}|
       \le\max(\delta,\eta)
       \quad\text{almost surely}.
\end{equation}
This bound holds before averaging any state variables
and is uniform over all pairs, including repeated labels.

Replace the edge factors in $W_\Gamma$ one at a time
using \eqref{eq:pairapprox}. Each original and replacement
factor has absolute value at most one. There are
$|E(\Gamma)|$ factors in each of the $k!$ products, so
\[
 W_\Gamma(\boldsymbol\alpha)
     =\Psi_\Gamma(q(T))+e_{\Gamma,T},\qquad
 \sup_T\|e_{\Gamma,T}\|_\infty
     \le k!|E(\Gamma)|\max(\delta,\eta).
\]
Since $0\le F_T^p\le1$, the supremum-norm bound and
the contraction properties of $\Pi_p$ and $\Sp_T$ give
\[
 \sup_{p<r,T}
   \|\Sp_T\Pi_p(F_T^p e_{\Gamma,T})\|_p
       \le k!|E(\Gamma)|\max(\delta,\eta).
\]
The weights $\mu_k^p(T)$ sum to one, so the same bound
controls the error in their weighted sum. We may
therefore make this replacement in \eqref{eq:graphsmall}.
The new weight $\Psi_\Gamma(q(T))$ is deterministic
for each shape and can be taken outside both $\Pi_p$
and $\Sp_T$. This proves \eqref{eq:graphqtests}.
\end{proof}

\subsection{Continuous invariant tests}\label{sec:continuoustests}
Proposition \ref{prop:graphtests} gives \eqref{eq:graphqtests}
for every graph $\Gamma$. We now explain why these polynomials suffice
to test every continuous permutation-invariant function
of the overlap matrix. Every monomial in the off-diagonal entries has the form
\[
 \prod_{1\le v<w\le k}q_{v,w}^{n_{v,w}},
 \qquad n_{v,w}\in\{0,1,2,\ldots\}.
\]
Represent it by a graph with $n_{v,w}$ edges between
$v$ and $w$. Its sum over all permutations of the
replicas is exactly $\Psi_\Gamma$. Thus the graph
polynomials span all permutation-invariant polynomials
in the off-diagonal entries. The graph with no edges
gives the constant $\Psi_\Gamma=k!$. The diagonal
requires no tests, since its entries are fixed at $q_*$.

Let $\cU_k$ be the compact space of symmetric ultrametric
matrices with entries in $[0,q_*]$ and diagonal $q_*$,
modulo simultaneous permutations of rows and columns.
A continuous function on this space is a continuous
function of the matrix entries satisfying
$
 \psi\bigl((q_{\pi(v),\pi(w)})_{v,w\le k}\bigr)=\psi(q)$
 for all $\pi\in\mathfrak S_k.$
On the compact set of matrices before taking this
quotient, Stone-Weierstrass gives uniform approximation
by polynomials in the off-diagonal entries. If $P$
approximates $\psi$, then
\[
 P_{\rm sym}(q)=\frac1{k!}\sum_{\pi\in\mathfrak S_k}
                   P\bigl((q_{\pi(v),\pi(w)})_{v,w\le k}\bigr)
\]
is invariant and has no larger approximation error.
Since $P_{\rm sym}$ is a linear combination of the
$\Psi_\Gamma$, these are the tests we need.

These tests retain the arrangement of the overlaps in
the matrix, up to relabeling. An unordered list of pair
overlaps would contain less information: two-level
partitions of six replicas into groups of sizes $(3,3)$
and $(4,1,1)$ have the same numbers of within-group and
between-group pairs, but different shapes.

\begin{proposition}[Continuous shape tests]\label{prop:tests}
For every continuous permutation-invariant $\psi$,
\begin{equation}\label{eq:testedF}
 \sup_{p<r}\Bigl\|
     \sum_T\mu_k^p(T)\psi(q(T))\,\Sp_T(\Pi_pF_T^p)
                    \Bigr\|_p\longrightarrow0.
\end{equation}
If $\kappa_m\to0$ for all $m<k$, then
\begin{equation}\label{eq:testedF2}
 \sup_{p<r}\Bigl\|
     \sum_T\mu_k^p(T)\psi(q(T))\,\Pi_pH_T^p
                    \Bigr\|_p\longrightarrow0.
\end{equation}
\end{proposition}
\begin{proof}
Equation \eqref{eq:graphqtests} and linearity give
\eqref{eq:testedF} for invariant polynomials. Choose
such a polynomial $P$ with $\|\psi-P\|_\infty\le\varepsilon$.
Since $\sum_T\mu_k^p(T)=1$ and
$\|\Sp_T\Pi_pF_T^p\|_p\le1$, the triangle inequality gives
\[
 \Bigl\|\sum_T\mu_k^p(T)
        (\psi(q(T))-P(q(T)))\,\Sp_T(\Pi_pF_T^p)\Bigr\|_p
       \le\varepsilon.
\]
Take the fine-cut limit with $P$ fixed, and then let
$\varepsilon\downarrow0$. This proves the first assertion.

For the second assertion, recall that $H_T^p$ depends
only on state coordinates along the common stem, so
$\Sp_TH_T^p=H_T^p$. Using \eqref{eq:commute}, contraction
and the triangle inequality, followed by H\"older's
inequality over shapes, gives
\begin{align*}
 &\Bigl\|\sum_T\mu_k^p(T)\psi(q(T))
       \bigl(\Sp_T\Pi_pF_T^p-\Pi_pH_T^p\bigr)\Bigr\|_p
       \le\|\psi\|_\infty
          \sum_T\mu_k^p(T)\|F_T^p-H_T^p\|_p\\
 &\quad\le\|\psi\|_\infty
       \Bigl(\sum_T\mu_k^p(T)\|F_T^p-H_T^p\|_p^4\Bigr)^{1/4}
\le\|\psi\|_\infty
            \Bigl(k^4B_k\max_{m<k}\kappa_m\Bigr)^{1/4}.
\end{align*}
The last step is the branch replacement estimate
\eqref{eq:branchbound}. The induction hypothesis makes
this bound tend to zero, uniformly in $p$.
\end{proof}

We have obtained the tested sums of common-stem moments
announced at the end of Section \ref{sec:induction}.
To finish the induction, we must pass from these sums
to the average of the individual covariance norms.
Sections \ref{sec:compactness} and \ref{sec:closing}
perform that step; no further cavity equations are used.

\section{Compactness in the overlap variables}\label{sec:compactness}

The tests of Section \ref{sec:graphs} control averages
of the centered common-stem moments. To control the
average of the fourth powers of their covariance norms,
we must rule out cancellations on increasingly fine
sets of overlap matrices. We approximate these moments
by functions constant on a fixed finite partition of
the physical overlap space. The comparison uses one
fine-cut representation at a time, with both moments
indexed by the same physical overlap tree.

\subsection{A fixed space on which to compare the cuts}\label{sec:fixedoverlaps}

Fix $k\ge2$. Let $\cU_k^{\rm lab}$ be the compact space
of labeled symmetric ultrametric $k\times k$ matrices with
entries in $[0,q_*]$ and diagonal $q_*$. Its quotient by
simultaneous permutations of rows and columns is the space
$\cU_k$ used in Section \ref{sec:continuoustests}.
Let $\mu_k$ be the law of the physical overlap matrix
$R=(R_{\ell\ell'})_{\ell,\ell'\le k}$ under
$\E G^{\otimes k}$. This is a fixed permutation-invariant
probability measure on $\cU_k^{\rm lab}$. We retain the
replica labels when describing trees and their probabilities;
the tests and the coarse cells will be invariant under
relabeling.
At a given cut write
\begin{equation}
 d(x)=a\quad\text{if }x\in I_a,\qquad
 T(R)_{\ell\ell'}=d(R_{\ell\ell'})\quad(\ell\ne\ell'),
 \qquad T(R)_{\ell\ell}=r.
\label{eq:dxfound}
\end{equation}
Then $T(R)$ has the shape law $\mu_k^0$ of the original RPC, and
\begin{equation}\label{eq:physicalquantization}
 \max_{\ell,\ell'}|q(T(R))_{\ell\ell'}-R_{\ell\ell'}|
       \le\max(\delta,\eta)\qquad\mu_k\text{-almost surely},
\end{equation}
where the quantized overlap matrix $q(T)$ was defined in \eqref{eq:qmatrix}. Thus the different finite shape laws can all be read from
the same random matrix $R$. In defining $d$ outside
$\supp\zeta$, use the intervals of the retained cut
thresholds, so that $d$ is defined and nondecreasing
on all of $[0,q_*]$.

For each $p<r$, place the common-stem moments on one
canonical path $(U_0,Z_0),\ldots,(U_r,Z_r)$, as in
Section \ref{sec:branches}, and define
\begin{equation}\label{eq:physicalf}
 f_p(R)=
 \begin{cases}
  \Pi_pH_{T(R)}^p,&a(T(R))>p,\\
  0,&a(T(R))\le p.
 \end{cases}
\end{equation}
Coordinates beyond the common stem are unused. Here $R$
has law $\mu_k$, independently of the uniform variables.
The function $f_p(R)$ and its variable space may change
with the cut; the measure $\mu_k$ does not. We suppress
the cut from the notation.

Recall the two facts
$\mu_k^p(T)=\mu_k^0(T)$ when $a(T)>p$ and
$\Pi_pH_T^p=0$ when $a(T)=p$. They give, for every bounded
test $\psi$, the exact identities
\begin{align}
 \sum_T\mu_k^p(T)\psi(q(T))\Pi_pH_T^p
   &=\int\psi(q(T(R)))f_p(R)\,\mu_k(dR),
       \label{eq:fixedlawtests}\\
 \sum_T\mu_k^p(T)\|\Pi_pH_T^p\|_p^4
   &=\int\|f_p(R)\|_p^4\,\mu_k(dR).
       \label{eq:fixedlawenergy}
\end{align}
Under the induction hypothesis, the first quantity tends
to zero in norm for every fixed continuous invariant $\psi$,
by \eqref{eq:testedF2}. We need to show that the second
quantity tends to zero. Tested convergence alone does not
imply this: cancellation could occur on progressively finer
sets of overlap matrices as the cut changes.

We will construct a finite partition into fixed
invariant Borel cells on which the moments admit a
uniform approximation. Section \ref{sec:closing} will
then use the tests on these cells to prove
\eqref{eq:fixedlawenergy} tends to zero. The next
subsection defines the coarse partition of the overlap
range and states the approximation theorem.

\subsection{The coarse partition and the approximation theorem}\label{sec:cells}

Fix $0<\tau<1$. Choose a finite ordered partition
$\cP$ of $[0,q_*]$ such that every atom of $\zeta$ of
mass greater than $\tau$ is a singleton member, and
every other member $J$ is an interval with
$\zeta(J)\le2\tau$. After removing the finitely many
large atoms, such a partition can be obtained by splitting
each remaining interval into pieces of mass between
$\tau$ and $2\tau$, except possibly its last piece.
Indeed, cut when the accumulated mass first reaches
$\tau$; any jump has size at most $\tau$. Endpoints
may be included or excluded as needed to give a disjoint
partition.

The partition $\cP$ is chosen from $\zeta$ and $\tau$
and kept fixed as the fine cut changes. It is applied
to the physical overlaps. The fine-cut bins $I_a$ need
not refine $\cP$ or isolate its singleton atoms.
Recall $d(x)$ in \eqref{eq:dxfound} and, for each $J\in\cP$, set
\begin{equation}\label{eq:coarsedepth}
 a(J)=\min_{x\in J}d(x),\qquad
 \bar d_p(x)=\max\{p,a(J)\}\quad(x\in J).
\end{equation}
The minimum exists because $d$ takes finitely many
integer values. Both $d$ and $\bar d_p$ are nondecreasing,
and
\(
 \bar d_p(x)\le\max\{p,d(x)\}.
\)
For a singleton member $J=\{x_0\}$,
$\bar d_p(x_0)=\max\{p,d(x_0)\}$.

Define a new discrete shape and its centered moment by
\begin{equation}\label{eq:coarsemoment}
 \overline T_p(R)_{\ell\ell'}=
 \begin{cases}
  \bar d_p(R_{\ell\ell'}),&\ell\ne\ell',\\
  r,&\ell=\ell',
 \end{cases}
 \qquad
 \bar f_p(R)=\Pi_pH_{\overline T_p(R)}^p.
\end{equation}
Monotonicity of $\bar d_p$ preserves ultrametricity,
so this is a valid shape on the same fine-cut levels
$0,\ldots,r$. Its moment is evaluated using the same
representation $h$ in \eqref{eq:he} as $f_p$ in \eqref{eq:physicalf}. If its first split is $p$, its centered moment is zero.

The value of $\bar f_p(R)$ depends only on which
members of $\cP$ contain the off-diagonal entries of
$R$. These memberships give a fixed finite partition
of the overlap space. We will define its invariant
cells in Section \ref{sec:closing}, where they are used.

\begin{theorem}[Finite approximation of the common-stem moments]\label{thm:compact}
Let $\E$ average the canonical state and site variables.
For the partition and moments above,
\begin{equation}\label{eq:compact}
 \sup_{p<r}\int\E|f_p(R)-\bar f_p(R)|^2\,\mu_k(dR)
       \le C_k\tau
\end{equation}
at every finite cut, for a constant depending only on $k$.
\end{theorem}

To prove the theorem, we first express both moments
by a recursion indexed by the physical overlap tree.
We then establish a product-measure bound for its
overlap coordinates and a variation estimate for one
assigned integer depth. The proof in
Section \ref{sec:compactproof} combines these estimates
by changing one assigned depth at a time and integrating
the resulting errors.

\subsection{Physical overlap trees and the moment recursion}\label{sec:recursion}

\paragraph{The physical overlap tree.}
For a physical overlap matrix $R$, consider the
equivalence classes in $\{1,\ldots,k\}$ defined, for
$0\le q\le q_*$, by
\begin{equation}\label{eq:equiv}
 \ell\sim_q\ell'
 \quad\Longleftrightarrow\quad R_{\ell\ell'}\ge q.
\end{equation}
Let $\mathfrak t(R)$ consist of all distinct such
classes, together with the singleton sets $\{\ell\}$.
Ordered by inclusion, they form a rooted tree whose
root is $v_0=\{1,\ldots,k\}$. The children of a vertex
are its maximal proper subsets in this collection.
The singleton leaves record replica occurrences and
remain distinct even when the corresponding states
coincide. Each vertex $v$ is a collection of replica indices.
For convenience, we denote this collection by $D(v)$,
although formally $D(v)=v$.

Write $\cB(\mathfrak t)$ for the nonsingleton,
or internal, vertices of a tree $\mathfrak t$.
Every internal vertex has at least two children, so
$|\cB(\mathfrak t)|\le k-1$.

The tree $\mathfrak t(R)$ records the labeled groups
and their parent--child relations. Its numerical
overlaps are separate coordinates:
\[
 x_v=\min_{\substack{\ell,\ell'\in v\\\ell\ne\ell'}}
          R_{\ell\ell'}
       \qquad(v\in\cB(\mathfrak t(R))).
\]
They increase strictly along edges between internal
vertices. Overlaps at incomparable vertices may be
equal. An internal vertex at $q_*$ has only singleton
children and records repeated occurrences of one state.
Fix a possible labeled tree $\mathfrak t$ with $m$
internal vertices. The domain of its physical overlap
coordinates is
\begin{equation}\label{eq:physicaldomain}
 D_{\mathfrak t}
 =\bigl\{\boldsymbol x\in[0,q_*]^{\cB(\mathfrak t)}:
       x_v<x_i\text{ whenever }i\text{ is an internal
       child of }v\bigr\}.
\end{equation}
For $\boldsymbol x\in D_{\mathfrak t}$, let
$R_{\mathfrak t}(\boldsymbol x)$ be the matrix with
diagonal $q_*$ and
\[
 R_{\mathfrak t}(\boldsymbol x)_{\ell\ell'}
       =x_{v(\ell,\ell')}\quad(\ell\ne\ell'),
\]
where $v(\ell,\ell')$ is the smallest vertex containing
both replica labels. This reconstructs $R$ from
$(\mathfrak t(R),\boldsymbol x(R))$. In particular,
fixing $\mathfrak t$ fixes only the tree structure;
the numerical overlaps still range over $D_{\mathfrak t}$.

Define their subprobability law by
\begin{equation}\label{eq:physicaltreelaw}
 \nu_{\mathfrak t}(B)
 =\mu_k\{R:\mathfrak t(R)=\mathfrak t,
                    \ \boldsymbol x(R)\in B\}.
\end{equation}
Its total mass is $\mu_k\{\mathfrak t(R)=\mathfrak t\}$;
we do not normalize by this probability. Thus, for
every nonnegative Borel function $F$,
\begin{equation}\label{eq:treecoordinates}
 \int_{\{\mathfrak t(R)=\mathfrak t\}}F(R)\,\mu_k(dR)
 =\int_{D_{\mathfrak t}}
       F(R_{\mathfrak t}(\boldsymbol x))\,
                         \nu_{\mathfrak t}(d\boldsymbol x).
\end{equation}

\paragraph{Discrete depth assignments and their moments.}
Let us fix the fine cut, its representation $h$, and $p<r$.
For a fixed physical overlap tree $\mathfrak t$, assign integer
depths
$\boldsymbol\ell=(\ell_v)_{v\in\cB(\mathfrak t)}$
to its internal vertices. Each singleton leaf has
assigned depth $r$. We allow all assignments satisfying
\begin{equation}\label{eq:weakdepths}
 p\le\ell_{v_0},\qquad
 \ell_v\le\ell_i\le r
       \quad\text{when }i\text{ is a child of }v.
\end{equation}
These integer depths may be equal for a vertex and
an internal child, while their physical overlaps in
\eqref{eq:physicaldomain} are strictly ordered.

On the canonical path put
\[
 \cF_a=\sigma(U_{\le a},Z_{\le a}),\qquad
 \cG_a^p=\sigma(U_{\le a},Z_{\le p})\quad(a\ge p),
 \qquad S=\frac{1+h(U_{\le r},Z_{\le r})}{2}.
\]
All expectations below use the product law of these
uniform variables in the fixed fine-cut representation.
Set $Q_{\{\ell\}}=S$ at each singleton leaf and define
the other functions recursively from the leaves towards the root:
\begin{equation}\label{eq:Qrec}
 Q_v=\prod_{i\text{ child of }v}\E[Q_i\mid\cF_{\ell_v}],
 \qquad
 H_{\mathfrak t}^p(\boldsymbol\ell)
       =\E[Q_{v_0}\mid\cG_{\ell_{v_0}}^p].
\end{equation}
The dependence of $Q_v$ on the depth assignment
$\boldsymbol\ell$ is suppressed. Each $Q_v$ is $\cF_{\ell_v}$-measurable
and lies in $[0,1]$. The tree $\mathfrak t$ specifies
which child factors multiply; the assigned depths
specify the conditional expectations in the existing
filtration. The last expectation averages the site
coordinates between $p$ and $\ell_{v_0}$, retaining
$(U_{\le\ell_{v_0}},Z_{\le p})$.

For $R=R_{\mathfrak t}(\boldsymbol x)$, recall
$d(x)$ in \eqref{eq:dxfound} and $\bar d_p(x)$
in \eqref{eq:coarsedepth}. The two assignments
needed in Theorem \ref{thm:compact} are
\begin{equation}\label{eq:twodepthassignments}
 d_v=\max\{p,d(x_v)\},\qquad
 \bar d_v=\bar d_p(x_v)
       \quad(v\in\cB(\mathfrak t)).
\end{equation}
Both satisfy \eqref{eq:weakdepths}, and $\bar d_v\le d_v$.
The depth assigned to the common ancestor of replicas
$i$ and $j$ is $\max\{p,d(R_{ij})\}$ under the first
assignment and $\bar d_p(R_{ij})$ under the second.
In both cases, the recursion \eqref{eq:Qrec} is indexed by all the
physical vertices of $\mathfrak t$, so vertices
with equal assigned depths are not identified.

We check that \eqref{eq:Qrec} gives the corresponding
common-stem moments. At distinct assigned depths,
this follows by conditional independence of the child
branches, as in \eqref{eq:childfeatures}--\eqref{eq:HT}.
If an internal child $i$ and its parent both have
assigned depth $a$, then
\[
 \E[Q_i\mid\cF_a]=Q_i
   =\prod_{j\text{ child of }i}\E[Q_j\mid\cF_a].
\]
Thus grouping factors through vertices at the same
assigned depth does not change the product.
At depth $r$, the recursion gives $Q_v=S^{|D(v)|}$,
as required for repeated terminal labels. Hence
\begin{equation}\label{eq:twomomentassignments}
 f_p(R)=\Pi_pH_{\mathfrak t}^p(\boldsymbol d),
 \qquad
 \bar f_p(R)=\Pi_pH_{\mathfrak t}^p(\bar{\boldsymbol d}).
\end{equation}
Indeed, if the first split of $T(R)$ exceeds
$p$, then
$H_{\mathfrak t}^p(\boldsymbol d)=H_{T(R)}^p$.
Otherwise the root has assigned depth $p$, so
$H_{\mathfrak t}^p(\boldsymbol d)$ depends only on
$(U_{\le p},Z_{\le p})$ and is killed by $\Pi_p$.
For the second assignment,
$H_{\mathfrak t}^p(\bar{\boldsymbol d})
=H_{\overline T_p(R)}^p$.

\subsection{The law of the physical overlap coordinates}\label{sec:treedomination}

To average errors over the physical matrices, we need
a bound for the coordinate laws $\nu_{\mathfrak t}$
defined in \eqref{eq:physicaltreelaw}. 

\begin{lemma}[Product domination]\label{lem:physicaldomination}
For every physical tree $\mathfrak t$ with $m$ internal
vertices, its coordinate law \eqref{eq:physicaltreelaw} satisfies
\begin{equation}\label{eq:physicaldomination}
 \nu_{\mathfrak t}\le\zeta^{\otimes m}.
\end{equation}
Equivalently, for every nonnegative Borel function $g$,
\[
\int_{D_{\mathfrak t}}
       g(R_{\mathfrak t}(\boldsymbol x))\,
                         \nu_{\mathfrak t}(d\boldsymbol x)
 \le \int_{D_{\mathfrak t}}
       g(R_{\mathfrak t}(\boldsymbol x))\,
                         \zeta^{\otimes m}(d\boldsymbol x).
\]
\end{lemma}
\begin{proof}
We use induction on the number $k$ of replicas.
For $k=2$, there is one internal vertex, whose
coordinate is $R_{12}$, so its law is exactly $\zeta$.

Suppose $k\ge3$ and fix a physical tree $\mathfrak t$.
Choose an internal vertex $v$ all of whose children
are singleton leaves, and let $b\ge2$ be their number.
Such a vertex exists by following internal children
until there are none. Delete one of its leaves,
which by exchangeability we may label $k$. Let
$R^-=(R_{ij})_{i,j\le k-1}$ and let $\mathfrak t^-$
be the tree induced by $\mathfrak t$ on the remaining
replica labels. Write
\[
 E=\{\mathfrak t(R)=\mathfrak t\},
 \qquad
 E^-=\{\mathfrak t(R^-)=\mathfrak t^-\}.
\]
Fix a nonnegative bounded Borel function $g$ on
$[0,q_*]^m$.

If $b\ge3$, all $m$ internal vertices and their
overlap coordinates survive the deletion. Write
$\boldsymbol x'$ for these coordinates. Under their
natural identification, $D_{\mathfrak t^-}=D_{\mathfrak t}$.
Since $E\subset E^-$, the definition \eqref{eq:treecoordinates}, marginal consistency and the induction hypothesis give
\begin{align*}
 \int g\,d\nu_{\mathfrak t}
 &=\int \mathbf1_E\,g(\boldsymbol x')\,\mu_k(dR)\\
 &\le\int \mathbf1_{E^-}\,g(\boldsymbol x')\,
                                      \mu_{k-1}(dR^-)
 =\int g\,d\nu_{\mathfrak t^-}
 \le\int_{D_{\mathfrak t}}g\,d\zeta^{\otimes m}.
\end{align*}

If $b=2$, label the other child of $v$ by $1$.
Deleting replica $k$ removes only the internal vertex
$v$ and leaves replica $1$ in its place. Write
$\boldsymbol x'$ for the $m-1$ surviving coordinates
and $x=x_v=R_{1k}$ for the removed coordinate. Set
\(
 M=\max_{2\le j\le k-1}R_{1j}.
\)
We have the event identity
\[
 \mathbf1_E
 =\mathbf1_{E^-}\mathbf1_{\{R_{1k}>M\}}
 \qquad \mu_k\text{-almost surely}.
\]
Indeed, on $E$, every $R_{1j}$ in the maximum occurs
at a proper ancestor of $v$, so $R_{1k}>M$.
Conversely, if $R_{1k}>M$, ultrametricity forces
$R_{jk}=R_{j1}$ for all $2\le j\le k-1$. Thus,
on $E^-$, replica $k$ attaches as the sibling of
replica $1$, giving exactly the tree $\mathfrak t$.

On $E^-$, the maximum $M$ is the surviving coordinate
of the parent of $v$, so we also write
$M=M(\boldsymbol x')$. Consequently,
\[
 D_{\mathfrak t}
 =\bigl\{(\boldsymbol x',x):
       \boldsymbol x'\in D_{\mathfrak t^-},\
       M(\boldsymbol x')<x\le q_*\bigr\}.
\]
Applying the Ghirlanda--Guerra identities \eqref{eq:GG}, all empirical terms vanish under the restriction $R_{1k}>M$, since every
$R_{1j}\le M$. Hence
\begin{align*}
 \int g\,d\nu_{\mathfrak t}
 &=\int \mathbf1_{E^-}\mathbf1_{\{R_{1k}>M\}}\,
       g(\boldsymbol x',R_{1k})\,\mu_k(dR)\\
 &=\frac1{k-1}
       \int_{D_{\mathfrak t^-}}
       \int_{(M(\boldsymbol x'),q_*]}
       g(\boldsymbol x',x)\,\zeta(dx)\,
                         \nu_{\mathfrak t^-}(d\boldsymbol x')\\
 &\le\frac1{k-1}
       \int_{D_{\mathfrak t^-}}
       \int_{(M(\boldsymbol x'),q_*]}
       g(\boldsymbol x',x)\,\zeta(dx)\,
                         \zeta^{\otimes(m-1)}(d\boldsymbol x')\\
 &=\frac1{k-1}
       \int_{D_{\mathfrak t}}g\,d\zeta^{\otimes m}
 \le\int_{D_{\mathfrak t}}g\,d\zeta^{\otimes m}.
\end{align*}
The inequality uses the induction hypothesis for
the surviving coordinates. This completes the proof.
\end{proof}

\subsection{Variation in one assigned depth}\label{sec:depthvariation}

Recall the coarse partition $\cP$ from
Section \ref{sec:cells}. When $x_v\in J\in\cP$, the
coarsening replaces the assigned depth
$\max\{p,d(x_v)\}$ by the fixed depth
$\max\{p,a(J)\}$. At a heavy atom isolated as a
singleton cell, these assignments agree. Every other
cell has $\zeta$-mass at most $2\tau$.

Fix a labeled physical tree $\mathfrak t$, a fine-cut
representation, $p<r$, and an assignment satisfying
\eqref{eq:weakdepths}. Hold all assigned depths fixed
except at one internal vertex $v$ with $b$ children,
and write $H(d)$ for
$H_{\mathfrak t}^p(\boldsymbol\ell)$ when $\ell_v=d$.
The admissible range is $p\le d\le d^+$ for the root
and $c\le d\le d^+$ otherwise, where $c$ is the
parent's assigned depth and $d^+$ is the smallest
assigned child depth. Let $[\alpha_j,e_j]$ be ordered
intervals in this range, with integer endpoints and
$e_j\le\alpha_{j+1}$. Let $\varpi_{j,d}\ge0$ be
deterministic weights supported on their integer
depths, with $\sum_d\varpi_{j,d}\le u$ for each $j$.
In the next Section \ref{sec:compactproof}, the intervals will correspond to the non-heavy cells of $\cP$, and the weights are their
overlap masses grouped by the original assigned depth.

\begin{lemma}[Variation in one assigned depth]\label{lem:depthvariation}
With the notation above,
\begin{equation}\label{eq:depthvariation}
 \sum_j\sum_d\varpi_{j,d}\,
       \E|H(d)-H(\alpha_j)|^2\le\chi_b u,
 \qquad
 \chi_b=\frac{4^b b^{2b}}{(b!)^2}.
\end{equation}
The bound is independent of the cut, $p$, and the
number of intervals.
\end{lemma}
\begin{proof}
We use the following elementary estimate. If $X_d$ is
a submartingale with respect to a filtration $(\cH_d)_d$
and $0\le X_d\le L$, then for $\alpha\le d$,
\begin{align*}
 \E|X_d-X_\alpha|^2
 &=\E X_d^2-\E X_\alpha^2
   -2\E\bigl[X_\alpha
       (\E[X_d\mid\cH_\alpha]-X_\alpha)\bigr]
\le\E X_d^2-\E X_\alpha^2.
\end{align*}
Moreover $v_d=\E X_d^2$ is nondecreasing and belongs to
$[0,L^2]$. Consequently
\begin{equation}\label{eq:doob}
 \sum_j\sum_d\varpi_{j,d}\E|X_d-X_{\alpha_j}|^2
 \le u\sum_j(v_{e_j}-v_{\alpha_j})\le uL^2.
\end{equation}
The same proof applies to bounded nonnegative
nondecreasing functions, pointwise. Shared endpoints
cause no problem, since the increments still telescope.

Let $Q_1,\ldots,Q_b$ be the fixed child moments of $v$.
For $J\subseteq\{1,\ldots,b\}$ put
\[
 X_d^J=\Bigl(\sum_{i\in J}\E[Q_i\mid\cF_d]\Bigr)^b.
\]
The sum inside the power is a nonnegative martingale.
Thus $X^J$ is a submartingale, by convexity, and
$0\le X_d^J\le b^b$. Polarization gives
\begin{equation}\label{eq:polarization}
 \prod_{i=1}^b y_i
  =\frac1{b!}\sum_{J\subseteq\{1,\ldots,b\}}
        (-1)^{b-|J|}\Bigl(\sum_{i\in J}y_i\Bigr)^b.
\end{equation}
To check the identity, expand the powers: inclusion and
exclusion cancels every monomial missing one of the
variables, leaving $b!y_1\cdots y_b$.

If $v$ is the root, \eqref{eq:Qrec} and
\eqref{eq:polarization} express $H(d)$ as the same
linear combination of
$\widetilde X_d^J=\E[X_d^J\mid\cG_d^p]$.
These are again submartingales. In fact, for $d<d'$,
\[
 \E[\widetilde X_{d'}^J\mid\cG_d^p]
 =\E[\E[X_{d'}^J\mid\cF_d]\mid\cG_d^p]
 \ge\E[X_d^J\mid\cG_d^p]=\widetilde X_d^J.
\]

If $v$ is not the root, first average at its parent depth
$c$. The functions
\[
 G_d^J=\E[X_d^J\mid\cF_c],\qquad c\le d\le d^+,
\]
are nondecreasing in $d$, by the submartingale property.
To pass from the moment at this parent to $H(d)$,
\eqref{eq:Qrec} successively multiplies by fixed sibling
moments in $[0,1]$ and takes conditional expectations
at the ancestors, followed by the site average on the
common stem. These operations form a fixed positive
linear map $\mathcal L$ with $\mathcal L1\le1$. Thus
\[
 H(d)=\frac1{b!}\sum_J(-1)^{b-|J|}\mathcal L G_d^J,
\]
and each $\mathcal L G_d^J$ is nondecreasing and lies in
$[0,b^b]$. The map is independent of $d$ because all
sibling subtrees and ancestor depths have been fixed.
This remains true when an ancestor has the same depth
as its child: the corresponding conditional expectation
then acts as the identity.

In either case, Cauchy--Schwarz over the $2^b$ subsets
in \eqref{eq:polarization}, followed by \eqref{eq:doob}
for each term with $L=b^b$, proves
\eqref{eq:depthvariation}.
\end{proof}

The positivity of $S$ makes the powers of the conditional
moments submartingales and the sibling multipliers
positive. The estimate controls the weighted squared
changes over all intervals together. It requires no
continuity of $h$ or of its conditional moments in
the assigned depth.

\subsection{Proof of the approximation theorem}\label{sec:compactproof}
\begin{proof}[Proof of Theorem \ref{thm:compact}]
Fix the fine cut, its representation, $p<r$, and
the coarse partition $\cP$ from Section \ref{sec:cells}.
We first compare the two assignments for each fixed
physical matrix, then integrate the error over its
physical overlap coordinates.

\emph{Changing the assigned depths.}
Fix a physical tree $\mathfrak t$ with $m\le k-1$
internal vertices and a vector
$\boldsymbol x\in D_{\mathfrak t}$. Recall the two
assignments
\[
 d_v=\max\{p,d(x_v)\},
 \qquad
 \bar d_v=\bar d_p(x_v)
 \qquad(v\in\cB(\mathfrak t)).
\]
Choose an ordering $v_1,\ldots,v_m$ of the internal
vertices, depending only on $\mathfrak t$, in which
every parent precedes its children. Define
\[
 d_u^{(j)}=
 \begin{cases}
  \bar d_u,&u\in\{v_1,\ldots,v_j\},\\
  d_u,&u\notin\{v_1,\ldots,v_j\},
 \end{cases}
 \qquad j=0,\ldots,m.
\]
Thus $\boldsymbol d^{(0)}=\boldsymbol d$ and
$\boldsymbol d^{(m)}=\bar{\boldsymbol d}$.
The physical tree and $\boldsymbol x$ stay fixed;
at each step we replace one assigned integer depth.
All physical vertex labels are retained, including
when assigned depths coincide.

Every intermediate assignment satisfies
\eqref{eq:weakdepths}. Indeed, at the step for a
nonroot vertex $v$, its parent already has its coarse
depth and its children still have their original depths.
For each internal child $i$,
\[
 \bar d_{\mathrm{parent}(v)}
       \le\bar d_v\le d_v\le d_i.
\]
For the root the lower bound is $p$, and singleton
leaves keep depth $r$.

Write $b_v$ for the number of children of $v$, and set
\[
 \Delta_{v_j}(\boldsymbol x)
   =H_{\mathfrak t}^p(\boldsymbol d^{(j-1)})
      -H_{\mathfrak t}^p(\boldsymbol d^{(j)}).
\]
By \eqref{eq:twomomentassignments},
\[
 f_p(R_{\mathfrak t}(\boldsymbol x))
       -\bar f_p(R_{\mathfrak t}(\boldsymbol x))
 =\Pi_p\sum_{v\in\cB(\mathfrak t)}
                         \Delta_v(\boldsymbol x).
\]
The $L^2$ contraction of $\Pi_p$ and Cauchy--Schwarz give
\begin{equation}\label{eq:coarsetelescope}
 \E\bigl|(f_p-\bar f_p)(R_{\mathfrak t}(\boldsymbol x))\bigr|^2
 \le m\sum_v\E|\Delta_v(\boldsymbol x)|^2.
\end{equation}
Here and below $\E$ averages the canonical uniform
variables in the fixed representation.

\emph{Passing to the product measure.}
The contribution of this tree structure to the error
in \eqref{eq:compact} is
\[
 \mathcal E_{\mathfrak t}
 :=\int_{\{\mathfrak t(R)=\mathfrak t\}}
       \E|f_p(R)-\bar f_p(R)|^2\,\mu_k(dR).
\]
Using the coordinate identity \eqref{eq:treecoordinates},
the pointwise bound \eqref{eq:coarsetelescope}, and
then Lemma \ref{lem:physicaldomination}, we obtain
\begin{align}
 \mathcal E_{\mathfrak t}
 &=\int_{D_{\mathfrak t}}
       \E\bigl|(f_p-\bar f_p)(R_{\mathfrak t}(\boldsymbol x))\bigr|^2
                    \,\nu_{\mathfrak t}(d\boldsymbol x)
                    \notag\\
 &\le m\sum_v\int_{D_{\mathfrak t}}
       \E|\Delta_v(\boldsymbol x)|^2\,
                         \nu_{\mathfrak t}(d\boldsymbol x)
\le m\sum_v\int_{D_{\mathfrak t}}
       \E|\Delta_v(\boldsymbol x)|^2\,
                         \zeta^{\otimes m}(d\boldsymbol x).
                         \label{eq:treeerrorproduct}
\end{align}
The last step applies measure domination to a
nonnegative integrand.

\emph{Integrating one physical overlap.}
For a vertex $v$, fix the other physical coordinates
$\boldsymbol x_{-v}$ and write
\[
 A_v(\boldsymbol x_{-v})
 :=\{x\in[0,q_*]:
              (\boldsymbol x_{-v},x)\in D_{\mathfrak t}\}.
\]
Outside the projection of $D_{\mathfrak t}$ onto the
remaining coordinates, this section is empty and
the inner integral below is zero. Thus Tonelli's
theorem gives
\begin{align}
 &\int_{D_{\mathfrak t}}
       \E|\Delta_v(\boldsymbol x)|^2\,
                         \zeta^{\otimes m}(d\boldsymbol x)
                         \notag\\
 &\quad=\int_{[0,q_*]^{m-1}}
       \left(\int_{A_v(\boldsymbol x_{-v})}
          \E|\Delta_v(\boldsymbol x_{-v},x)|^2\,\zeta(dx)\right)
                    \zeta^{\otimes(m-1)}(d\boldsymbol x_{-v}).
                         \label{eq:physicalslice}
\end{align}
We now bound the inner integral uniformly in
$\boldsymbol x_{-v}$.

Fix a nonempty section and abbreviate it by $A_v$.
First take a single $x\in A_v$, and let $J\in\cP$
and $I_a$ be the coarse and fine intervals containing
$x$. Recall that
\[
 d(x)=a,
 \qquad
 a(J)=\min_{y\in J}d(y)
     =\min\{a':J\cap I_{a'}\ne\emptyset\}.
\]
The original and coarse depths assigned to $v$ are
therefore
\[
 d_v=\max\{p,a\},
 \qquad
 \bar d_v=\max\{p,a(J)\}=:\alpha_J.
\]
Thus every original depth arising from the coarse
interval $J$ is replaced by the same depth $\alpha_J$.
This destination is determined by the whole interval
$J$, independently of the section $A_v$.

At the step for $v$, all other assigned depths are
fixed once $\boldsymbol x_{-v}$ is fixed, because
the order of the moves depends only on $\mathfrak t$.
Write $H(d)$ for the resulting moment when $v$ has
assigned depth $d$ and these other depths are held
fixed. The same function $H$ is used for every
$x\in A_v$, and
\[
 \Delta_v(\boldsymbol x_{-v},x)
 =H(\max\{p,a\})-H(\alpha_J)
 \qquad(x\in J\cap I_a\cap A_v).
\]

If $J=\{x_0\}$ is an isolated heavy atom, then
$a(J)=d(x_0)$, so the original and coarse assigned
depths agree and its contribution is zero.
In the following sums, $J$ ranges over the remaining
members of $\cP$ that meet $A_v$; each satisfies
$\zeta(J)\le2\tau$.

For a fixed such $J$, only the fine intervals meeting
$J\cap A_v$ contribute to the integral. The largest
starting depth occurring in this section is
\[
 e_J:=\max_{x\in J\cap A_v}\max\{p,d(x)\}.
\]
The maximum exists because $d$ has finite range.
Thus $\alpha_J$ is determined by the whole interval
$J$, whereas $e_J$ is determined by $J\cap A_v$.
The depth $\alpha_J$ need not itself occur as an
original assigned depth in this section.

On each piece $J\cap I_a\cap A_v$, the difference
$\Delta_v$ is constant as a function of $x$.
Grouping these pieces by their starting depth
$d=\max\{p,a\}$, define
\[
 \varpi_{J,d}
 :=\sum_{\substack{0\le a\le r\\\max\{p,a\}=d}}
       \zeta(J\cap I_a\cap A_v),
 \qquad \alpha_J\le d\le e_J.
\]
In particular, fine intervals with $a\le p$ all
contribute to the weight at $d=p$. Splitting the
integral into these pieces gives
\[
 \int_{J\cap A_v}
       \E|\Delta_v(\boldsymbol x_{-v},x)|^2\,\zeta(dx)
 =\sum_{d=\alpha_J}^{e_J}
       \varpi_{J,d}\,\E|H(d)-H(\alpha_J)|^2.
\]
The total weight for this coarse interval is small:
\[
 \sum_{d=\alpha_J}^{e_J}\varpi_{J,d}
 =\zeta(J\cap A_v)\le\zeta(J)\le2\tau.
\]
We check that Lemma \ref{lem:depthvariation} applies
to these changes of depth. Let $c$ be the current
parent depth, or $c=p$ if $v$ is the root, and let
$d^+$ be the smallest current child depth, with
singleton leaves assigned depth $r$.
For a nonroot vertex, the parent's physical overlap
belongs to the same or an earlier coarse interval
than $J$. Since the parent has already received its
coarse assignment, $c\le\alpha_J$; for the root,
this follows from $c=p$.
The children still have their original assignments,
and every $x\in A_v$ lies below the physical overlaps
of the internal children. Consequently,
\[
 c\le\alpha_J\le\max\{p,d(x)\}\le e_J\le d^+
 \qquad(x\in J\cap A_v).
\]
Thus $[\alpha_J,e_J]$ lies in the admissible range
for $H(d)$. Moreover, if $J$ precedes $J'$, every
point of $J$ lies below every point of $J'$.
Monotonicity of $d$ therefore gives
\[
 e_J\le\max\{p,a(J')\}=\alpha_{J'}.
\]
The intervals are ordered as required by the lemma;
neighboring intervals may share an endpoint.

Applying Lemma \ref{lem:depthvariation} to all these
intervals together, with $u=2\tau$, gives
\begin{align}
 \int_{A_v}
       \E|\Delta_v(\boldsymbol x_{-v},x)|^2\,\zeta(dx)
 &=\sum_J\sum_{d=\alpha_J}^{e_J}
       \varpi_{J,d}\,\E|H(d)-H(\alpha_J)|^2
       \notag\\
 &\le2\chi_{b_v}\tau.
       \label{eq:coordinatecoarsening}
\end{align}
The lemma bounds the entire sum over $J$, so there
is no factor depending on the number of coarse
intervals or fine-cut levels.

This bound is uniform in $\boldsymbol x_{-v}$.
Since the outer product measure in
\eqref{eq:physicalslice} has total mass one,
substituting \eqref{eq:coordinatecoarsening} into
\eqref{eq:physicalslice} and then
\eqref{eq:treeerrorproduct} gives
\[
 \mathcal E_{\mathfrak t}
 \le2m\sum_{v\in\cB(\mathfrak t)}\chi_{b_v}\tau
 \le2m^2\max_{2\le b\le k}\chi_b\,\tau.
\]
There are finitely many labeled physical trees on
$k$ replica leaves, and each has $m\le k-1$ internal
vertices. Their events partition the overlap space.
Summing proves \eqref{eq:compact}, uniformly in $p$
and in the fine cut.
\end{proof}

\section{Completion of the induction}\label{sec:closing}

We prove the remaining induction target by combining the tests of Section \ref{sec:graphs} with the finite approximation of Section \ref{sec:compactness}. We first reduce the target to finitely many cell averages, then use the fixed physical overlap law to show that each of these averages tends to zero.

\begin{proposition}[Completion of the induction]\label{prop:closing}
Fix $k\ge2$ and suppose that $\kappa_m\to0$ for every
$1\le m<k$. Then, as $\delta,\eta\to0$,
\begin{equation}\label{eq:targetproved}
 \sup_{p<r}\sum_T\mu_k^p(T)\|\Pi_pH_T^p\|_p^4
       \longrightarrow0.
\end{equation}
Consequently, $\kappa_k\to0$.
\end{proposition}
\begin{proof}
Recall that $\mu_k$ is the fixed law of the physical
overlap matrix $R$, and $f_p(R)$ is the centered
common-stem moment defined in \eqref{eq:physicalf}.
By \eqref{eq:fixedlawenergy}, it suffices to prove
\[
 \sup_{p<r}\int\|f_p(R)\|_p^4\,\mu_k(dR)
       \longrightarrow0.
\]
We will approximate $f_p$ by its averages on a fixed
finite partition of the overlap space, using the coarse
partition $\cP$ from Section \ref{sec:cells}.

Fix $0<\tau<1$ and its corresponding coarse partition $\cP$.
Write $J(x)$ for the member containing $x$.
Group the finitely many membership matrices
$(J(R_{\ell\ell'}))_{\ell<\ell'}$ by permutations
of the replica labels, and let $E_c$ be the corresponding
inverse images in $\cU_k^{\rm lab}$. These are invariant
Borel sets, and the moments $H_T^p$ are unchanged by
relabeling, so $\bar f_p$ is constant on each $E_c$.
The cells and their masses depend only on $k,\zeta,\tau$;
the values of $\bar f_p$ may depend on the fine cut
and on $p$. Discard cells of zero $\mu_k$-mass, write
$\mu_c=\mu_k(E_c)>0$, and define
\begin{equation}\label{eq:cellaverage}
 g_{p,c}=\frac1{\mu_c}\int_{E_c}f_p(R)\,\mu_k(dR),
 \qquad g_p(R)=g_{p,c}\quad(R\in E_c).
\end{equation}
This average integrates only the overlap matrix $R$;
the canonical state and site variables are held fixed.
The cells and the positive denominators $\mu_c$ are
fixed before the fine-cut limit, although $f_p$ and
the cell averages may change with the cut and with $p$.

For each realization of the canonical variables,
$g_p$ is the conditional expectation of $f_p$ with
respect to the cell partition in the $R$ variable.
It is therefore the best $L^2(\mu_k)$ approximation
by a cell-constant function. The coarsened moment
$\bar f_p$ is constant on the same cells, so Theorem
\ref{thm:compact} gives
\begin{equation}\label{eq:cellprojection}
 \sup_{p<r}\int\E|f_p(R)-g_p(R)|^2\,\mu_k(dR)
 \le \sup_{p<r}\int\E|f_p(R)-\bar f_p(R)|^2\,\mu_k(dR)
 \le C_k\tau.
\end{equation}
Here $\E$ averages the canonical state and site variables.
Since $|f_p|,|g_p|\le1$, their difference is bounded
by two. The comparison \eqref{eq:normconversion} yields
\[
 \sup_{p<r}\int\|f_p(R)-g_p(R)\|_p^4\,\mu_k(dR)
       \le4C_k\tau.
\]
The triangle inequality and $(x+y)^4\le8(x^4+y^4)$
then give
\begin{equation}\label{eq:energy}
 \sup_{p<r}\int\|f_p(R)\|_p^4\,\mu_k(dR)
 \le32C_k\tau
       +8\sum_c\mu_c\sup_{p<r}\|g_{p,c}\|_p^4.
\end{equation}
Thus, for fixed $\tau$, it remains to show that the
finitely many cell averages $g_{p,c}$ tend to zero
in norm, uniformly in $p$.

Proposition \ref{prop:tests}, under the induction
hypothesis, and \eqref{eq:fixedlawtests} give
\[
 \sup_{p<r}\Bigl\|\int\psi(q(T(R)))f_p(R)\,\mu_k(dR)
                  \Bigr\|_p\longrightarrow0
\]
for every fixed continuous invariant $\psi$.
By \eqref{eq:physicalquantization} and uniform
continuity, we may replace $\psi(q(T(R)))$ by
$\psi(R)$: since $|f_p(R)|\le1$, the error in norm
is bounded by the uniform modulus of continuity
of $\psi$ at $\max(\delta,\eta)$. Hence
\begin{equation}\label{eq:weakH}
 \sup_{p<r}\Bigl\|\int\psi(R)f_p(R)\,\mu_k(dR)
                  \Bigr\|_p\longrightarrow0.
\end{equation}
Here the benefit of using the physical overlap matrix
is explicit: both the measure $\mu_k$ and the test
$\psi(R)$ are independent of the fine cut. We can
therefore approximate each fixed cell indicator once
in this measure and use that approximation for all cuts.

More precisely, every bounded Borel invariant test
$\psi$ can be approximated in $L^1(\mu_k)$ by continuous
functions. Averaging the approximations over replica
permutations makes them invariant without increasing
their error, because $\mu_k$ is permutation invariant.
For such an approximation $\varphi$, the triangle
inequality and $|f_p(R)|\le1$ give
\[
 \Bigl\|\int(\psi-\varphi)(R)f_p(R)\,\mu_k(dR)\Bigr\|_p
       \le\int|\psi-\varphi|\,d\mu_k,
\]
uniformly in the cut and in $p$. First taking the
fine-cut limit with $\varphi$ fixed, then sending its
$L^1$ approximation error to zero, extends
\eqref{eq:weakH} to $\psi$. In particular, it applies
to $\psi=\one_{E_c}$, including cells containing atoms.
Dividing by the fixed positive mass $\mu_c$ gives
\begin{equation}\label{eq:cellaveragesmall}
 \sup_{p<r}\|g_{p,c}\|_p\longrightarrow0.
\end{equation}
No individual shape probability or cut-dependent
mass is divided out.

For fixed $\tau$, the finite sum in \eqref{eq:energy}
now tends to zero. The upper limit of its left side
is at most $32C_k\tau$. Letting $\tau\downarrow0$
and using \eqref{eq:fixedlawenergy} proves
\eqref{eq:targetproved}.

Finally, the contraction of $\Pi_p$, the triangle inequality,
and Lemma \ref{lem:branch} give
\[
 \kappa_k\le
 8\sup_{p<r}\sum_T\mu_k^p(T)\|\Pi_pH_T^p\|_p^4
       +8k^4B_k\max_{m<k}\kappa_m
       \longrightarrow0.
\]
This finished the proof.
\end{proof}

\section{Proof of the M\'ezard–Parisi lower bound}\label{sec:bridge}

We now prove Theorems \ref{thm:variational} and \ref{thm:MP}. Continuity transfers the finite-spin approximation to the cavity functional. We then remove the global mixture, the root site coordinate and the perturbation, and combine the resulting variational bound with Theorem \ref{thm:cavitybound} and the Franz--Leone upper bound.

Recall the trial classes $\cT_{\rm root}$ and
$\cT_0$ from Section \ref{sec:whydefects}. The latter
is the trial class in the Franz--Leone bound as stated
in \cite{rRSB}.

\begin{proof}[Proof of Theorems \ref{thm:variational} and \ref{thm:MP}]
Fix $\epsilon>0$ and $G\in\cG_\epsilon$. We first
check continuity of \eqref{eq:PG} under the spin
approximation of Proposition \ref{prop:MPapprox}.
Truncate clause arities, counts and coupling magnitudes
in both terms of the functional. The arguments of the
logarithms then lie in a fixed compact subset of
$(0,\infty)$, uniformly over the spin law. Approximate
the logarithm there by a polynomial. Each power of
the Gibbs average is an average over finitely many
replicas. By \eqref{eq:extension}, the clause
Boltzmann factors are multiaffine in their incoming
magnetizations, and different incoming arguments use
distinct independent site coordinates. Their products
therefore expand into finitely many spin moments:
each site--replica pair occurs at most once in a
monomial. Proposition \ref{prop:MPapprox} gives
convergence of these moments, and hence of each
truncated functional.

The truncation error is uniform over the spin law:
deleting clauses changes either logarithm by at most
the sum of their absolute bounds. In particular, the
REM arity tail costs at most
$\sum_{d>D}(2d-1)\E J_d\to0$; integrability of the
same clause bounds controls the count and magnitude
tails. Taking the fine-cut limit at fixed truncation
and polynomial approximation, then removing both,
proves convergence of the full functional evaluated
on the resampled field \eqref{eq:MProot} to
$\cP_\epsilon(G)$.

\emph{1. Remove the global mixture.}
Conditionally on $U_*=u$, the resampled field is the
trial with function $\widehat h_u$. Since $U_*$ is
averaged outside the logarithms,
\[
 \cP_\epsilon(G)
 =\lim_{\delta,\eta\to0}
       \E_{U_*}\cP_\epsilon(r,\boldsymbol\zeta,\widehat h_{U_*})
 \ge\inf_{\cT_{\rm root}}\cP_\epsilon.
\]
This proves Theorem \ref{thm:variational}. No
uniformity over $G$ is needed, since the trial
infimum is the same for every $G$.

\emph{2. Remove the root site coordinate.}
Unlike $U_*$, the coordinates $Z_*^I$ are independent
across input sites, so fixing one common value would
change the trial. Instead, given a rooted trial
$(r,\boldsymbol\zeta,h_0)$, prepend a level with
parameter $0<\zeta'<\zeta_0$ and use
\[
 \widetilde h_0(x_0,x_1,\ldots,x_{r+1})
       =h_0(x_1,\ldots,x_{r+1}).
\]
This trial belongs to $\cT_0$: the old root slot is
now at depth one, and the new root is ignored.
To compare the functionals, condition on the clause
disorder in either expected logarithm. After the
original $r$ steps of the RPC recursion, let $Y_i$
be the remaining function of the root site bundle
$\omega=(Z_*^I)_I$, for term $i=1,2$. The rooted
term is $\E_\omega Y_i$; the extra level replaces it
by $\zeta'^{-1}\log\E_\omega e^{\zeta'Y_i}$, by the
same recursion \cite[Theorem 2.9]{SKbook}. If
$\Theta_i$ is the sum of the absolute clause bounds
in that term, including $|g|$ in the cavity term,
then $|Y_i|\le\Theta_i$. Jensen's inequality and
Hoeffding's lemma give
\begin{equation}\label{eq:rootcompare}
 0\le\frac1{\zeta'}\log\E_\omega e^{\zeta'Y_i}
             -\E_\omega Y_i
       \le\frac{\zeta'}2\Theta_i^2,
       \qquad i=1,2.
\end{equation}
The second moments of $\Theta_i$ are finite by
Section \ref{sec:integrability}. Averaging over the
disorder changes the whole functional by at most
$\frac{\zeta'}2\max_i\E\Theta_i^2$. Thus every rooted
trial is approximated in functional value by trials
in $\cT_0$. Together with
$\cT_0\subseteq\cT_{\rm root}$, this proves
\begin{equation}\label{eq:sameinf}
 \inf_{\cT_{\rm root}}\cP_\bullet
       =\inf_{\cT_0}\cP_\bullet,
       \qquad \cP_\bullet=\cP_\epsilon\text{ or }\cP_0.
\end{equation}

\emph{3. Remove the perturbation.}
Deleting its clauses from both logarithms gives,
uniformly over all trials,
\begin{equation}\label{eq:Premove}
 |\cP_\epsilon(r,\boldsymbol\zeta,h_0)
       -\cP_0(r,\boldsymbol\zeta,h_0)|
 \le\E|g|+\sum_{d\ge2}(2d-1)\E J_d
 \le C\sqrt\epsilon.
\end{equation}
Here $d$ and $d-1$ are the two Poisson means; the
last bound follows from \eqref{eq:tablemax}.

We can now prove Theorem \ref{thm:MP}. By Theorem \ref{thm:cavitybound}, there
is $G_\epsilon\in\cG_\epsilon$ with
$\liminf_NF_N^\epsilon\ge\cP_\epsilon(G_\epsilon)$.
Theorem \ref{thm:variational}, \eqref{eq:Premove}
and the finite-model estimate \eqref{eq:pressurepert}
therefore give
\[
 \liminf_{N\to\infty}F_N^{\rm model}
 \ge\cP_\epsilon(G_\epsilon)-C\sqrt\epsilon
 \ge\inf_{\cT_{\rm root}}\cP_0-2C\sqrt\epsilon.
\]
Letting $\epsilon\downarrow0$, after the fine-cut
limit at fixed $\epsilon$, proves \eqref{eq:lowerfinal}.
For the two classes of models in \eqref{eq:MPformula}, the
Franz--Leone bound \cite{FL,PT,Talagrand,rRSB} and
\eqref{eq:sameinf} give
\[
 F_N^{\rm model}\le\inf_{\cT_0}\cP_0
                 =\inf_{\cT_{\rm root}}\cP_0.
\]
The matching bounds prove the existence of the limit
and the formula \eqref{eq:MPformula}.
\end{proof}

\appendix
\section{Elementary facts about the covariance norm}\label{app:norm}

\begin{lemma}\label{lem:normfacts}
The quantity $\|f\|_p$ in \eqref{eq:Xnorm} is a norm on the bounded kernels modulo equality almost everywhere. It satisfies the triangle inequality, the contraction properties stated after \eqref{eq:Xnorm}, and
\begin{equation}\label{eq:mixedkernel}
 \E_{U_{\le p},Z,Z'}\left|\E_u f(u,Z)g(u,Z')\right|^2
       \le\|f\|_p^2\|g\|_p^2 .
\end{equation}

\end{lemma}
\begin{proof}
Fix the state prefix $U_{\le p}$ and suppress it from the notation. Instead of $U_T$ in \eqref{eq:Xnorm}, we will simply write $u.$ Let
\[
 K_f(u,u')=\int f(u,Z)f(u',Z)\,dZ,
\]
and abusing the notation, let us denote 
\begin{equation}
 \|f\|_p^4
 :=\iint
       \left|
         \int
           f(u,Z)
           f(u,Z')\,du
       \right|^2
       dZ\,dZ'.
\end{equation}
for the conditional norm without averaging over $U_{\le p}$.

Fubini and Cauchy--Schwarz give
\begin{align*}
 \iint\left|\int f(u,Z)g(u,Z')\,du\right|^2dZ\,dZ'
 &=\int K_f(u,u')K_g(u,u')\,du\,du'
 \\
& \le\|K_f\|_2\|K_g\|_2 =  \|f\|_p^2  \|g\|_p^2.
\end{align*}
Averaging and applying Cauchy--Schwarz in $U_{\le p}$ proves \eqref{eq:mixedkernel} for unconditional norms. 

Expand the kernel $\int(f+g)(u,Z)(f+g)(u,Z')\,du$ into its four terms; Minkowski's inequality and the mixed bound give
\[
 \|f+g\|_p^2\le\|f\|_p^2+2\|f\|_p\|g\|_p+\|g\|_p^2 ,
\]
which gives the triangle inequality at a fixed state prefix. Taking the $L^4$ norm over $U_{\le p}$ and applying Minkowski's inequality once more gives the triangle inequality for the original norm \eqref{eq:Xnorm}.

For the state-variable contractions, let $P$ be either
the centering operator $I-\E_u$ or conditional expectation
over a subcollection of the coordinates of $u$. The same
operator acts on $f(\cdot,Z)$ for every $Z$. Fubini gives
\[
 K_{Pf}(u,u')
 =P_uP_{u'}K_f(u,u'),
\]
where the two operators act separately on the two state
arguments. Conditional expectation and centering are
contractions in $L^2$, by Jensen's inequality and the
variance identity, respectively. Applying these contractions
successively gives
\[
 \int |K_{Pf}(u,u')|^2\,du\,du'
 \le
 \int |K_f(u,u')|^2\,du\,du'.
\]
By Fubini, these integrals are the conditional fourth powers
of the corresponding norms. Averaging over $U_{\le p}$
therefore proves the state-variable contractions.

For the site-variable contraction, write the full site tuple
as $(Z,Y)$. Since the two factors are averaged separately,
\[
 \int(\E_Y f)(u,Z)(\E_Y f)(u,Z')\,du
 =
 \iint
 \left(\int f(u,Z,Y)f(u,Z',Y')\,du\right)
 dY\,dY'.
\]
Jensen's inequality in the independent variables $Y,Y'$
bounds the square of this expression by
\[
 \iint
 \left|\int f(u,Z,Y)f(u,Z',Y')\,du\right|^2
 dY\,dY'.
\]
Integrating over $Z,Z'$ and averaging over $U_{\le p}$
proves $\|\E_Y f\|_p\le\|f\|_p$.

For \eqref{eq:normconversion}, Cauchy--Schwarz gives, with
the state prefix fixed,
\[
 \left|\int e(u,Z)e(u,Z')\,du\right|^2
 \le
 \left(\int |e(u,Z)|^2\,du\right)
 \left(\int |e(u,Z')|^2\,du\right).
\]
Integrating over $Z,Z'$ and then averaging over $U_{\le p}$
gives the first inequality in \eqref{eq:normconversion}.
Since $|e|\le M$, the conditional integral
$\int |e|^2\,du\,dZ$ is at most $M^2$, which gives the
second inequality.

To check nondegeneracy of the original norm \eqref{eq:Xnorm}, suppose that $\|f\|_p=0$.
Then, for almost every retained state prefix,
$K_f=0$ almost everywhere. Consequently, for every
$\varphi\in L^2(du)$,
\[
 \int\left|\int\varphi(u)f(u,Z)\,du\right|^2dZ
 =
 \iint
 \varphi(u)K_f(u,u')\varphi(u')\,du\,du'
 =0.
\]
Taking $\varphi$ in a countable dense subset of $L^2(du)$
shows that $f=0$ almost everywhere.
\end{proof}

\section{Shape probabilities}\label{app:shapes}

We prove Lemma \ref{lem:product} by successively removing
replicas. Starting from a shape $T$, apply the first-split
construction of Section \ref{sec:shapes} repeatedly to
its nonsingleton child shapes. Call a replica group
\emph{simple} if all its children are singletons. This
includes groups in which all replicas have the same
terminal label, by the convention in Section
\ref{sec:shapes}. Following nonsingleton children always
reaches a simple group, so a removal can begin whenever
at least two replicas remain.

\paragraph{Removing one replica.}
For the auxiliary cascade rooted at $p$, write
\[
 \zeta_{j-1}^{(p)}
 =
 \begin{cases}
  0,&j=p,\\
  \zeta_{j-1},&p<j\le r,
 \end{cases}
 \qquad \zeta_r=1.
\]
Suppose that a shape $T$ has $N\ge2$ replicas and contains
a simple group of $m\ge2$ replicas with split depth $d$. Let $T^-$ be the shape obtained by removing
one of these replicas. Then
\begin{equation}\label{eq:pruningfactor}
 \mu_N^p(T)
 =
 \frac{(m-1)\zeta_d-\zeta_{d-1}^{(p)}}{N-1}
 \,\mu_{N-1}^p(T^-).
\end{equation}
The numerator depends only on the simple group from
which a replica is removed; the denominator counts all the remaining
replicas in the shape.

Here is the short consequence of the GG identities
that proves \eqref{eq:pruningfactor}. Conditional on
the shape of $n$ replicas, consider a cluster at depth
$j$ containing $s$ of them. Apply the GG identities
with one of these $s$ replicas distinguished and
with the test that the new replica belongs to the
same depth-$j$ cluster. The conditional probability
of this event is
\[
 \frac{1-\zeta_{j-1}^{(p)}}{n}
 +\frac{s-1}{n}
 =
 \frac{s-\zeta_{j-1}^{(p)}}{n}.
\]
Indeed, the first numerator is the corresponding
two-replica probability, and precisely $s-1$ of the
other old replicas belong to this cluster.
For $n=1$, the same formula follows directly from
the two-replica law.

By exchangeability, the replica removed from $T$
may be taken to be the last one. Suppose first that
$d<r$. In $T^-$, the depth-$d$ cluster in question
contains $m-1$ replicas, lying in distinct clusters
at depth $d+1$. To recover $T$, the new replica must
join the depth-$d$ cluster and avoid these $m-1$
occupied depth-$(d+1)$ clusters. By ultrametricity,
this also specifies its overlaps with all replicas
outside the cluster. Its conditional probability is
therefore
\[
 \frac{m-1-\zeta_{d-1}^{(p)}}{N-1}
 -(m-1)\frac{1-\zeta_d}{N-1}
 =
 \frac{(m-1)\zeta_d-\zeta_{d-1}^{(p)}}{N-1}.
\]
If $d=r$, the new replica must join a terminal state
occupied by $m-1$ replicas. Its conditional probability
is
$
 \frac{m-1-\zeta_{r-1}}{N-1},
$
which is the same formula with $\zeta_r=1$.
This proves \eqref{eq:pruningfactor}. The identity
also applies when $m=2$, when the removal leaves a
singleton in place of that group.

\begin{proof}[Proof of Lemma \ref{lem:product}]
\emph{(a)}
Suppose that $a(T)>p$. Repeatedly remove one replica
from a simple group until only one replica remains.
Every split depth encountered is greater than
$p$, since deleting replicas cannot decrease their
first split. Thus
$\zeta_{d-1}^{(p)}=\zeta_{d-1}$ in every factor of
\eqref{eq:pruningfactor}.
The same sequence of removals has exactly the same
factors under $\mu^p$ and $\mu^0$, and both sequences
end with the one-replica probability $1$.
Consequently, $\mu_k^p(T)=\mu_k^0(T)$.

\medskip
\emph{(b)}
We use induction on the number $k$ of replicas.
Put $a=a(T)$ and let $T_1,\ldots,T_b$ be the child
shapes, with sizes $k_1,\ldots,k_b$.

If all children are singletons, the event specifying
$T$ requires any chosen pair of replicas to have
last-common-ancestor depth $a$. Hence
$ \mu_k^p(T)\le w_a^{(p)}.$
Since all child probabilities equal one, this is
\eqref{eq:rootdom}. This case includes $a=r$, with
the singleton-child convention for replicas having
one terminal label, and in particular starts the induction at
$k=2$.

Otherwise, choose a child $T_i$ with $k_i\ge2$.
Within this child, choose a simple group of $m\ge2$
replicas with split depth $d>a$, and remove one of
its replicas. Denote the resulting full shape
by $T^-$ and the modified child shape by $T_i^-$.
The first split remains $a$, and all other child
shapes are unchanged.

Set  $A=(m-1)\zeta_d-\zeta_{d-1}.$
Because $d>a\ge p$, the numerator in
\eqref{eq:pruningfactor} is the same whether the
removal is performed in the whole shape under
$\mu^p$ or in the child shape under $\mu^a$.
Thus
\[
 \mu_k^p(T)
   =\frac{A}{k-1}\mu_{k-1}^p(T^-),
 \qquad
 \mu_{k_i}^a(T_i)
   =\frac{A}{k_i-1}\mu_{k_i-1}^a(T_i^-).
\]
Applying the induction hypothesis to $T^-$ gives
\begin{align*}
 \mu_k^p(T)
 &\le
 \frac{A}{k-1}\,
 w_a^{(p)}\,
 \mu_{k_i-1}^a(T_i^-)
 \prod_{j\ne i}\mu_{k_j}^a(T_j)\\
 &=
 \frac{k_i-1}{k-1}\,
 w_a^{(p)}
 \prod_{j=1}^b\mu_{k_j}^a(T_j) \le
 w_a^{(p)}
 \prod_{j=1}^b\mu_{k_j}^a(T_j).
\end{align*}
This proves \eqref{eq:rootdom}. The comparison of
the two denominators is the only inequality needed
while pruning a nonsingleton child.
\end{proof}


\begin{thebibliography}{99}
\addcontentsline{toc}{section}{References}

\bibitem{ASS} M. Aizenman, R. Sims and S. L. Starr. Extended variational principle for the Sherrington-Kirkpatrick spin-glass model. \emph{Phys. Rev. B} \textbf{68} (2003), 214403.

\bibitem{Aldous} D. Aldous. Representations for partially exchangeable arrays of random variables. \emph{J. Multivariate Anal.} \textbf{11} (1981), 581--598.

\bibitem{Austin} T. Austin. Exchangeable random measures. \emph{Ann. Inst. H. Poincar\'e Probab. Statist.} \textbf{51} (2015), 842--861.

\bibitem{AP} T. Austin and D. Panchenko. A hierarchical version of the de Finetti and Aldous-Hoover representations. \emph{Probab. Theory Relat. Fields} \textbf{159} (2014), 809--823.

\bibitem{Bolthausen} E. Bolthausen and A.-S. Sznitman. On Ruelle's probability cascades and an abstract cavity method. \emph{Comm. Math. Phys.} \textbf{197} (1998), 247--276.

\bibitem{Coja2} A. Coja-Oghlan, K. Panagiotou. The asymptotic $K$-sat threshold. \emph{Adv. Math.} \textbf{288} (2016), 985--1068.

\bibitem{CojaPerkins} A. Coja-Oghlan and W. Perkins.
Spin systems on Bethe lattices.
\emph{Comm. Math. Phys.} \textbf{372} (2019), 441--523.


\bibitem{Concetti} F. Concetti.
The full replica symmetry breaking in the Ising spin glass on random
regular graph.
\emph{J. Statist. Phys.} \textbf{173} (2018), 1459--1483.


\bibitem{DerridaREM} B. Derrida. Random-energy model: an exactly solvable model of disordered systems. \emph{Phys. Rev. B} \textbf{24} (1981), 2613--2626.

\bibitem{DeSanctis} L. De Sanctis.
Random multi-overlap structures and cavity fields in diluted spin glasses.
\emph{J. Statist. Phys.} \textbf{117} (2004), 785--799.


\bibitem{DSS}
J. Ding, A. Sly and N. Sun.
Proof of the satisfiability conjecture for large $k$.
\emph{Ann. of Math. (2)} \textbf{196} (2022), no.~1, 1--388.

\bibitem{FLRZ} S. Franz, M. Leone, F. Ricci-Tersenghi and R. Zecchina.
Exact solutions for diluted spin glasses and optimization problems.
\emph{Phys. Rev. Lett.} \textbf{87} (2001), 127209.


\bibitem{FL} S. Franz and M. Leone. Replica bounds for optimization problems and diluted spin systems. \emph{J. Statist. Phys.} \textbf{111} (2003), 535--564.



\bibitem{GG} S. Ghirlanda and F. Guerra. General properties of overlap probability distributions in disordered spin systems. Towards Parisi ultrametricity. \emph{J. Phys. A} \textbf{31} (1998), 9149--9155.

\bibitem{Guerra} F. Guerra. Broken replica symmetry bounds in the mean field spin glass model. \emph{Comm. Math. Phys.} \textbf{233} (2003), 1--12.

\bibitem{Hoover} D. N. Hoover. Row-column exchangeability and a generalized model for probability. In \emph{Exchangeability in Probability and Statistics}, North-Holland (1982), 281--291.

\bibitem{Kallenberg} O. Kallenberg. On the representation theorem for exchangeable arrays. \emph{J. Multivariate Anal.} \textbf{30} (1989), 137--154.

\bibitem{MMZ}
S. Mertens, M. M\'ezard and R. Zecchina.
Threshold values of random $K$-sat from the cavity method.
\emph{Random Structures Algorithms} \textbf{28} (2006), 340--373.

\bibitem{MPV} M. M\'ezard, G. Parisi and M. A. Virasoro. \emph{Spin Glass Theory and Beyond}. World Scientific Lecture Notes in Physics, Vol. 9, World Scientific, Teaneck, NJ (1987).

\bibitem{MP1} M. M\'ezard and G. Parisi. The Bethe lattice spin glass revisited. \emph{Eur. Phys. J. B} \textbf{20} (2001), 217--233.

\bibitem{MPZ} M. M\'ezard, G. Parisi and R. Zecchina.
Analytic and algorithmic solution of random satisfiability problems.
\emph{Science} \textbf{297} (2002), 812--815.


\bibitem{MZcavity} M. M\'ezard and R. Zecchina.
Random $K$-satisfiability problem: From an analytic solution to an
efficient algorithm.
\emph{Phys. Rev. E} \textbf{66} (2002), 056126.


\bibitem{MPzero} M. M\'ezard and G. Parisi.
The cavity method at zero temperature.
\emph{J. Statist. Phys.} \textbf{111} (2003), 1--34.



\bibitem{OAI} OpenAI. The M\'ezard--Parisi formula
for diluted spin glasses.
\href{https://github.com/openai/math/blob/main/preprints/The-Mezard-Parisi-formula-for-diluted-spin-glasses-September-23-2026/paper.pdf}{OpenAI Math Release preprint}
(2026).

\bibitem{Spins} D. Panchenko. Spin glass models from the point of view of spin distributions. \emph{Ann. Probab.} \textbf{41} (2013), 1315--1361.

\bibitem{Ultra} D. Panchenko. The Parisi ultrametricity conjecture. \emph{Ann. of Math.} \textbf{177} (2013), 383--393.

\bibitem{PPF} D. Panchenko. The Parisi formula for mixed $p$-spin models. \emph{Ann. Probab.} \textbf{42} (2014), 946--958.

\bibitem{SKbook} D. Panchenko. \emph{The Sherrington-Kirkpatrick Model}. Springer Monographs in Mathematics, Springer, New York (2013).

\bibitem{1RSB} D. Panchenko. Structure of $1$-RSB asymptotic Gibbs measures in the diluted $p$-spin models. \emph{J. Statist. Phys.} \textbf{155} (2014), 1--22.

\bibitem{HEPS} D. Panchenko. Hierarchical exchangeability of pure states in mean field spin glass models. \emph{Probab. Theory Relat. Fields} \textbf{161} (2015), 619--650. 

\bibitem{rRSB} D. Panchenko. Structure of finite-RSB asymptotic Gibbs measures in the diluted spin glass models. \emph{J. Statist. Phys.} \textbf{162} (2016), 1--42. 

\bibitem{GGcavInput} D. Panchenko. Ghirlanda--Guerra identities and cavity equations for diluted spin glass models. Preprint (2026).

\bibitem{PT} D. Panchenko and M. Talagrand. Bounds for diluted mean-fields spin glass models. \emph{Probab. Theory Relat. Fields} \textbf{130} (2004), 319--336.

\bibitem{Parisi79} G. Parisi. Infinite number of order parameters for spin-glasses. \emph{Phys. Rev. Lett.} \textbf{43} (1979), 1754--1756.

\bibitem{Parisi80} G. Parisi. A sequence of approximated solutions to the S-K model for spin glasses. \emph{J. Phys. A} \textbf{13} (1980), L115--L121.

\bibitem{ParisiBethe} G. Parisi.
The marginally stable Bethe lattice spin glass revisited.
\emph{J. Statist. Phys.} \textbf{167} (2017), 515--542.

\bibitem{Ruelle} D. Ruelle. A mathematical reformulation of Derrida's REM and GREM. \emph{Comm. Math. Phys.} \textbf{108} (1987), 225--239.

\bibitem{SK} D. Sherrington and S. Kirkpatrick. Solvable model of a spin-glass. \emph{Phys. Rev. Lett.} \textbf{35} (1975), 1792--1796.

\bibitem{Talagrand} M. Talagrand. \emph{Mean Field Models for Spin Glasses}, Vols. I and II. Springer (2011).

\bibitem{TPF} M. Talagrand. The Parisi formula. \emph{Ann. of Math.} \textbf{163} (2006), 221--263.

\end{thebibliography}
\end{document}